\documentclass[11pt,a4paper]{amsart}
\usepackage{amssymb,amsfonts,amsmath,mathrsfs}
\usepackage[left=2cm,right=2cm,top=2cm,height=250mm]{geometry}

\usepackage[overload]{empheq}

\usepackage{color}
\usepackage[colorlinks=true,linkcolor=blue,citecolor=red, pdfborder={0 0 0}]{hyperref}
\usepackage{enumerate}
\numberwithin{equation}{section}
\newtheorem{theorem}{Theorem}[section]
\newtheorem{lem}{Lemma}[section]
\newtheorem{Def}{Definition}[section]
\newtheorem{rem}{Remark}[section]
\newtheorem{prop}{Proposition}[section]
\newtheorem{cor}{Corollary}[section]

\def\u{u^\varepsilon}

\def\r{r^\varepsilon}
\def\q{q^\varepsilon}

\def\d{\partial}
\def\R{\Bbb R}

\def\N{\Bbb N}

\def\Nx{\nabla}
\def\Dt{\partial_t}
\def\({\left(}
\def\){\right)}
\def\eb{\varepsilon}

\def\Cal{\mathcal}

\def\L{\Lambda}

\def\E{\mathcal{E}}
\def\A{\mathcal{A}}
\def\B{\mathcal{B}}

\def\di{\operatorname{div}}

\DeclareMathOperator{\dist}{dist}

\begin{document}
\title[]{Homogenisation with error estimates of attractors for
damped semi-linear anisotropic  wave equations}
\author[] {Shane Cooper, Anton Savostianov}

%

\subjclass[2000]{35B40, 35B45, 35L70}

\keywords{damped wave equation, global attractor, exponential attractor, homogenisation, homogenization, error estimates}
\address{Durham University, Department of Mathematical Sciences, Durham, DH1 3LE, United Kingdom.}
\email{shane.a.cooper@durham.ac.uk, anton.savostianov@durham.ac.uk}

\maketitle

\begin{abstract}
Homogenisation of global $\A^\eb$ and exponential $\Cal M^\eb$ attractors for the damped semi-linear anisotropic wave equation $\Dt^2\u +\gamma\Dt \u-\di \left(a\left( \tfrac{x}{\eb} \right)\Nx \u \right)+f(\u)=g$, on a bounded domain $\Omega \subset \R^3$, is performed. Order-sharp estimates between trajectories $u^\eb(t)$ and their homogenised trajectories $u^0(t)$ are established. These estimates are given in terms of the operator-norm  difference between resolvents of the elliptic operator $\di\left(a\left( \tfrac{x}{\eb} \right)\Nx \right)$ and its homogenised limit $\di\left(a^h\Nx \right)$. Consequently, norm-resolvent estimates on the Hausdorff distance between the anisotropic attractors and their homogenised counter-parts $\A^0$ and $\Cal M^0$ are established.  These results imply error estimates of the form $\dist_X(\A^\eb, \A^0) \le C \eb^\varkappa$ and $\dist^s_X(\Cal M^\eb, \Cal M^0) \le C \eb^\varkappa$ in the spaces $X =L^2(\Omega)\times H^{-1}(\Omega)$ and $X =(C^\beta(\overline{\Omega}))^2$. In the natural energy space $\E : = H^1_0(\Omega) \times L^2(\Omega)$, error estimates $\dist_\E(\A^\eb, \textsc{T}_\eb \A^0) \le C \sqrt{\eb}^\varkappa$ and $\dist^s_\E(\Cal M^\eb, \textsc{T}_\eb \Cal M^0) \le C \sqrt{\eb}^\varkappa$  are established where $\textsc{T}_\eb$ is first-order correction for the homogenised attractors suggested by asymptotic expansions. Our results are applied to Dirchlet, Neumann and periodic boundary conditions.
\end{abstract}

\section*{Introduction}
\label{s.i}
In this article we consider the following damped  semi-linear wave equation in a bounded smooth domain $\Omega\subset\R^3$ with rapidly oscillating coefficients:
\begin{equation}\label{eq.dw}
\begin{cases}
\Dt^2\u +\gamma\Dt \u-\di \left(a\left( \tfrac{x}{\eb} \right)\Nx \u \right)+f(\u)=g(x),\quad x\in\Omega,\ t\geq 0,\\
(\u,\Dt \u)|_{t=0}=\xi, \qquad
\u|_{\partial\Omega}=0.
\end{cases}
\end{equation} 
Such equations appear, for example, in the context of non-linear ascoustic oscillations in periodic composite media (see for example \cite{SaPa80}). 

For fixed $\eb>0$, the long-time behaviour of $u^\eb$ has been intensively studied in many works under various assumptions on the non-linearity $f$ and force $g$. In the context of dissipative PDEs the long-time dynamics can be studied in terms of global attractors. Intuitively speaking, the global attractor is a compact subset of the infinite-dimensional phase space which attracts all trajectories that originate from bounded regions of phase space. Therefore, the global attractor is in some sense a `much smaller' subset of phase space that characterises the long-time dynamics of the system  (see for example \cite{BV,CV,Ha88,MZ.hbk2008,SeYo02,Te}).

It is well-known that for suitable assumptions on the non-linearity (cf. \cite{BV,Te}) that problem \eqref{eq.dw} possesses a global attractor $\A^\eb$ and an important question to ask, from the point of view of applications, is about the asymptotic structure, with respect to $\eb$, of the global attractor $\A^\eb$ in the limit of small $\eb$. Asymptotics for global attractors have been studied, in  the context of reaction diffusion equations and the damped wave equation, with respect to  `lower-order' rapid spatial oscillations in the dampening, non-linearity and/or forces $g$ (see \cite{CCPan18}, \cite{CCPat}, \cite{FV01}, \cite{FV02}). Yet surprisingly, to the best knowledge of the authors, little or no work has been performed on the asymptotics of attractors for hyperbolic dissipative systems with `higher-order' rapid spatial oscillations such as in \eqref{eq.dw}. We mention here the works \cite{EZ02ho} that perform a quantitative analyse of the asymptotics of global attractors in the context of reaction diffusion equations. We also mention the works \cite{PaChu1,PaChu2} that determine the limit-behaviour of global attractors, in the context of reaction-diffusion and hyperbolic equations, for a particular choice of rapidly oscillating coefficients that degenerate in the limit of small period. Aside from the very limited amount of work done on the asymptotics of global attractors for dissipative PDEs with rapid oscillations, no work has been done on the asymptotics of exponential attractors. This article is dedicated to performing these studies for problems of the form \eqref{eq.dw}.

\vspace{10pt} In this article we aim to study the long-time behaviour of trajectories $u^\eb$ to \eqref{eq.dw}, for small parameter $\eb$, from the point of view of homogenisation theory.  In homogenisation theory,  the mapping 
\[
A_\eb u : = - \di \left(a(\tfrac{\cdot}{\eb})\Nx u \right),
\]
 for periodic uniformly elliptic and bounded coefficients $a(\cdot)$, is well-known to converge (in an appropriate sense)  in the limit of small $\eb$ to 
\[
A_0 u : = - \di \big(a^h\Nx u \big),
\]
where $a^h$ is the `effective' or `homogenised' constant-coefficient matrix associated to  $a(\cdot)$ (see for example \cite{JKO94} and references therein). As such, it is  natural to compare the long-time dynamics of $\u$ to the long-time dynamics of $u^0$ the solution to homogenised problem
\begin{equation}\label{eq.dwh}
\begin{cases}
\Dt^2u^0 +\gamma\Dt u^0-\di \left(a^h\Nx u^0 \right)+f(u^0)=g(x),\quad x\in\Omega,\ t\geq 0,\\
(u^0,\Dt u^0)|_{t=0}=\xi,
\qquad u^0|_{\partial\Omega}=0.
\end{cases}
\end{equation}

 Homogenisation theory has been studied intensively since the 1970's and amongst the extensive works we focus on works related to quantitative estimates of the form
	\begin{equation}\label{int.e1}
	 \| A^{-1}_\eb - A^{-1}_0 \|_{\Cal L(L^2(\Omega))} \le C \eb,
	\end{equation}
where the mappings have been equipped with appropriate boundary conditions. Such (sharp) order-$\eb$ results, that are now standard, has been proved by various authors using various techniques (see the monograph \cite{ZhPas16} for a review of some of these techniques). We mention here the results of particular interest to our article; in the case of bounded domain with Dirichlet or Neumann boundary conditions the order-sharp estimates were proved for the first time  in \cite{Su17bd,Su17Ne} and utilised the (order-sharp) estimate proved in \cite{BiSu04} for the whole space (and periodic torus).

While some work has been done to provide order-sharp operator estimates for individual trajectories in the parabolic (cf. \cite{JKO94,ZhPas16}) or hyperbolic settings (for smooth enough initial data) (cf. \cite{BSu08,DSu16, Su17, Me18}),  no work is done on providing order-sharp operator estimates for attractors in dissipative PDEs. 

Our first main result is the following estimate\footnote{Here $\dist_X(A,B)$ denotes the one-sided Hausdorff metric between sets $A$ and $B$ in the strong topology of $X$.} between the global attractors $\A^\eb$ and $\A^0$, associated to problem \eqref{eq.dw} and \eqref{eq.dwh} respectively, in the energy spaces $\E^{-1} : =L^2(\Omega) \times H^{-1}(\Omega)$ and $(C^\beta(\overline{\Omega}))^2$ (see Theorem \ref{th.d(Ae,Ah)<e1/2.w} and Corollary \ref{co.d(Ae,A0)}):
\begin{equation}
\label{int.e2}
\begin{aligned}
& \dist_{\E^{-1}} \big( \A^\eb , \A^0 \big) \le C  \| A^{-1}_\eb - A^{-1}_0 \|_{\Cal L(L^2(\Omega))}^\varkappa, \\
& \dist_{(C^\beta(\overline{\Omega}))^2} \big( \A^\eb , \A^0 \big) \le C  \| A^{-1}_\eb - A^{-1}_0 \|_{\Cal L(L^2(\Omega))}^{\theta\varkappa},
\end{aligned}
\end{equation}
for some $\varkappa, \theta \in(0,1)$. Upon combining this result with the operator estimate \eqref{int.e1} gives the desired error estimates between global attractors.  

The above inequality is new in the homogenisation theory of attractors. Moreover, this result is important from the general perspective as it establishes the upper semi-continuity of global attractors of the damped wave equation in terms of the elliptic part of the PDE. Indeed, in the proof of this result  we do not use the asymptotic structure in $\eb$ of $S_\eb(t)$ in terms of $S_0(t)$.  The arguments are purely operator-theoretic in nature and only require that the elliptic operator is self-adjoint and boundedly invertible (see Section \ref{s.Er}). In particular, if  $A_\eb$ and $A_0$ were positive elliptic operators  $A = \di (a \nabla)$ and $B = \di(b\nabla)$ for two different matrices $a$ and $b$, the above continuity result still holds. Additionally, the same can be said for different boundary conditions: can replace Dirichlet boundary conditions with other types of boundary  conditions under the sole requirement that $A = \di (a \nabla)$ defines a self-adjoint  operator in $L^2(\Omega)$ (see Section \ref{s.difbc} for details).

Let us say a few words on the method of proof of \eqref{int.e2}.
This result is essentially proved by establishing the following (sharp) estimate between trajectories $u^\eb(t)$ and $u^0(t)$ for initial data in $\A^\eb$ (Theorem \ref{th.|ue-uh|E-1}):
\begin{equation}
\label{int.e1.1}\|  u^\eb(t) -  u^0(t) \|_{L^2(\Omega)} + \| \Dt u^\eb(t) - \Dt u^0(t) \|_{H^{-1}(\Omega)} \le M e^{Kt} \| A^{-1}_\eb - A^{-1}_0 \|_{\Cal L(L^2(\Omega))}, \quad t\ge 0.
\end{equation}	 Then, to prove \eqref{int.e2}, we combine this novel estimate with the exponential attraction property of $\A^0$ which is known to hold `generically' on an open dense subset of forces $g$:
	\begin{equation*}
\left\{ \ \begin{aligned}
&\textit{ $\exists\, \sigma>0$ such that for every bounded set $B \subset \E$ the following estimate holds:} \\[2pt]
&\dist_{\E}(S_0(t)B,\A^0)\leq M(\|B\|_\E)e^{-\sigma t},\quad t\geq 0.
\end{aligned} \right.
\end{equation*}
 Notice that estimate \eqref{int.e1.1}  is optimal; indeed, upon substituting the right-hand side with $\eb$ we arrive at the expected order-sharp estimates in $\eb$ (just as in the elliptic case \eqref{int.e1}).

Aside from \eqref{int.e2}, a natural question to ask is if we can compare the global attractors in the energy space $\E : =H^1_0(\Omega) \times L^2(\Omega)$. In general estimates of the form \eqref{int.e2} are not to be expected in $\E$ and this is due to the fact that, on the level of asymptotic expansions, the trajectories $\Nx u^\eb(t)$ are not close to $\Nx u^0(t)$ but instead are close to 
\[
\Cal T_\eb u^0(t,x) : = u^0(t,x) + \eb \sum_{i=1}^3 N_i \big( \tfrac{x}{\eb} \big) \partial_{x_i} u^0(t,x).
\]
Here $N_i$ are the solutions to the so-called auxiliary cell problem (see Section \ref{s.pre}). Indeed, in Homogenisation theory it is known that \eqref{int.e1} does not generally hold in $H^1(\Omega)$ but rather the following `corrector' estimate
\[
\begin{aligned}
\|  A^{-1}_\eb g - \Cal T_\eb A^{-1}_0 g \|_{H^1(\Omega)} \le C \sqrt{\eb} \| g \|_{L^2(\Omega)},
\end{aligned}
\]
holds (cf. the above citations on error estimates in homogenisation of elliptic systems). 
For this reason, we introduce the notion of correction to attractors: 
\[
\textsc{T}_\eb \xi : = (\Cal T_\eb \xi^1, \xi^2), \qquad \xi=(\xi^1,\xi^2) \in  \A^0,
\]
and our next main result is the following corrector estimate (Theorem \ref{th.gEest}):
\begin{equation}
\label{int.e3}\dist_{\E } \big( \A^\eb , \textsc{T}_\eb \A^0 \big) \le C  \sqrt{\eb}^\varkappa.
\end{equation}
To the best of our knowledge, in all previous works, no corrector estimates  were provided in the homogenisation of attractors. To prove this result we naturally aim to establish an inequality of the form:
\begin{equation}
\label{int.e3.1}
\| u^\eb(t) - \Cal T_\eb u^0(t) \|_{H^1_0(\Omega)} \le M e^{Kt} \sqrt{\eb},  \quad t\ge 0,
\end{equation}
for initial data $\xi \in \A^\eb$. It turns out that for such initial data the trajectory $u^0(t)$ does not contain enough regularity for such a result to hold. This issue is due to the hyperbolic nature of the problem and does not appear, for example, in the context  of parabolic equations. To overcome this issue we introduce specially prepared initial data $\xi_0$ for the trajectory $u^0$ as follows: $\xi_0^1 \in H^1_0(\Omega)$ is the solution to
\[
\di ( a^h \Nx \xi_0^1 ) = \di (a (\tfrac{\cdot}{\eb}) \Nx \xi^1) \quad \text{ in $\Omega$}.
\]
Then, for such a choice of initial data, we readily establish inequality \eqref{int.e3.1} (Theorem \ref{th.Dterror} and Corollary \eqref{co.|ue-uh|E})  and consequently prove \eqref{int.e3}.

An important question from the point of view of applications is whether or not the estimates \eqref{int.e2}, \eqref{int.e3} hold in the \emph{symmetric} Hausdorff distance 
\[
\dist^s ( \A^\eb , \A^0 ) = \max \big\{ \dist (\A^\eb, \A^0) , \dist( \A^0, \A^\eb) \big\}.
\]
To prove this one would need to show that for sufficiently small $\eb$ the global attractor $\A^\eb$ is in fact (generically) an exponential attractor with exponent, and set of generic forces, independent of $\eb$. Such a result seems reasonable from the perspective of considering $\A^\eb$ to be an `appropriate' perturbation of the global attractor $\A^0$ and applying the theory of regular attractors, see for example \cite{BV,VZC2013}. Such a result has yet to be established and we intend to carry out this study in future work. 

That being said, it is known that, in general,  global attractors are not continuous (in the \emph{symmetric} Hausdorff distance) under perturbations and that the rate of attraction  can be arbitrarily slow. For this reason the theory of exponential attractors was developed; such exponential attractors are known to be stable under perturbations and attract bounded sets \emph{exponentially} fast in time. Importantly,  exponential attractors also occupy `small' subsets of  phase space in the sense that they have finite fractal dimension, cf. \cite{EFNT, EMZ2000, EMZ2005, FG2004}. 

Motivated by the above discussion, and the desire for estimates in the symmetric Hausdorff distance, we also study the relationship between exponential attractors associated to problems \eqref{eq.dw} and  \eqref{eq.dwh}. In fact we construct exponential attractors $\Cal M^\eb$ and $\Cal M^0$ whose (finite) fractal dimension and exponents of attraction are independent of $\eb$, and we determine the following analogues of \eqref{int.e2} and \eqref{int.e3} in the symmetric distance (Theorem \ref{th.exAe} and Corollary \ref{co.expat}):
\begin{equation}
\label{int.e4}
\left\{ \begin{aligned}
&\dist^s_{\E^{-1}} \big( \Cal M^\eb , \Cal M^0 \big)   \le C  \| A^{-1}_\eb - A^{-1}_0 \|_{\Cal L(L^2(\Omega))}^\varkappa, \\
& \dist^s_{(C^\beta(\overline\Omega))^2} \big( \Cal M^\eb , \Cal M^0 \big) \le C  \| A^{-1}_\eb - A^{-1}_0 \|_{\Cal L(L^2(\Omega))}^{\theta\varkappa}, \\
&\dist^s_{\E} \big( \Cal M^\eb , \textsc{T}_\eb\Cal M^0 \big)   \le C  \sqrt{\eb}^\varkappa.
\end{aligned} \right.
\end{equation}
To establish the last inequality above we developed further (in Theorem \ref{th.expAttr.E}) the known abstract construction of  exponential attractors of semi-groups to include the case of  semi-groups that admit asymptotic expansions (i.e. `corrections' such as $\Cal T_\eb$).

\vspace{10pt}

{
We end the introduction with some words on the structure of this article.  In Section \ref{s.pre}, we formulate precise assumptions on the non-linearity $f$ and the elliptic part of  \eqref{eq.dw}, \eqref{eq.dwh}. Also, we recall relevant known well-posedness results as well as results on the existence of global attractors associated with \eqref{eq.dw}, \eqref{eq.dwh}. For the reader's convenience, details on the corresponding attractor theory is provided in Appendix \ref{s.sma}. In Section \ref{s.smth}, for the dynamical systems generated by problems \eqref{eq.dw}, \eqref{eq.dwh}, we establish  existence  and smoothness results for  an attracting set (which contains the global attractors). These results will be crucial in justifying  error estimates between anisotropic and homogenised attractors. In Section \ref{s.conv}, we establish the convergence, in the limit of $\eb \rightarrow 0$,  of the anisotropic global attractor $\Cal A^\eb$ to the homogenised attractor $\Cal A^0$ in the spaces $\E^{-1}$ and $(C^\beta(\overline{\Omega}))^2$. In Section \ref{s.Er}, we derive the central (order-sharp) estimate \eqref{int.e1.1} on the difference between trajectories $u^\eb(t)$ and $u^0(t)$ of the corresponding anisotropic and homogenised problems. Then, based on this, we demonstrate the quantitative  estimates \eqref{int.e2} on the distance between global attractors $\Cal A^\eb$ and $\Cal A^0$. Estimate \eqref{int.e3} between the global attractor $\A^\eb$ and first-order correction $\textsc{T}_\eb\Cal A^0$ in the energy space $\E$ is proved in Section \ref{s.gaE}. Section \ref{s.ea} is devoted to exponential attractors $\Cal M^\eb$, $\Cal M^0$ associated with problems \eqref{eq.dw}, \eqref{eq.dwh} and consists of two parts. In Subsection \ref{s.ExE-1}, existence of the exponential attractors is proved and estimates \eqref{int.e4} in $\E^{-1}$ and $(C^\beta(\overline{\Omega}))^2$ are obtained. The results in this section rely on a variant of a standard abstract result on the construction of exponential attractors; this construction is included in Appendix \ref{app.E-1}. In Subsection \ref{s.dE(Me,M0)}, we compare the distance between the exponential attractor $\Cal M^\eb$ and the first-order correction $\textsc{T}_\eb\Cal M^0$ in the energy space $\E$. Subsection \ref{s.dE(Me,M0)} rests on a new abstract theorem, presented in Appendix \ref{app:estinE}, which compares the distance between exponential attractors which admit correction. We discuss, and prove the corresponding results for the cases of Neumann and periodic boundary conditions in Section \ref{s.difbc}. Some refinements of the results obtained in Sections \ref{s.smth}-\ref{s.Er} related to boundary corrections in homogenisation theory  are the subject of Appendix \ref{app.improvement}.             
}

\section*{ Notations}
We document here notations frequently used throughout the article. The $L^2(\Omega)$ inner product is given by $(u,v):=\int_\Omega u(x)v(x)\,dx$, with norm denoted by $\|u\|:=(u,u)^{1/2}$ for $u, v \in L^2(\Omega)$. We frequently consider initial data in the energy spaces $\E^{-1}:=L^2(\Omega)\times H^{-1}(\Omega)$, and $\E:=H^1_0(\Omega)\times L^2(\Omega).$ These spaces are equipped with norms whose squares are given as $\|\xi\|^2_{\E^{-1}}:=\|\xi^1\|^2+\|\xi^2\|^2_{H^{-1}(\Omega)}$ and $\|\xi\|^2_\E:=\|\Nx \xi^1\|^2+\|\xi^2\|^2$ for admissible pairs\footnote{Here we adopt the common clash of notation for $(\cdot,\cdot)$ to mean both an inner product and represent a pair in a product space. It will be clear from the context which meaning is appropriate.} $\xi=(\xi^1,\xi^2)$. For any function $z(t)$ we set $\xi_z(t)$ to be the pair $(z(t),\Dt z(t))$  where $\Dt z$ denotes the distributional (time) derivative. For a Banach space $E$, $B_E(0,r)$ denotes the ball centered at $0$ of radius $r$ in $E$; the symbol $[\ \cdot\ ]_E$ denotes the closure in  $E$; the one-sided and symmetric Hausdorff distances between two sets $A, B \subset E$ are respectively defined as  $\dist_{E}(A,B):=\sup_{a\in A}\inf_{b\in B}\|a-b\|_{E}$ and	$\dist^s_E (A, B): = \max \big\{ \dist_E (A,B) , \dist_E(B,A) \big\} $. The standard Euclidean basis is denoted by $\{e_k\}_{k=1}^3$. 

\section{Preliminaries}\label{s.pre}
Throughout the article, unless stated otherwise, we adopt the convention that $M$ and $K$ denote generic constants whose precise value may vary from line to line.

For a given matrix $a(\cdot)=\{a_{ij}(\cdot)\}_{i,j=1}^3$ we denote by $a^h=\{a^h_{ij}\}_{i,j=1}^3$  the  homogenised matrix corresponding to $a(\cdot)$ whose constant coefficients are given by the formula
\begin{equation*}
a^h_{ij}:=\int_Q\Big(a_{ij}(y)+\sum_{k=1}^3a_{ik}(y)\partial_{y_k}N_j(y)\Big)\,dy.
\end{equation*}
Here 
$N_i$, $i\in \{1,2,3\}$, is the solution to the so-called cell problem:
\begin{equation}
\label{Ni}
\begin{cases}
-\di_y\big(a(y) \nabla_yN_i(y) \big)=\di_y\big(a(y)e_i\big),\quad \ y\in Q=[0,1)^3,\\[2pt]
\int_QN_i(y)\,dy=0, \qquad \ N_i(\cdot+e_j) = N_i (\cdot) \qquad   j \in\{1,2,3\}.
\end{cases}
\end{equation}

It is well-known that if $a(\cdot)$ is symmetric, bounded and uniformly elliptic, then so is $a^h$ with the  \emph{exact}  same bounds  (see for example \cite[Section 1]{JKO94}). Furthermore,  as  $a^h$ is constant it is clearly periodic.  Consequently, both problem \eqref{eq.dw} and \eqref{eq.dwh} are problems of the form 
\begin{equation}\label{eq.generic}
\begin{cases}
\Dt^2 u +\gamma\Dt u-\di \left( a \Nx u \right)+f(u)=g(x),\quad x\in\Omega,\ t\geq 0,\\
(u,\Dt u)|_{t=0}=\xi, \qquad u|_{\partial\Omega}=0,
\end{cases}
\end{equation} 
with the same generic assumptions on coefficients, forces and non-linearity; we collect these assumptions together here:
\begin{equation}\label{mainassumptions}\tag{H1}
\left\{ \quad \begin{aligned}
&\textit{Let $\Omega\subset\R^3$ be a bounded smooth domain, $g\in L^2(\Omega)$, $a(\cdot)= \{a_{ij}(\cdot)\}_{i,j=1}^3$ satisfying }\\
& a_{ij}\in L^\infty(\R^3),\hspace{1cm} a_{ij}=a_{ji},\hspace{1cm} a_{ij}(\cdot +e_k)=a_{ij}(\cdot),\quad \ i,j,k\in\{1,2,3\},\\ 
& \& \quad 	\nu |\eta|^2 \leq a(y)\eta.\eta\leq \nu^{-1}|\eta|^2,\ \nu>0,\qquad \forall y\in\R^3,\ \forall \eta \in \R^3; \\
&\textit{and $f\in C^2(\R)$   satisfying  } \\
&f(s)s\geq-K_1, \hspace{1cm}
  f'(s)\geq -K_2,\hspace{1cm}
   |f''(s)|\leq K_3(1+|s|), \hspace{1cm}
    f(0)=0, \quad  s \in \R,
\end{aligned} \right.
\end{equation}
for some positive constants $\nu, K_i$.

\begin{rem}
\label{rem.nonlinearity} 
\begin{enumerate}[{a.}] \noindent We note that above assumptions on $f$ imply the following bounds which are important in obtaining dissipative estimates.
	\item{There exists $K_4 > 0$ and $K_5>0$ such that $	\hspace{10pt}	| f'(s) | \le K_4 (1+ |s|^2), \quad | f(s) | \le K_5 (1+ |s|^3), \quad  s \in \mathbb{R}.$
	}\label{fcubic}
\item{\vspace{-5pt}
		The anti-derivative $ F(s)=\int_0^sf(\tau)\,d\tau$ satisfies $		-\frac{K_2}{2}s^2\leq F(s)\leq f(s)s+\frac{K_2}{2}s^2, \quad s \in \R.$
}
\item{ \vspace{5pt}	For all $\mu>0$ there exists $K_\mu>0$ such that
$\hspace{10pt}
		F(s) \ge - K_\mu-\mu s^2, \qquad s\in \mathbb{R}.
$
	}
\end{enumerate}
	Also note that the assumption $f(0)=0$ is, in fact, not a restriction since $f(0)$ always can be included into the forcing term $g$. 
\end{rem}
We begin with some basic existence, continuity and dissipative estimate results. Particular attention is paid to the dependence of these results on the matrix $a$, assuming that the other variables ($\Omega$ and  $f$) are fixed. As these results are standard we shall omit the proofs, commenting here that they are easily argued by the techniques employed in Appendix \ref{s.sma}.
\begin{theorem}
\label{th.dis}
Assume \eqref{mainassumptions}.  Then, for any initial data $\xi\in\E$, problem \eqref{eq.dw} possesses a unique energy solution $u$ with $\xi_{u}\in C(\R_+;\E)$. Moreover, the following dissipative estimate is valid:
\begin{equation}
\label{est.dis}
\|\xi_{u}(t)\|^2_\E+\int_t^\infty\|\Dt u(\tau)\|^2\,d\tau\leq M(\|\xi\|_\E)e^{-\beta t}+M(\|g\|),\quad t\geq 0,
\end{equation}
for some non-decreasing function $M$ and constant $\beta>0$ that depend only on $\nu$.
\end{theorem}
A consequence of the dissipative estimate \eqref{est.dis}, growth restrictions on $f$, and uniform ellipticity of $a(\cdot)$ we have the following continuous dependence on initial data.
\begin{cor}\label{co.St.inC}
Let $u_1$ and $u_2$ be two energy solutions to problem \eqref{eq.generic} with initial data $\xi_{1}$, $\xi_{2}\in\E$ respectively. Then the following estimate
\begin{equation*}
\|\xi_{u_1}(t)-\xi_{u_2}(t)\|_\E\leq M e^{Kt} \|\xi_{1}-\xi_{2}\|_\E,\quad t\geq 0,
\end{equation*}
holds for some constant $M>0$ and $K=K(\| \xi_1 \|_{\E}, \| \xi_2 \|_{\E}, \|g \|, \nu)$. 
\end{cor}
Additionally, we have the following continuous dependence in $\E^{-1}$.
\begin{cor}\label{cor.contE-1}
	Let $u_1$ and $u_2$ be two energy solutions to problem \eqref{eq.generic} with initial data $\xi_{1}$, $\xi_{2}\in\E$ respectively. Then the following estimate
	\begin{equation*}
	\|\xi_{u_1}(t)-\xi_{u_2}(t)\|_{\E^{-1}}\leq M e^{Kt} \|\xi_{1}-\xi_{2}\|_{\E^{-1}},\quad t\geq 0,
	\end{equation*}
holds for some constant $M>0$ and $K=K(\| \xi_1 \|_{\E}, \| \xi_2 \|_{\E}, \|g \|, \nu)$.
\end{cor}
We now proceed to study the long-time behaviour of solutions $u$ from the point of view of infinite-dimensional dynamical systems. In particular the problem \eqref{eq.generic} defines a dynamical system $(\E,S(t))$ by
\begin{equation}
\label{ds.e}
S(t):\E\to\E,\quad S(t)\xi=\xi_{u}(t),
\end{equation}
where $u$ is a solution to the problem \eqref{eq.generic} with initial data $\xi$. The limit behaviour of a dissipative dynamical system as time goes to $+\infty$ can be described in terms of a so-called global attractor. Let us briefly recall its definition (see \cite{BV,CV,MZ.hbk2008,Te}).

\begin{Def}\label{de.A} Let $S(t):\E\to\E$ be a semi-group acting on a Banach space $\E$. Then a set $\A$ is called a global attractor for the dynamical system $(\E,S(t))$ if it possesses the following properties:
\par
1. The set $\A$ is compact in $\E$;
\par
2. The set $\A$ is strictly invariant:
\[
S(t) \A = \A, \quad \forall t \ge0;
\]
\par
3. The set $\A$ uniformly attracts every bounded set $B$ of $\E$, that is 
\begin{equation*}
\lim_{t\to+\infty}\dist_{\E}\big(S(t)B,\Cal A\big)=0.
\end{equation*}
\end{Def}

One can show that if a global attractor exists then it is unique. Also, the following description of the global attractor in terms of bounded trajectories is known (see e. g. \cite{BV,CV}):
\begin{equation}
\label{bddedsolns}
\Cal A=\{\xi_0\in\E: \exists\, \xi(t)\in L^\infty(\R;\E),\ \xi(0)=\xi_0,\ S(t)\xi(s)=\xi(t+s),\ s\in\R,\ t\geq 0 \}.
\end{equation}

Now, the dissipative estimate \eqref{est.dis} implies the existence of a bounded positively invariant absorbing set $\B\subset\E$ (which depends only on $\nu$):
	\begin{equation}
	\label{StB C B}
	S(t)\B\subset\B,\quad \forall t\geq0.
	\end{equation}
To prove that a global attractor exists for problem \eqref{eq.generic} we utilise the following classical result  (\cite{BV,CV,MZ.hbk2008,Te}).
\begin{theorem}\label{th.Ag}
A dynamical system $(\E,S(t))$ possesses a global attractor $\Cal A$ in $\E$ if the following conditions hold:
\par
1. The dynamical system $(\E,S(t))$ is asymptotically compact: there exists a compact set $\mathscr{K}\subset \E$ such that
\[
\lim_{t\to+\infty}\dist_{\E}(S(t)B,\mathscr{K})=0,\qquad for\ all\ bounded\ sets\ B\subset\E;
\]
\par
2. For each $t\geq 0$ the operators $S(t):\E\to\E$ are continuous.

\hspace{10pt}Under such conditions, it follows that $\A$ not only exists but also  $\Cal A \subset \mathscr{K}$.
\end{theorem}
Note that Corollary \ref{co.St.inC} implies that the evolution operator $S(t)$, given by \eqref{ds.e}, has continuous dependence on the initial data. Let us focus on the existence of a compact attracting set. 

Introducing the space
\begin{equation}
\label{E1space}
\left\{
\begin{aligned}
& \E^1 : = \{ \xi = (\xi^1, \xi^2)\in \E \, | \, \di (a \Nx \xi^1 ) \in L^2(\Omega),\, \xi^2 \in H^1_0(\Omega)\}, \\ 
& \Vert \xi \Vert_{\E^1}^2 : =  \Vert \di (a \Nx \xi^1 ) \Vert^2 + \| \nabla \xi^2 \|^2,
\end{aligned} \right.
\end{equation}
we have the following known result that states there exists an attracting ball in $\E^1$.
\begin{theorem}\label{cor.smooth1}
	Assume \eqref{mainassumptions}, and let $S$ be the semi-group defined by \eqref{ds.e}. Then, there exists a ball in $\E^1$ that  attracts the set $\Cal B$, from \eqref{StB C B}, in $\E$. More precisely,  the inequality
	\[
	\dist_{\E} \big( S(t) \Cal B , B_{\E^1}(0,R) \big) \le M e^{- \beta t}, \quad t\ge 0,
	\]
	holds	for some positive constants $R$, $M$ and $\beta$ that depend only on $\nu$.
\end{theorem}
The proof of Theorem \ref{cor.smooth1} is presented for the reader's convenience in Appendix \ref{s.sma} and is based on a splitting  of trajectory $u$, into the smooth and contractive parts, that was developed in \cite{PataZel}.

Consequently, as $\E^1$ is compact in $\E$ we see from Theorem \ref{cor.smooth1} that $\mathscr K = B_{\E^1}(0,R)$ is a compact attracting set and, by Theorem \ref{th.Ag}, there exists a global attractor. That is, the following result holds.
\begin{theorem}
	\label{th.ExAttrinE}Assume \eqref{mainassumptions}. Then, the dynamical system $(\E,S(t))$ given by \eqref{ds.e} possesses a global attractor $\Cal A \subset \Cal E^1$ such that:
	\begin{equation}
	\label{Ae=Ke(0)}
	\|\A \|_{\Cal E^1 }\leq M(\|g\|),\qquad \Cal A=\Cal K|_{t=0},
	\end{equation}
	where $\Cal K$ is the set of bounded energy solutions to problem \eqref{eq.generic} defined \emph{for all} $t\in\R$, cf. \eqref{bddedsolns}. 
\end{theorem}
\section{Smoothness of the global attractor}
\label{s.smth}
Above we demonstrated that the global attractor $\A$ is a bounded subset of $\E^1$. We shall now establish some additional regularity of $\A$. These results will be used  later on  to  derive homogenisation error estimates.

We are going to show that $\A$ is contained in the more regular set
\begin{equation*}
\left\{\begin{aligned}
& \Cal E^2 : = \big\{ \xi \in \E^1 \, | \, \big(\di(a\Nx \xi^1) + g \big) \in H^1_0(\Omega) \text{ and } \di (a \Nx \xi^2 ) \in L^2(\Omega)  \big\}, \\
& \Vert \xi \Vert_{\Cal E^2}^2 : =  \Vert \di (a \Nx \xi^1 ) + g \Vert^2_{H^1_0(\Omega)} +\|\di(a\Nx \xi^1)\|^2 + \Vert \di (a \Nx \xi^2 ) \Vert^2,
\end{aligned} \right.
\end{equation*}
and that $\A$ is bounded in the following sense: $\| \A \|_{\Cal E^2} \le M$.

To this end,  we shall show that $ B_{\E^1}(0,R)$ is exponentially attracted, in $\E$, to some `ball' \footnote{Note that the convex functional $\Vert \cdot \Vert_{\Cal E^2}$ is \emph{not} a norm and the set $\Cal E^2$ is an affine  subset of $\E^1$.}
\[
B_{\Cal E^2}(0,R_1) : = \{ \xi \in \Cal E^2 \, | \, \Vert \xi \Vert_{\Cal E^2} \le R_1 \}.
\]
Then by 
 utilising the so-called transitivity property of exponential attraction we establish that $\Cal B$ (from \eqref{StB C B}) is attracted to $ B_{\Cal E^2}(0,R_1) $ exponentially in $\E$ and, therefore,  we will show that $\A$ is bounded in  $\Cal E^2$.

Let us begin with the following theorem which provides a useful dissipative estimate for problem \eqref{eq.dwh} with initial data in $\E^1$ (see \eqref{E1space}). 
\begin{theorem}
	\label{th.disE2}
Assume \eqref{mainassumptions}. Then for any initial data $\xi\in \E^1$ the energy solution $u$ to problem \eqref{eq.generic} is such that $\xi_{u}\in L^\infty(\R_+;\E^1)$, and the following dissipative estimate is valid:
	\begin{equation*}
	\|\Dt^2u(t)\|+\|\xi_{u}(t)\|_{\E^1}\leq M(\|\xi\|_{\E^1})e^{-\beta t}+M(\|g\|),\ t\geq 0,
	\end{equation*}
	for some non-decreasing function $M$ and constant $\beta>0$ that depend only on $\nu$.
\end{theorem}
Since this result is standard we omit the proof. We only remark here that, by differentiating the first equation of \eqref{eq.generic} in time, one first obtains a dissipative estimate for $\|\xi_{\Dt u}(t)\|_\E$ which readily implies the uniform bound on $\|  \di (a \Nx  u)(t)\|$.

\begin{rem}
\label{rem.holder}
Note that by elliptic regularity we have the inequality
\begin{equation}
\label{we.Ca}
\|u \|_{C^\alpha(\overline{\Omega})}\leq C \Vert \di (a \Nx u) \Vert, \quad C=C(\nu)>0,
\end{equation}
for sufficiently small $\alpha = \alpha(\nu)$ and admissible $u$. Here $C^\alpha(\overline{\Omega})$ is the H\"older space of order $\alpha$: 
\begin{align*}
& C^\alpha(\overline{\Omega})=\bigg\{u\in C(\overline{\Omega}):\sup_{\substack{x,y\in\overline{\Omega},\\ 
		x\neq y}}\tfrac{|u(x)-u(y)|}{|x-y|^\alpha}<\infty\bigg\}, \qquad  \|u\|_{C^\alpha(\overline{\Omega})}:=\max_{x\in\overline{\Omega}}|u(x)|+\sup_{\substack{x,y\in\overline{\Omega},\\ 
		x\neq y}}\tfrac{|u(x)-u(y)|}{|x-y|^\alpha}.
\end{align*}
Thus, we have a dissipative estimate for $u$, given by Theorem \ref{th.disE2}, in the $C^\alpha(\overline{\Omega})$ norm.
\end{rem}
Consider $G \in H^1_0(\Omega)$ such that $
-\di (a \Nx G)=g \in L^2(\Omega),$
and, for initial data $\xi \in  B_{\E^1}(0,R)$, the decomposition of  the solution $u$ to \eqref{eq.generic} as follows:  $u=v+w$ where
\begin{equation}
\label{eq.ve2}
\begin{cases}
\Dt^2 v+\gamma\Dt v-\di (a \Nx v)=0,\quad x\in\Omega,\ t\geq 0,\\
\xi_{v}|_{t=0}=(\xi^1 - G, \xi^2),\qquad v|_{\partial\Omega}=0,\\
\end{cases}
\end{equation}
and
\begin{equation}
\label{eq.we2}
\begin{cases}
\Dt^2 w+\gamma\Dt w-\di(a \Nx w)=-f(u) + g,\quad x\in\Omega,\ t\geq 0,\\
\xi_{w}|_{t=0}=(G,0),\qquad w|_{\d\Omega}=0.\\
\end{cases}
\end{equation}
It is clear from  standard linear  estimates (e.g. Theorem \ref{th.dis} for $f = g = 0$) we
	\begin{equation}\label{le.ve2}
	\Vert  \xi_v \Vert_{\E} \leq e^{-\beta t}M(\|g\|),\quad t\geq 0,
	\end{equation} 
	for some constant $\beta>0$ and non-decreasing function $M$ that depend only on $\nu$.
Additionally, we have the following lemma on the regularity of $w$.
\begin{lem}\label{le.we2}
	Assume \eqref{mainassumptions}, $\xi \in B_{\E^1}(0,R)$ and $w$ solves \eqref{eq.we2}. Then
	\begin{equation*}
\| \di(a \Nx w)(t) +g \|_{H^1_0(\Omega)} +	\| \di(a\Nx  \Dt w)(t)\|\leq M(\|g\|),\quad  t\geq0,
	\end{equation*}
	for some non-decreasing function $M$ that depends only on $\nu$.
\end{lem}
\begin{proof} By differentiating the first equation of \eqref{eq.we2} in time and by our choice of initial data $(G,0)$ we find that $p:=\Dt w$ solves
	\begin{equation}
	\label{eq.qe2}
	\begin{cases}
	\Dt^2 p+\gamma\Dt p-\di(a \Nx p)=-f'(u)\Dt u = : G_1,\quad  x\in\Omega,\ t\geq 0,\\
	\xi_{p}|_{t=0}=\big(0,-f(\xi^1)\big),\qquad p|_{\d\Omega}=0.
	\end{cases}
	\end{equation}
Moreover,  $q := \Dt p$ solves
	\begin{equation*}
	\begin{cases}
	\Dt^2 q+\gamma\Dt q-\di(a \Nx q)= -f''(u)|\Dt u|^2-f'(u)\Dt^2 u = : G_2,\quad x\in\Omega,\ t\geq 0,\\
	\xi_{q}|_{t=0}=\big(-f(\xi^1),\gamma f(\xi^1)-f'(\xi^1) \xi^2\big),\qquad q|_{\d\Omega}=0.
	\end{cases}
	\end{equation*} 
By the dissipative estimate in $\E^1$ (cf. Theorem \ref{th.disE2} and Remark \ref{rem.holder}) we find that 
	\[
	\|\Nx \Dt u(t)\| +
 \| u(t) \|_{C^\alpha(\overline{\Omega})}\leq M(\|g\|), \quad t\ge0.
	\]
This inequality and the conditions on the non-linearity $f$ (see \eqref{mainassumptions}) imply that
	\begin{align*}
&\|\xi_{p}(0)\|_\E +  \|G_1\|_{L^\infty(\R_+;L^2(\Omega))}\leq M(\|g\|);\\
& \|\xi_{q}(0)\|_\E + \|G_2\|_{L^\infty(\R_+;L^2(\Omega))}\leq M(\|g\|).  
	\end{align*}
	Therefore, using the dissipative estimate in $\E$ (\eqref{est.dis}) we conclude 
	\begin{align*}
	\|\Nx p(t)\|+\|\Dt p(t)\|\leq M(\|g\|),\qquad \& \qquad 
	\|\Nx q(t)\|+\|\Dt q(t)\|\leq M(\|g\|),\quad t\geq 0.
	\end{align*}
	Returning back to $p=\Dt w$, we rewrite \eqref{eq.qe2} to find
	\begin{equation*}
\Vert 	\di(a \Nx \Dt w)(t) \Vert =  \Vert -G_1(t)	+\gamma\Dt p(t)+ \Dt q(t) \Vert \le M(\Vert g\Vert), \quad t\ge0.
	\end{equation*}
	Rewriting the first equation in \eqref{eq.we2}, and using cubic growth of $f$ (see Remark \ref{rem.nonlinearity}.\ref{fcubic}) gives
	\[
\Vert 	\di(a\Nx w)(t) + g \Vert_{H^1_0(\Omega)} = \Vert  q(t) + \gamma p(t) +f(u(t)) \Vert_{H^1_0(\Omega)} \le M(\Vert g \Vert), \quad t\ge0. 
	\]
Hence, the desired result holds and the proof is complete.
\end{proof}
Combining \eqref{le.ve2} and Lemma  \ref{le.we2} produces the following result.
\begin{cor}
	\label{cor.smooth2}
	Assume \eqref{mainassumptions} and let $S(t)$ be the semi-group defined by \eqref{ds.e}. Then, there exists a `ball' in $\Cal E^2$ that  attracts $B_{\E^1(0,R)}$ in $\E$. More precisely, the inequality
	\[
	\dist_{\E} \big( S(t) B_{\E^1}(0,R) , B_{\Cal E^2}(0,R_1) \big) \le M e^{- \beta t}, \quad t\ge 0,
	\]
	holds	for some positive constants $R_1$, $M$ and $\beta$ that depend only on $\nu$.
\end{cor}
Let us now recall the so-called {transitivity property of exponential attraction} (cf. \cite[Theorem 5.1]{FG2004} for a proof):
\begin{theorem}
\label{thm.trans}
Let $E$ be a Banach space, $S(t)$ a semi-group acting on $E$, and $E_1$ be a positively invariant subset of $E$, i.e. $S(t) E_1 \subset E_1$ for all $t\ge 0$, such that
\[
\Vert S(t) \xi^1  -S(t) \xi^2  \Vert_{E}\le M_0 e^{K_0t}  \Vert \xi^1 - \xi^2 \Vert_{E}, \qquad  \xi^1, \xi^2 \in E_1,
\]
for some constants $M_0$, $K_0 > 0$. Furthermore, assume that there exist subsets $E_2 \subset E_1$ and  $E_3 \subset E$ such that 
\[
\dist_{E} \big( S(t) E_1 , E_2 \big) \le M_1 e^{-\beta_1 t}, \quad \dist_{E} \big( S(t) E_2 , E_3 \big) \le M_2 e^{-\beta_2 t}, \qquad t\ge 0,
\]
for some $M_1, M_2$, $\beta_1>0$ and $\beta_2>0$. Then 
\[
\dist_{E} \big( S(t) E_1 , E_3 \big) \le M e^{-\beta t}, \qquad t\ge 0,
\]
for $M = M_0 M_1 +M_2$ and $\beta = \tfrac{\beta_1 \beta_2}{ K_0 + \beta_1 + \beta_2}$.
\end{theorem}
Note that Theorem \ref{th.dis} (in particular \eqref{StB C B}), Theorem \ref{cor.smooth1} and Corollary \ref{cor.smooth2} imply that the assumptions of the above theorem  hold for $E = \E$, $E_1 = \Cal B, E_2 = B_{\E^1}(0,R)$ and $E_3 = B_{\Cal E^2}(0,R_1)$. Therefore, we see  that $B_{\Cal E^2}(0,R_1)$ attracts the positively invariant absorbing set $\B$ and, therefore,  bounded sets in $\E$. That is the following result holds.
\begin{theorem}
\label{th.smoothattracting}
Assume \eqref{mainassumptions}, $S(t)$ given by \eqref{ds.e} and $ B_{\Cal E^2}(0,R_1)$ given by Corollary \ref{cor.smooth2}. Then, for every bounded $B$ in $\E$  the following assertion
\[
\dist_{\E} \big( S(t) B,  B_{\Cal E^2}(0,R_1) \big) \leq M(\|B\|_\E)e^{-\beta t},\quad t\geq 0,
\]
holds for some non-decreasing $M$ and $\beta>0$ that depend only on $\nu$. \end{theorem}

We are now ready to prove that  the global attractor is bounded in $\Cal E^2$.

\begin{theorem}\label{th.ExAttr}
Assume \eqref{mainassumptions} and let $\A$ be the global attractor of  the dynamical system $(\E,S(t))$ given by \eqref{ds.e}. Then
\begin{equation}
\label{Ae=Ke(0)}
\|\A \|_{\Cal E^2 }\leq M(\|g\|),
\end{equation}
for some non-decreasing $M$ that depends only on $\nu$.
\end{theorem}
\begin{rem}
\label{rem.holder2}
Note that \eqref{Ae=Ke(0)} implies  the following estimate
\begin{equation}
\label{Ae.sm1}
\| \A \|_{(C^\alpha(\overline{\Omega}))^2}\leq M(\|g\|),
\end{equation}
for a non-decreasing function $M$ that depends only on $\nu$ and the exponent $\alpha$ from Remark \ref{rem.holder}.
\end{rem}
\begin{proof}[Proof of Thoerem \ref{th.ExAttr}] The proof follows from  the strict invariance of the global attractor (property 2. of Definition \ref{de.A}) and Theorem \ref{th.smoothattracting}. Indeed, 	for an arbitrary $\delta$-neighbourhood $\Cal O_\delta(B_{\Cal E^2}(0,R_1))$ of $B_{\Cal E^2}(0,R_1)$ in $\E$, one has
	\[
	\A = S(t) \A \subset \Cal O_\delta(B_{\Cal E^2}(0,R_1)), 
	\]
for some $t=t(\delta)$. Therefore $\Cal A\subset [B_{\Cal E^2}(0,R_1)]_\E$ and it remains to note that, since $B_{\Cal E^2}(0,R)$ is closed in $\E$, the identity $[B_{\Cal E^2}(0,R_1)]_\E=B_{\Cal E^2}(0,R_1)$ holds.
%
\end{proof}
	We end this section with one more result which will be useful later.
	\begin{theorem}
		\label{th.disD}
		Assume \eqref{mainassumptions}. Then, for any initial data $\xi\in\Cal E^2$, the energy solution $u$ to problem \eqref{eq.generic} is such that $\xi_u\in L^\infty(\R_+;\Cal E^2)$ and the following dissipative estimate is valid:
		\begin{equation}
		\label{est,disD}
		\|\Dt^3 u(t)\|+\|\nabla\Dt^2 u(t) \|+\|\xi_u(t)\|_{\Cal E^2}\leq M(\|\xi\|_{\Cal E^2})e^{-\beta t}+M(\|g\|),\quad t\geq 0,
		\end{equation} 
		for some non-decreasing function $M$ and constant $\beta>0$ that depend only on $\nu>0$. 
	\end{theorem}
	
	The proof is very close to the proof of Lemma \ref{le.we2} and for this reason is omitted. We only remark that, since $\Cal E^2\subset \E^1$ and the dissipative estimate in $\E^1$ is already known, we see that the quantity $\|u(t)\|_{L^\infty(\Omega)}$ is bounded. Thus, basically, one applies linear dissipative estimates to the equations for $p$ and $q$ in the proof of Lemma \ref{le.we2} with the appropriately changed initial data.    
\section{Homogenisation and convergence of global attractors}\label{s.conv}
Let us now consider the dynamical systems $S_\eb(t)$ and $S_0(t)$ generated by problems \eqref{eq.dw} and \eqref{eq.dwh} respectively.  In Theorem \ref{th.ExAttr} we established that $S_\eb$ (respect. $S_0$) has a global attractor $\Cal A^\eb$ (respect.  $\Cal A^0$). Moreover, Theorem \ref{th.ExAttr} informs us that  $\Cal A^\eb$ is a, uniformly in $\eb$, bounded subset of $ \Cal E^2_\eb$  and $\Cal A^0$ is a bounded subset of $\Cal E^2_0$, where
\begin{equation}
\label{Dep}
\left\{ \begin{aligned}
&\Cal E^2_\eb : = \big\{ \xi  \in (H^1_0(\Omega))^2 \, | \, \big( \di (a(\tfrac{\cdot}{\eb}) \Nx \xi^1 ) + g\big) \in H^1_0(\Omega),\ \di (a(\tfrac{\cdot}{\eb}) \Nx \xi^2 ) \in L^2(\Omega)  \big\},\\
& \Vert \xi \Vert_{\Cal E^2_\eb}^2 : =  \Vert \di (a(\tfrac{\cdot}{\eb}) \Nx \xi^1 ) + g \Vert^2_{H^1_0(\Omega)} +\|\di(a(\tfrac{\cdot}{\eb})\Nx \xi^1)\|^2 + \Vert \di (a(\tfrac{\cdot}{\eb})\Nx \xi^2 ) \Vert^2,
\end{aligned} \right.
\end{equation}

and
\begin{equation}
\label{D0}
\left\{ \begin{aligned}
&\Cal E^2_0 : = \big\{ \xi  \in (H^1_0(\Omega))^2 \, | \, \big( \di (a^h \Nx \xi^1 ) + g\big) \in H^1_0(\Omega),\ \di (a^h \Nx \xi^2 ) \in L^2(\Omega)  \big\},\\
& \Vert \xi \Vert_{\Cal E^2_0}^2 : =  \Vert \di (a^h \Nx \xi^1 ) + g \Vert^2_{H^1_0(\Omega)} +\|\di(a^h\Nx \xi^1)\|^2 + \Vert \di (a^h\Nx \xi^2 ) \Vert^2 .
\end{aligned} \right.
\end{equation}
\begin{rem}
\label{rem.holder3}
 We note that, by elliptic regularity (see Remark \ref{rem.holder}), the global attractors $\A^\eb$ are \emph{uniformly} in $\eb$ bounded subsets of $\Cal E^2_\eb \cap (C^\alpha(\overline{\Omega}))^2$. Additionally  for $\A^0$, as $a^h$ is constant, we can readily deduce that $\A^0$ is a bounded subset of  $\Cal E^2_0 \cap (H^2(\Omega))^2$. That is, the inequalities
\[
\Vert \A^\eb \Vert_{\Cal E^2_\eb} + \Vert \A^\eb \Vert_{(C^\alpha(\overline{\Omega}))^2} \le M(\Vert g \Vert), \qquad \& \qquad \Vert \A^0 \Vert_{\Cal E^2_0} + \Vert \A^0 \Vert_{(H^2({\Omega}))^2} \le M(\Vert g \Vert), 
\]
hold for some non-decreasing function $M$ independent of $\eb$.
\end{rem}
The main result of this section is  following theorem which establishes convergence of the global attractors $\A^\eb$ to the global attractor $\A^0$ in the one-sided Hausdorff distance.
\begin{theorem}
\label{th.AetoAh1}
The global attractor $\A^{\eb}$ of the problem \eqref{eq.dw} converges to the global attractor $\A^0$ of the homogenised problem \eqref{eq.dwh} in the following sense
\begin{equation*}
\lim_{\eb \rightarrow 0}\dist_{(C^\beta(\overline{\Omega}))^2}(\A^{\eb},\A^0)=0,
\end{equation*} 
for any $0\leq \beta <\alpha$ where $\alpha$ is given in Remark \ref{rem.holder3}.
\end{theorem}
To prove Theorem \ref{th.AetoAh1} we shall use the following classical homogenisation theorem for elliptic PDEs (see for example \cite[Section 1]{JKO94}).
\begin{theorem}(Homogenisation theorem)\label{homthm}
Let $\Omega\subset\R^3$ be a bounded smooth domain, $a(\cdot)$ a positive bounded periodic matrix and $\eb_n \rightarrow 0$ as $n\rightarrow \infty$. Then for any sequence $g_n \in H^{-1}(\Omega)$ that strongly converges to $g$ in $H^{-1}(\Omega)$ we have that $u_n \in H^1_0(\Omega)$ the weak solution of
\[
\di(a(\tfrac{x}{\eb_n})\Nx u_n ) = g_n,
\]
weakly converges in $H^1_0(\Omega)$ to $u_0$ the weak solution of 
\[
\di (a^h\Nx u_0 ) = g.
\]
\end{theorem}
\begin{rem}\label{rem:convE}
In general, one can not expect \emph{strong} convergence of $u_n$ to $u_0$ in $H^1_0(\Omega)$ since this would imply that the homogenised matrix $a^h$ is simply the average $\int_Qa(y)\,dy$. Clearly this formula for the homogenised matrix is, in general, not true and it is known that the equality $a^h=\int_Qa(y)\,dy$ holds if, and only if, $\di_ya=0$ in weak sense.

A consequence of the above observation is that, in general, we can not expect convergence of the attractors $\A^{\eb}$ to $\A^0$ in the strong topology of $\E$. To obtain such convergence results a correction to $\A^0$ needs to be made, see Section \ref{s.gaE} for further information.
\end{rem}
\begin{proof}[Proof of Theorem \ref{th.AetoAh1}]

Fix an arbitrary sequence $\eb_n\rightarrow 0$ and $\xi_n \in \A^{\eb_n}$. To prove the result it is sufficient to show that there exists   $\xi_0 \in \A^0$ such that  $\xi_n$ converges, up to some subsequence, to $\xi_0$ in $(C^\beta(\overline{\Omega}))^2$  as $n \rightarrow \infty$.

For each $n \in \mathbb{N}$, we denote by $u_n \in \mathcal{K}^{\eb_n}$ the bounded (for all time) in $\E$ solution of \eqref{eq.dw}  that satisfies $\xi_{u_n}(0) = \xi_n$. Now, 
$\A^{\eb}$ is a (uniformly in $\eb$) bounded subset of $\big(H^1_0(\Omega) \cap C^\alpha(\overline{\Omega})\big)^2$ (see Remark \ref{rem.holder3}). Moreover, for any fixed $0 \leq  \beta < \alpha$, it is well-known that $C^\beta(\overline{\Omega})$ is compactly embedded in  $C^\alpha(\overline{\Omega})$. Therefore, up to some discarded subsequence,
\begin{equation}
\label{homprfe0}
\text{ $\xi_n$ converges strongly in $(C^\beta(\overline{\Omega}))^2$ to some $\xi_0 \in \big( H^1_0(\Omega) \cap C^\beta(\overline{\Omega}) \big)^2$.}
\end{equation}  

It remains to prove that $\xi_0 \in \A^0$, and this is established if we demonstrate that $\xi_0 = \xi_{u_0}(0)$ for some bounded (for all time) in $\E$  solution $u_0$ to \eqref{eq.dwh}. The remainder of the proof is to establish the existence of such a $u_0$. In what follows convergence is meant up to an appropriately discarded subsequence.

By Remark \ref{rem.holder3} and the strict invariance  of $\A^\eb$ (property 2 of Definition \eqref{de.A})  there exists $M>0$ such that 
\begin{equation}
\label{hompfe1}
\begin{aligned}
&\|\Nx u_n(t)\| + \|\di(a(\tfrac{x}{\eb_n})\Nx u_n)(t) \| + \|u_n(t)\|_{C^\alpha(\overline{\Omega})}\\
& \hspace{.2\textwidth} +\|\Nx \Dt u_n(t)\|+\|\di(a(\tfrac{x}{\eb_n})\Nx \Dt u_n)(t) \|+\|\Dt u_n(t) \|_{C^\alpha(\overline{\Omega})} \le M, 
\end{aligned}
\end{equation}
for all $ n \in \mathbb{N}$ and all $ t \in \mathbb{R}$. 

Let us fix $z \in \mathbb{Z}$. Using \eqref{hompfe1} we find
\[
u_n\mbox{ is bounded in }W_1 : = \{ w \in L^\infty\big( [z,z+2] ; H^1_0(\Omega) \big) \, | \,  \Dt w \in L^\infty\big( [z,z+2] ; L^2(\Omega) \big) \}.
\]
Similarly, since (cf. \eqref{eq.dw})
\begin{equation}
\label{hompfe4}
\Dt^2 u_n = - \gamma \Dt u_n + \di(a(\tfrac{x}{\eb_n})\Nx u_n) - f(u_n) + g,
\end{equation}
assertion \eqref{hompfe1} and the cubic growth condition of $f$ (Remark \ref{rem.nonlinearity}\eqref{fcubic}) imply that 
\[\Dt u_n \mbox{ is bounded in }W_1.\]

Furthermore, differentiating \eqref{hompfe4} in $t$ gives
\begin{equation*}
\Dt^3 u_n = - \gamma \Dt^2 u_n + \di(a(\tfrac{x}{\eb_n})\Nx \Dt u_n) - f'(u_n)\Dt u_n.
\end{equation*}
This equation, along with  \eqref{hompfe1}, the boundedness of $\Dt u_n$ in $W_1$ and growth assumption on $f$ imply that
\[
\Dt^2 u_n\mbox{ is bounded in }\{ w \in L^\infty\big( [z,z+2] ; L_2(\Omega) \big) \, | \,  \Dt w \in L^\infty\big( [z,z+2] ; H^{-1}(\Omega) \big) \}.
\] 
Therefore, since the embeddings $H^1_0(\Omega)\subset L^2(\Omega)$ and $L^2(\Omega)\subset H^{-1}(\Omega)$ are compact, by Aubin-Lions lemma we deduce that

\begin{equation}\label{hompfe3.2}
\begin{aligned}
u_n & \longrightarrow u  & \  & \text{strongly in $C\big( [z, z+ 2] ;  L^2(\Omega)\big)$ as $n \longrightarrow \infty$;}\\
\Dt u_n & \longrightarrow \Dt u & &\text{strongly in $C\big( [z, z+ 2] ;  L^2(\Omega)\big)$ as $n \rightarrow \infty$;}\\
\Dt^2  u_n &\longrightarrow \Dt^2 u& &\text{strongly in $C\big( [z, z+ 2] ;  H^{-1}(\Omega)\big)$ as $n \rightarrow \infty$.}
\end{aligned}
\end{equation}
 

Let us demonstrate that $u$ solves \eqref{eq.dwh} on the time interval $[z,z+2]$. To this end we are going to pass to the limit in
\begin{equation}\label{hompfe3.4}
-\di\big(a(\tfrac{x}{\eb_n})\Nx u_n\big)=-\Dt^2 u_n-\gamma \Dt u_n-f(u_n)+g=:h_n.
\end{equation}
Due to \eqref{hompfe3.2} we know that 
\begin{equation*}
h_n(t)\longrightarrow-\Dt^2 u(t)-\gamma \Dt u(t)-f(u(t))+g\mbox{ strongly in }H^{-1}(\Omega) \mbox{ for all } t\in[z,z+2].
\end{equation*}
Therefore, by an application of the homogenisation theorem (Theorem \ref{homthm}), we conclude, that for every $t\in[z,z+2]$, $u_n(t)$ weakly converges in $H^1_0(\Omega)$ to the solution $u_0(t)$ of the homogenised problem
\begin{equation*}
-\di(a^h\Nx u_0(t))=-\Dt^2 u(t)-\gamma \Dt u(t)-f(u(t))+g.
\end{equation*} 
It follows from \eqref{hompfe3.2} and the weak convergence $u_n(t) \rightharpoonup u_0(t)$ in $H^1_0(\Omega)$ that $u(t)=u_0(t)$ for all $t\in[z,z+2]$. Consequently, 
from this identity and the above equation, we see that $u_0$ (weakly) solves
 \[
 \Dt^2 u_0 + \gamma \Dt u_0 -\di\big(a^h\Nx u_0 \big) + f(u_0)=g, \quad  t \in [z,z+2].
 \]
Let us argue that the above equation holds for all time. Indeed, by a Cantor diagonalisation argument we see that the convergences \eqref{hompfe3.2} can be taken to hold for all $z \in \mathbb{Z}$. Then, by noting that any $\phi \in C^\infty_0(\mathbb{R};C^\infty_0(\Omega))$ can be represented as  a finite sum of smooth functions whose individual supports (w.r.t to time) are in some $[z,z+2]$, we deduce that $u_0$ weakly solves the homogenised equation \eqref{eq.dwh}.
  Hence, $u_0$ is a bounded in $\E$ solution to \eqref{eq.dwh} for all time.

It remains to show that $\xi_{u_0}(0) =( u_0(0) , \Dt u_0(0))$ equals  $\xi_0$. On the one hand,  from \eqref{homprfe0} we see that $\xi_n$ converges strongly to $\xi_0$ in $(L^2(\Omega))^2$. On the other hand, by  \eqref{hompfe3.2} (for $z=0$) $\xi_n = (u_n(0), \Dt u_n(0))$ converges strongly to $( u_0(0) , \Dt u_0(0)) $ in $(L^2(\Omega))^2$. Hence, $( u_0(0) , \Dt u_0(0)) = \xi_0$ and the proof is complete.
 \end{proof}

\section{Rate of convergence to the homogenised global attractor}
\label{s.Er}
We shall begin with recalling an important result on error estimates in homogenisation theory  of elliptic PDEs. 
Recall, for fixed $\eb >0$, the mappings
\begin{equation}
\label{op.A}
\begin{aligned}
A_\eb u : =-\di(a(\tfrac{\cdot}{\eb})\Nx u), \quad \& \quad A_0u:=-\di(a^h\Nx u).
\end{aligned}
\end{equation}

\begin{theorem}[Theorem 3.1, \cite{ZhPas16}]
\label{th.Er.ell}
Let $\Omega\subset\R^3$ be a bounded smooth domain, symmetric periodic matrix $a(\cdot)$
satisfying uniform ellipticity and boundedness assumptions, $A_\eb$ and $A_0$ given by \eqref{op.A} and $g\in L^2(\Omega)$. Let also $u^\eb,u^0\in H^1_0(\Omega)$ solve the problems
\begin{equation*}
\begin{cases}
A_\eb u^\eb=g,\quad \text{ in $\Omega$},\\
u^\eb|_{\d\Omega}=0,
\end{cases} \qquad \& \qquad
\begin{cases}
A_0 u^0=g,\quad \text{ in $\Omega$},\\
u^0|_{\d\Omega}=0.
\end{cases}
\end{equation*}
Then, the following estimate
\begin{align}
\label{ell.err.L2}
&\|u^\eb-u^0\|\leq C\eb\|g\|,
\end{align}
holds for some constant $C=C(\nu,\Omega)$. 
\end{theorem}

\begin{rem}
\label{rem:resolvents}
Note that inequality \eqref{ell.err.L2} is equivalent to the following operator estimate on resolvents:
\[
\Vert A^{-1}_\eb - A^{-1}_0 \Vert_{\Cal L(L^2(\Omega))} \le C \eb.
\]
\end{rem}
In what follows we wish to compare properties of the semi-groups  associated to \eqref{eq.dw} and \eqref{eq.dwh} via estimates in terms of $\eb$. In fact, we shall provide stronger estimates in terms of the difference $\Vert A^{-1}_\eb - A^{-1}_0 \Vert_{\Cal L(L^2(\Omega))}$. The mentioned $\eb$ estimates then immediately follow by Remark \ref{rem:resolvents}.  

Our first important result is the following continuity estimate.
\begin{theorem}
\label{th.|ue-uh|E-1}
Let $\Cal E^2_\eb$ be the 
set \eqref{Dep}, $R>0$. Then, for all $\xi\in B_{\Cal E^2_\eb}(0,R) = \{ \xi \in \Cal E^2_\eb,\, \| \xi \|_{\Cal E^2_\eb} \le R\}$,  the inequality 
\begin{equation}
\label{est.|ue-uh|E-1}
\|S_{\eb}(t)\xi-S_0(t)\xi\|_{\E^{-1}}\leq M e^{K t} \Vert A^{-1}_\eb - A^{-1}_0 \Vert_{\Cal L(L^2(\Omega))},\quad t\geq 0,
\end{equation}
holds for some non-decreasing functions $M=M(R,\|g\|)$ and $K=K(R,\|g\|)$ which are independent of $\eb>0$.
\end{theorem}   
\begin{proof}[Proof of Theorem \ref{th.|ue-uh|E-1}]
Let us fix $\xi$,  set $
\xi_{u^\eb}(t):=S_{\eb}(t)\xi$, $ \xi_{u^0}(t):=S_0(t)\xi,
$
and define $\r:=\u-u^0$. Then, $\r$ solves
\begin{equation}
\label{eq.re}
\begin{cases}
\Dt^2\r+\gamma\Dt\r+A_0 \r= A_0 u^\eb - A_\eb u^\eb + f(u^0) - f(\u),\quad x\in\Omega,\ t\geq0,\\
\xi_{\r}|_{t=0}=0,\qquad \r|_{\d\Omega}=0.
\end{cases}
\end{equation}
By testing the first equation  in \eqref{eq.re} with $A_{0}^{-1}\Dt\r$ we deduce that 
\begin{multline}
\label{drE-1}
\frac{d}{dt}\Big(\frac{1}{2} 
{ \big( \Dt\r, A_{0}^{-1}\Dt\r\big) }+\frac{1}{2} \|\r\|^2\Big)+\gamma \big( \Dt\r, A_{0}^{-1}\Dt\r\big)=\\
\big( A_0 u^\eb-A_\eb u^\eb,A_{0}^{-1}\Dt \r \big)+\big(f(u^0)-f(\u),A_{0}^{-1}\Dt\r\big).
\end{multline}
We compute
\[
\begin{aligned}
\big( A_0 u^\eb-A_\eb u^\eb,A_{0}^{-1}\Dt \r \big) = \big( A_0 u^\eb,A_{0}^{-1}\Dt\r\big) - \big( A_\eb u^\eb,A_{0}^{-1}\Dt\r\big)
 &  =  \big(  u^\eb,\Dt\r\big) -  \big( A_\eb u^\eb,A_{0}^{-1}\Dt\r\big)\\
  &  =  \big( A_\eb u^\eb,A_{\eb}^{-1}\Dt\r\big) -  \big( A_\eb u^\eb,A_{0}^{-1}\Dt\r\big) \\
   &  = \big( A_\eb u^\eb, (A_{\eb}^{-1}-A_{0}^{-1})\Dt\r\big).
\end{aligned}
\]
Furthermore, 
\[
\begin{aligned}
 \big( A_\eb u^\eb, (A_{\eb}^{-1}-A_{0}^{-1})\Dt\r\big) = \frac{d}{dt}  \big( A_\eb u^\eb, (A_{\eb}^{-1}-A_{0}^{-1})\r \big)   -  \big( A_\eb \Dt u^\eb, (A_{\eb}^{-1}-A_{0}^{-1})\r\big).
\end{aligned}
\]
Therefore, we can rewrite \eqref{drE-1} as
\begin{equation}
\begin{aligned}
&\frac{d}{dt} \Lambda +\gamma \big( \Dt\r, A_{0}^{-1}\Dt\r\big)  = -  \big( A_\eb \Dt u^\eb, (A_{\eb}^{-1}-A_{0}^{-1})\r\big)+\big(f(u^0)-f(\u),A_{0}^{-1}\Dt\r\big),
\end{aligned}\label{drE-1.1}
\end{equation}
for
\[
\Lambda(t) : =\tfrac{1}{2} 
{ \big( \Dt\r(t), A_{0}^{-1}\Dt\r(t)\big) }+\tfrac{1}{2} \|\r(t)\|^2 - \big( A_\eb u^\eb(t), (A_{\eb}^{-1}-A_{0}^{-1})\r(t) \big), \quad t\ge0.
\]
%
%
%
%
We now aim to bound the right-hand-side of \eqref{drE-1.1} in terms of $\Vert A^{-1}_\eb - A^{-1}_0 \Vert^2_{\Cal L(L^2(\Omega))}$ and $\Lambda$,  then subsequently apply Gronwall's inequality and the following  standard estimate
\begin{equation}\label{errgron1}
\nu \| \phi \|^2_{H^{-1}(\Omega)} \le  \big( \phi,  A_{0}^{-1} \phi \big)
\le \nu^{-1}\|  \phi\|^2_{H^{-1}(\Omega)}, \quad  \phi \in H^{-1}(\Omega)
\end{equation}
to deduce the desired result.

 To this end, let us first estimate the non-linear term. Using 
the growth restriction on $f'$ (see Remark \ref{rem.nonlinearity}\ref{fcubic})
and H\"older's inequality (for exponents $(p_1,p_2,p_3)=(3,2,6)$) we compute 
\begin{equation}
\label{[f(ue)-f(uh)].Ae-1re'}
\begin{aligned}
\left|\big(f(\u)-f(u^0),A_{0}^{-1}\Dt\r\big)\right| & \leq M \big( (1 + | u^\eb|^2 + |u^0|^2)|r^\eb|,|A_{0}^{-1}\Dt\r|\big) \\
& \le M\Vert 1 + | u^\eb|^2 + |u^0|^2\Vert_{L^3(\Omega)} \Vert r^\eb \Vert \Vert A_{0}^{-1}\Dt\r\Vert_{L^6(\Omega)}.
\end{aligned}
\end{equation}
Then, by the Sobolev embedding $L^6(\Omega) \subset H^1(\Omega)$, the fact that $u^\eb$ and $u^0$ are bounded in $\E$ (see dissipative estimate \eqref{est.dis}) and \eqref{errgron1} we compute
\begin{flalign*}
\left|\big(f(\u)-f(u^0),A_{0}^{-1}\Dt\r\big)\right|    \leq M \Vert r^\eb \Vert \Vert A_{0}^{-1}\Dt\r\Vert_{H^1(\Omega)}  &\le  M \Vert r^\eb \Vert \Vert \Dt\r\Vert_{H^{-1}(\Omega)} \nonumber \\
& \le M \Vert r^\eb\Vert \big( \Dt r^\eb, A_{0}^{-1}\Dt\r\big)^\frac{1}{2} \\
& \le M_1 \big( \tfrac{1}{2}  \Vert r^\eb \Vert^2 + \tfrac{1}{2} \big( \Dt r^\eb, A_{0}^{-1}\Dt\r\big) \big),
\end{flalign*}
for some positive $M_1$. By utilising the above inequality in \eqref{drE-1.1} we infer that
\begin{equation*}
\begin{aligned}
&\frac{d}{dt} \Lambda  \le  \big(    2M_1  A_\eb u^\eb - A_\eb \Dt u^\eb, (A_{\eb}^{-1}-A_{0}^{-1})\r \big) - 2M_1  \big( A_\eb u^\eb, (A_{\eb}^{-1}-A_{0}^{-1})\r \big) +
\\ & \hspace{2cm} + M_1 \big( \tfrac{1}{2}  \Vert r^\eb \Vert^2 + \tfrac{1}{2} \big( \Dt r^\eb, A_{0}^{-1}\Dt\r\big) \big).  
\end{aligned}
\end{equation*}
Now, by the dissipative estimate in $\Cal E^2_\eb$ (Theorem \ref{th.disD}) we have the following uniform bounds in $t$  and $\eb$:
\begin{equation}
\label{reptbounds}
 \Vert A_\eb u^\eb(t) \Vert + \Vert A_\eb \Dt u^\eb(t) \Vert\le M, \quad  t \ge 0, \ \eb>0,
\end{equation}
which we use along with the Cauchy-Schwarz inequality to compute
\[
\begin{aligned}
\left|  \big(    2M_1  A_\eb u^\eb - A_\eb \Dt u^\eb, (A_{\eb}^{-1}-A_{0}^{-1})\r \big)  \right|  & \leq M   \Vert A^{-1}_\eb - A^{-1}_0 \Vert_{\Cal L(L^2(\Omega))}^2 +  \tfrac{M_1}{2} \Vert  r^\eb \Vert^2.
\end{aligned}
\]
By collecting the above inequalities together we deduce that
\[
\frac{d}{dt} \Lambda \le  M   \Vert A^{-1}_\eb - A^{-1}_0 \Vert_{\Cal L(L^2(\Omega))}^2 + 2M_1 \Lambda.
\]
Consequently, by applying Gronwall's inequality and the initial data $\xi_{r^\eb}|_{t=0} = 0$ we have
\[
\tfrac{1}{2} 
{ \big( \Dt\r(t), A_{0}^{-1}\Dt\r(t)\big) }+\tfrac{1}{2}\|\r(t)\|^2 - \big( A_\eb u^\eb(t), (A_{\eb}^{-1}-A_{0}^{-1})\r(t) \big) \le e^{2M_1t}\tfrac{M}{M_1}  \Vert A^{-1}_\eb - A^{-1}_0 \Vert_{\Cal L(L^2(\Omega))}^2, \quad t\ge 0.
\]
Now, we compute
\[
\begin{aligned}
\left| \big( A_\eb u^\eb, (A_{\eb}^{-1}-A_{0}^{-1})\r \big) \right| &  \le \Vert A_\eb u^\eb \Vert  \Vert A^{-1}_\eb - A^{-1}_0 \Vert_{\Cal L(L^2(\Omega))} \Vert \r \Vert \\
& \le   \Vert A_\eb u^\eb \Vert^2  \Vert A^{-1}_\eb - A^{-1}_0 \Vert_{\Cal L(L^2(\Omega))}^2 + \tfrac{1}{4}  \Vert \r \Vert^2 .
\end{aligned}
\]
Hence, the above two inequalities along with \eqref{errgron1} and \eqref{reptbounds} demonstrate \eqref{est.|ue-uh|E-1} and the proof is complete.
\end{proof}

	Along with Theorem \ref{th.|ue-uh|E-1}, to prove error estimates on the distance between global attractors we need the following exponential attraction property of $\A^0$:
	\begin{equation}\tag{H2}
	\label{d(StB,Ah)<e-t}
	\left\{ \quad \begin{aligned}
	&\textit{ there exists a constant $\sigma>0$ such that for every bounded set $B \subset \E$ the estimate} \\[2pt]
	&\dist_{\E}(S_0(t)B,\A^0)\leq M(\|B\|_\E)e^{-\sigma t},\quad t\geq 0, \\[2pt]
	&\textit{holds for some non-decreasing function $M$.}
	\end{aligned} \right.
	\end{equation}
It is known that, for problem \eqref{eq.dwh}, the property \eqref{d(StB,Ah)<e-t} is a generic assumption in the sense that it holds for an open dense subset of forces $g \in L^2(\Omega)$ (cf. \cite{BV}).	


We are now ready to formulate and prove our main result of this section.
\begin{theorem}
\label{th.d(Ae,Ah)<e1/2.w}
Assume \eqref{mainassumptions} and \eqref{d(StB,Ah)<e-t}. Let  $\A^\eb$ and $\A^0$ be the global attractors of the dynamical systems $(\E,S_\eb(t))$ and $(\E,S_0(t))$ corresponding to the problems \eqref{eq.dw} and \eqref{eq.dwh}. Then the following estimate
\begin{equation}
\label{d(Ae,Ah)<e1/2.w}
\dist_{\E^{-1}}(\A^\eb,\A^0)\leq M\Vert A^{-1}_\eb - A^{-1}_0 \Vert_{\Cal L(L^2(\Omega))}^{\varkappa}, \quad \varkappa=\frac{\sigma}{(K+\sigma)},
\end{equation}
holds. Here, $K$ is as in Theorem \ref{th.|ue-uh|E-1}, $\sigma$ as in \eqref{d(StB,Ah)<e-t}, and $M=M(\|g\|)$ is  a non-decreasing function independent of $\eb$.
\end{theorem}
\begin{proof}
The assertion follows from the already obtained estimate \eqref{est.|ue-uh|E-1} and the exponential attraction property \eqref{d(StB,Ah)<e-t}. Indeed, let $\xi_{\eb}\in\A^\eb \subset B_{\Cal E^2_\eb}(0,R_1)$ be arbitrary. Then due to \eqref{Ae=Ke(0)} there exists a complete bounded trajectory $\xi_{\u}(t)\in\Cal K^\eb$, such that $\xi_{\u}(0)=\xi_{\eb}$. Let us fix an arbitrary $T\geq 0$ and consider $\xi_{-T,\eb}=\xi_{\u}(-T)\in\A^\eb$. By Theorem \ref{th.|ue-uh|E-1} we deduce
\begin{equation*}
\|\xi_{\eb}-S_0(T)\xi_{-T,\eb}\|_{\E^{-1}}\leq M\kappa e^{KT},\qquad \text{for } \kappa = \Vert A^{-1}_\eb - A^{-1}_0 \Vert_{\Cal L(L^2(\Omega))}.
\end{equation*}   
for some $M$ and $K$ which are independent of $\eb$ and $\xi_{\eb}\in \A^\eb$. On the other hand, due to exponential attraction \eqref{d(StB,Ah)<e-t} we have
\begin{equation*}
\dist_{\E^{-1}}(S_0(T)\xi_{-T,\eb},\A^0)\leq M e^{-\sigma T}.
\end{equation*}
Therefore, using the triangle inequality, we derive
\begin{equation}
\label{t12}
\dist_{\E^{-1}}(\xi_{\eb},\A^0)\leq M(\kappa e^{KT}+e^{-\sigma T}).
\end{equation}  
We recall that $T\geq 0$ is arbitrary and therefore we choose $T$ that minimizes the right hand side of \eqref{t12}. For example, taking $T=T(\eb)$ such that $\kappa e^{KT}=e^{-\sigma T}$ yields
\begin{equation*}
\dist_{\E^{-1}}(\xi_{\eb},\A^0)\leq 2M\Vert A^{-1}_\eb - A^{-1}_0 \Vert_{\Cal L(L^2(\Omega))}^{\varkappa},\quad \varkappa=\frac{\sigma}{(K+\sigma)},
\end{equation*}
and since $\xi_{\eb}\in\A^\eb$ is arbitrary we obtain the desired inequality \eqref{d(Ae,Ah)<e1/2.w}.
\end{proof}
To complement the convergence result in Theorem  \ref{th.AetoAh1}, we have the following error estimates.
\begin{cor}	
\label{co.d(Ae,A0)}
Assume \eqref{mainassumptions} and \eqref{d(StB,Ah)<e-t}. Let $\alpha>0$ be given by Remark \ref{rem.holder}, $\varkappa$ as in Theorem \ref{th.d(Ae,Ah)<e1/2.w} and $0\leq \beta < \alpha$. Then the inequality
\begin{equation*}
\dist_{(C^\beta(\overline{\Omega}))^2}(\A^\eb,\A^0)\leq M\Vert A^{-1}_\eb - A^{-1}_0 \Vert_{\Cal L(L^2(\Omega))}^{\theta \varkappa}, \quad \theta = \tfrac{\alpha-\beta}{2+\alpha}, 
\end{equation*}
for some non-decreasing function $M=M(\|g\|)$ which is independent of $\eb$.
\end{cor}
\begin{proof}
The corollary follows directly from the uniform boundedness of $\A^\eb$ and $\A^0$ in $\big(C^\alpha(\overline{\Omega})\big)^2$ (Remark \ref{rem.holder3}), the estimate on the distance between attractors in $\E^{-1}$ (cf. \eqref{d(Ae,Ah)<e1/2.w}) and the interpolation inequalities
\begin{align*}
& \Vert u \Vert_{L^\infty(\Omega)} \le  C \Vert u \Vert_{H^{-1}(\Omega)}^\vartheta\|u\|_{C^\alpha(\overline{\Omega})}^{1-\vartheta}, \qquad \forall u \in H^{-1}(\Omega) \cap C^\alpha(\overline{\Omega}),\ \mbox{where }\vartheta=\tfrac{\alpha}{2+\alpha},\\
& \|u\|_{C^\beta(\overline{\Omega})}\le 2  \|u\|_{C^\alpha(\overline{\Omega})}^{\beta/\alpha}  \Vert u \Vert_{L^\infty(\Omega)}^{(1-\beta/\alpha)}, \qquad \forall u \in  C^\alpha(\overline{\Omega}).
\end{align*}	
\end{proof}

\section{Approximation of global attractors with error estimates in the energy space $\E$}\label{s.gaE}
In addition to the obtained estimates in Section \ref{s.Er}  on the distance in $\E^{-1}$ we would like to obtain estimates in the energy space $\E$. Note that we can not expect, in general, convergence of the global attractors in the strong topology of $\E$, cf. Remark \ref{rem:convE}. As in the elliptic case, estimates in $H^1(\Omega)$-norm require involving the  correction  $\eb \sum_i N_i(\tfrac{\cdot}{\eb}) \partial_{x_i} u^0$ of homogenised trajectories $u^0$. To this end, we introduce the `correction' operator $\Cal T_\eb: H^2(\Omega) \rightarrow H^1(\Omega)$ given by
	\begin{equation}
	\label{ue1}
	\mathcal{T}_\eb w(x):=w(x)+\eb\sum_{i=1}^3N_i\left( \tfrac{x}{\eb} \right)\partial_{x_i}w(x),\quad  x\in \Omega.
	\end{equation}
	Here, $N_i$, $i\in\{1,2,3\}$, are the solutions to the cell problem \eqref{Ni}. 
%

Now, it is known that $N_i$, $i=1,2,3$, are multipliers in $H^1(\Omega)$; in particular the following \emph{non-trivial} estimate holds (see \cite[Section 3]{ZhPas16}): there exists $C=C(\nu,\Omega)$ such that
\[
\int_\Omega | \nabla_y N_i(\tfrac{x}{\eb}) u(x)|^2 \, {\rm d}x  \le C  \int_\Omega\big( |u(x)|^2 + \eb^2 | \nabla u(x) |^2 \big) \, {\rm d}x, \quad \forall u \in H^1(\Omega).
\]
Consequently, the following inequality
\begin{equation}
\label{Tepbound}
\| \nabla \Cal T_\eb w\| \le C \big( \| \nabla w \| + \eb \| w \|_{H^2(\Omega)} \big), \quad \forall w \in H^2(\Omega),
\end{equation}
holds for some $C>0$ independent of $\eb$ and $w$. Indeed, this follows from the above multiplier estimate and the fact $N_i \in L^\infty(Q)$ (by elliptic regularity).

Now, we are ready to present the well-known corrector estimate result in elliptic homogenisation theory which improves the $L^2$-estimate given in Theorem \ref{th.Er.ell} to $H^1$-norm. 
\begin{theorem}[Theorem 3.1, \cite{ZhPas16}]
	\label{th.Er.ellCOR}
	Let $\Omega\subset\R^3$ be a bounded smooth domain, periodic matrix $a(\cdot)$
	satisfying uniform elliptic and boundedness assumptions, $A_\eb$ and $A_0$ given by \eqref{op.A} and $g\in L^2(\Omega)$. Let also $u^\eb,u^0\in H^1_0(\Omega)$ solve the problems
	\begin{equation*}
	\begin{cases}
	A_\eb u^\eb=g,\quad \text{in $\Omega$},\\
	u^\eb|_{\d\Omega}=0,
	\end{cases} \qquad \& \qquad
	\begin{cases}
	A_0 u^0=g,\quad \text{in $\Omega$},\\
	u^0|_{\d\Omega}=0.
	\end{cases}
	\end{equation*}
	Then, the following estimate
	\begin{align}
	\label{ell.err.L2COR}
	&\|u^\eb-\Cal T_\eb u^0\|_{H^1_0(\Omega)}\leq C\sqrt{\eb}\|g\|,
	\end{align}
	holds for some constant $C=C(\nu,\Omega)$. 
\end{theorem}
\begin{rem}
	\label{rem:resolventsCOR}
	Note that inequality \eqref{ell.err.L2COR} is equivalent to the following operator estimate:
	\[
	\Vert A^{-1}_\eb g - \Cal T_\eb A^{-1}_0 g \Vert_{H^1_0(\Omega)} \le C \sqrt{\eb} \| g \|, \quad g \in L^2(\Omega).
	\]
\end{rem}

As in Theorem \ref{est.|ue-uh|E-1}, we would like to compare the distance between $S_\eb(t) \xi$, for $\xi \in \Cal E^2_\eb$, to some trajectory for $S_0$ but this time in the energy space $\E$. However, here the trajectory $S_0(t) \xi$ is not a suitable candidate as it does not have the sufficient regularity needed to apply the above corrector estimates. To overcome this difficulty we carefully choose our initial data for the homogenised problem \eqref{eq.dwh}. 

More precisely, let us recall the spaces $\Cal E^2_\eb$, $\Cal E^2_0$ given in \eqref{Dep}, \eqref{D0}, 
and introduce the bounded linear operator $\Pi_\eb: \Cal E^2_\eb \rightarrow \Cal E^2_0$ given by 
\begin{equation}
\label{xih.ch}
\Pi_\eb (\xi^1, \xi^2): = (\xi^1_0, \xi^2_0), \quad \text{ where } \quad \begin{cases}\text{the term $\xi^i_0 \in H^2(\Omega)\cap H^1_0(\Omega)$, $i=1,2$, satisfies} \\
\di(a^h\Nx\xi^i_{0})=\di(a\left( \tfrac{\cdot}{\eb} \right)\Nx \xi^i).
\end{cases}
\end{equation}
 The operator $\Pi_\eb$ has the following nice properties.
\begin{lem}
	\label{le.xih}
The operator $\Pi_\eb: \Cal E^2_\eb \rightarrow \Cal E^2_0$
 is a bijection that satisfies:
	\begin{align}
	\label{est.|xih|De}
	& \| \Pi_\eb \xi \Vert_{\Cal E^2_0} = \| \xi \|_{\Cal E^2_\eb},  &\xi\in\Cal E^2_\eb;\\
	\label{|xih-xie|<Me}
	& \|\Pi_\eb \xi-\xi\|_{(L^2(\Omega))^2}\leq \Vert A^{-1}_\eb - A^{-1}_0 \Vert_{\Cal L(L^2(\Omega))}\|\xi\|_{\Cal E^2_\eb} ,\hspace{-2cm}  &\xi\in\Cal E^2_\eb.
	\end{align}
\end{lem}

\begin{proof} The bijective property and equality \eqref{est.|xih|De} directly follow from the definitions of $\Cal E^2_\eb$, $\Cal E^2_0$ and the identity $\Pi_\eb (\xi^1 , \xi^2) =  (A^{-1}_0 A_\eb \xi^1,A^{-1}_0 A_\eb\xi^2)$. Inequality \eqref{|xih-xie|<Me} follows from the identity
	\[
	A^{-1}_0A_\eb  \xi^i - \xi^i =  (A^{-1}_0  - A_\eb^{-1}) A_\eb \xi^i.
	\]
\end{proof}

We now compare $S_\eb(t) \xi$ with $ S_0(t)\Pi_\eb \xi$ in $\E$ for $\xi \in \Cal E^2_\eb$.
The following result is the direct analogue of Theorem \ref{th.|ue-uh|E-1} when one replaces the initial data $\xi$ by $\Pi_\eb\xi$ in problem \eqref{eq.dwh}.
\begin{theorem}\label{th.Dterror}
Let $\Cal E^2_\eb$ be the 
set \eqref{Dep}.  
 Then, for every  $\xi \in B_{\Cal E^2_\eb}(0,R)$, the following inequalities
	\begin{align}
\label{dttreasy}
&\Vert S_\eb(t) \xi - S_0(t) \Pi_\eb \xi \Vert_{\E^{-1}}\leq Me^{Kt} \Vert A^{-1}_\eb - A^{-1}_0 \Vert_{\Cal L(L^2(\Omega))},\quad t\geq 0,  \\ 	\label{Dttestimate1}
&	
\Vert \Dt S_\eb(t) \xi - \Dt S_0(t) \Pi_\eb \xi \Vert_{\E^{-1}}\leq Me^{Kt} \Vert A^{-1}_\eb - A^{-1}_0 \Vert_{\Cal L(L^2(\Omega))}^{1/2},\quad t\geq 0,
	\end{align}
	hold	for some non-decreasing functions $M=M(R,\|g\|)$ and $K=K(R,\|g\|)$ which are independent of $\eb>0$.
\end{theorem}
\begin{proof}
First note that inequality \eqref{dttreasy} is a consequence of the Lipschitz continuity of $S_0$ in $\E^{-1}$ (Corollary \ref{cor.contE-1}), Lemma \ref{le.xih}  and \eqref{est.|ue-uh|E-1}. Indeed,
\[
\begin{aligned}
\Vert S_\eb(t) \xi - S_0(t) \Pi_\eb \xi \Vert_{\E^{-1}} &  \le \Vert S_\eb(t) \xi - S_0(t) \xi \Vert_{\E^{-1}} + \Vert S_0(t) \xi - S_0(t) \Pi_\eb \xi \Vert_{\E^{-1}} \\
&\le \Vert S_\eb(t) \xi - S_0(t) \xi \Vert_{\E^{-1}} + M e^{Kt} \Vert \xi -  \Pi_\eb \xi \Vert_{\E^{-1}} \le Me^{Kt} \Vert A^{-1}_\eb - A^{-1}_0 \Vert_{\Cal L(L^2(\Omega))}.
\end{aligned}
\]
 It remains to prove \eqref{Dttestimate1}.

Set $\xi_{u^\eb}(t):=S_\eb(t)\xi$, $\xi_{u^0}(t):=S_0(t)\Pi_\eb\xi$. We begin by noting the following uniform bounds in $t$ and $\eb$:
\begin{equation}
\label{Eerror.e1}
\begin{aligned}
\| \Dt^2 u^\eb\| + \| \Nx  u^\eb \| + \| A_\eb  \Dt u^\eb \|  + \| \Dt^2 u^0\| +    \| \Nx  u^0 \|  \le M.
\end{aligned}
\end{equation}
Indeed, these bounds  are a consequence of identity 
 $\Pi_\eb B_{\Cal E^2_\eb}(0,R) = B_{\Cal E^2_0}(0,R)$ and the dissipative estimates for $u^\eb$ and $u^0$ in $\Cal E^2_\eb$ and $\Cal E^2_0$ respectively (Theorem \ref{th.disD} for $a=a(\tfrac{\cdot}{\eb})$ and  $a=a^h$ respectively).
 

 Now, the difference $r^\eb : = u^\eb - u^0$ solves
\begin{equation}\label{secondrep}\begin{cases}
\Dt^2\r = -\gamma\Dt\r+ A_0 u^0 - A_\eb u^\eb + f(u^0) - f(\u),\quad x\in\Omega,\ t\geq0,\\
\xi_{\r}|_{t=0}=\xi - \Pi_\eb \xi,\quad \r|_{\d\Omega}=0.
\end{cases}
\end{equation}
Note that by the definition of $\Pi_\eb, \eqref{xih.ch}$, we have
\[
\xi_{\Dt r^\eb}|_{t=0} = \big(\xi^2 - \xi^2_0 , \gamma ( \xi^2_0 -  \xi^2) + f(\xi^1_{0}) - f(\xi^1) \big).
\]
Upon handling the non-linearity as in \eqref{[f(ue)-f(uh)].Ae-1re'}, and utilising Lemma \ref{le.xih} we conclude that
\begin{equation}
\label{DttrID}
\Vert \xi_{\Dt r^\eb}(0) \Vert_{\E^{-1}} \le C \Vert A^{-1}_\eb - A^{-1}_0 \Vert_{\Cal L(L^2(\Omega))}.
\end{equation}

Now,  by differentiating the first equation in \eqref{secondrep} in time (and then  adding $A_0 \Dt r^\eb$ to both sides) we find that $q^\eb : = \Dt r^\eb$ solves
\begin{equation*}
\begin{cases}
\Dt^2\q+\gamma\Dt\q + A_0 \q= A_0 \Dt  u^\eb - A_\eb \Dt u^\eb + f'(u^0) \Dt u^0  - f'(\u)\Dt \u,\quad x\in\Omega,\ t\geq0,\\
\xi_{\q}|_{t=0}=\xi_{\Dt r^\eb}(0),\qquad \q|_{\d\Omega}=0.
\end{cases}
\end{equation*}
Testing the first equation in the above problem  with $A_{0}^{-1}\Dt\q$ gives  
\begin{multline*}
\frac{d}{dt}\Big(\frac{1}{2} 
{ \big( \Dt\q, A_{0}^{-1}\Dt\q\big) }+\frac{1}{2} \|\q\|^2\Big)+\gamma \big( \Dt\q, A_{0}^{-1}\Dt\q\big)=\\
\big( A_0 \Dt u^\eb-A_\eb \Dt u^\eb,A_{0}^{-1}\Dt \q \big)+\big(f'(u^0)\Dt u^0 -f'(\u)\Dt \u,A_{0}^{-1}\Dt\q\big).
\end{multline*}

We aim to prove the inequality 
\begin{equation}
\label{dttrgron}
\frac{d}{dt} \Lambda \le M e^{Kt} \Vert A^{-1}_\eb - A^{-1}_0 \Vert_{\Cal L(L^2(\Omega))} + M \Lambda, \quad \Lambda : = \frac{1}{2} 
{ \big( \Dt\q, A_{0}^{-1}\Dt\q\big) }+\frac{1}{2} \|\q\|^2
\end{equation}
for some $M$ and $K$ independent of $\eb$ and $\xi_0$, which subsequently implies the desired result  via an application of Gronwall's inequality and \eqref{DttrID}. As usual, we shall utilise the $H^{-1}$-norm equivalence given by \eqref{errgron1}.

So it remains to prove \eqref{dttrgron}.  By arguing as in Theorem \ref{th.|ue-uh|E-1}, we utilise the identity $\Dt \q = \Dt^2 \u - \Dt^2 u^0$  and uniform bounds  \eqref{Eerror.e1}   to compute
 \begin{equation}
 \label{e.toimprove}
\begin{aligned}
| \big( A_0 \Dt u^\eb-A_\eb \Dt u^\eb,A_{0}^{-1}\Dt \q \big) | & = | \big( A_\eb \Dt u^\eb,(A^{-1}_\eb -A_{0}^{-1})\Dt \q \big) | \\  &\le \Vert A_\eb \Dt \u \Vert \Vert A^{-1}_\eb - A^{-1}_0 \Vert_{\Cal L(L^2(\Omega))} \Vert \Dt \q \Vert  \\ & \le M\Vert A^{-1}_\eb - A^{-1}_0 \Vert_{\Cal L(L^2(\Omega))}.
\end{aligned} 
 \end{equation}
Let us now handle the non-linear term. We compute 
\[
\big(f'(u^0)\Dt u^0 -f'(\u)\Dt \u,A_{0}^{-1}\Dt\q\big) = -\big(f'(u^0)\q , A^{-1}_0 \Dt \q \big) + \big ( (f'(u^0) -f'(\u))\Dt \u,A_{0}^{-1}\Dt\q\big) =: I_1 + I_2.
\] 
The arguments to bound  $I_1$ and $I_2$ will use the uniform  bounds on $\u$ and $u^0$ given by \eqref{Eerror.e1}.

  By the growth condition on $f$ and the $H^{-1}$-norm  equivalence \eqref{errgron1}, we compute
\[
\begin{aligned}
| I_1| = | \big(f'(u^0)\q , A^{-1}_0 \Dt \q \big) | & \le M \big(  (1+| u^0|^2)  |\q|, |A^{-1}_0 \Dt \q| \big) \le M \Vert 1+ |u^0|^2\Vert_{L^3(\Omega)}  \Vert  \q \Vert  \Vert  A^{-1}_0 \Dt \q \Vert_{L^6(\Omega)} \\
& \le  M \Vert  \q \Vert  \Vert  \Dt \q \Vert_{H^{-1}(\Omega)}   \le M\big( \tfrac{1}{2} \Vert  \q \Vert^2 +\tfrac{1}{2}  ( \Dt \q, A^{-1}_0 \Dt \q ) \big).
\end{aligned}
\]
Additionally, by H\"{o}lder's inequality (for exponents $(p_1,p_2,p_3,p_4) = (6,2,6,6)$) we compute
\[
\begin{aligned}
|I_2| =|  \big ( (f'(u^0) -f'(\u))\Dt \u,A_{0}^{-1}\Dt\q\big)| &  \le M\big( (1 + |u^0| + |\u| ) | \r | |\Dt \u | , | A_{0}^{-1}\Dt\q| \big)  \\
& \le M \Vert 1 + |u^0| + |\u| \Vert_{L^6(\Omega)} \Vert \r \Vert  \Vert \Dt \u\Vert_{L^6(\Omega)}   \Vert A_{0}^{-1}\Dt\q \Vert_{L^6(\Omega)} \\
& \le M \big( \tfrac{1}{2} \Vert \r \Vert^2 +  \tfrac{1}{2}( \Dt \q , A^{-1}_0 \Dt \q)  \big).
\end{aligned}
\]

The above assertion and \eqref{dttreasy} imply
\[
|I_2|   \le C \big( e^{2Kt} \Vert A^{-1}_\eb - A^{-1}_0 \Vert_{\Cal L(L^2(\Omega))}^2 + \tfrac{1}{2}( \Dt \q , A^{-1}_0 \Dt \q) \big).
\] 

Combining the above calculations leads to the inequality \eqref{dttrgron}. The proof is complete.
\end{proof}
The following estimate  is an immediate consequence of Theorem \ref{th.Dterror} and standard elliptic theory.
\begin{cor}
	\label{co.|ue-uh|E}
Let $\Cal E^2_\eb$ be the set \eqref{Dep}, $\xi\in B_{\Cal E^2_\eb}(0,R)$ and set $\xi_{u^\eb}(t):=S_\eb(t)\xi$, $\xi_{u^0}(t):=S_0(t)\Pi_\eb\xi$. Let $\Cal T_\eb$ be given by \eqref{ue1}.
Then, the following inequality 
\begin{equation}\label{Dttestimate}
	\|u^\eb(t)-\Cal T_\eb u^0(t)\|_{H^1(\Omega)}
\leq  Me^{Kt} \sqrt{\eb},\quad t\geq 0,
\end{equation}
holds	for some non-decreasing $M=M(R,\|g\|)$ and $K=K(R,\|g\|)$ which are independent of $\eb>0$.
\end{cor}
\begin{proof}
Note that $u^\eb \in H^1_0(\Omega)$ satisfies the equation
\[
A_\eb u^\eb = - \Dt^2 u^\eb - \gamma \Dt u^\eb - f(u^\eb) + g = : F_\eb(t), \quad t\ge 0,
\]
and $u^0 \in H^1_0(\Omega)$ satisfies
\[
A_0 u^0 = - \Dt^2 u^0 - \gamma \Dt u^0 -f(u^0) +g =: F_0(t), \quad t \ge0.
\]
Since  $\xi \in B_{\Cal E^2_\eb}(0,R)$ then by \eqref{est.|xih|De} we have $\Pi_\eb \xi \in B_{\Cal E^2_0}(0,R)$ and the dissipative estimate in $\Cal E^2_0$ (Theorem \ref{th.disD} for $a=a^h$) gives  $F_0 \in L^\infty\big( \R_+ ; L^2(\Omega) \big)$. Let us introduce the intermediate function $\widetilde{u}_\eb = \widetilde{u}^\eb(t) \in H^1_0(\Omega)$ the solution to
\[
A_\eb \tilde{u}^\eb = F_0(t), \quad t\ge0.
\] 
Then, by Theorem \ref{th.Er.ellCOR} we have
\[
\| \tilde{u}^\eb(t) - \Cal T_\eb u^0(t) \|_{H^1_0(\Omega)} \le C \sqrt{\eb} \| F_0(t) \|, \quad t\ge0,
\]
and, since $A^{-1}_\eb$ is uniformly bounded in $\Cal L(H^{-1}(\Omega), H^1_0(\Omega))$, we have
\[
\| u^\eb(t) - \tilde{u}^\eb(t) \|_{H^1_0(\Omega)} \le C \| F_\eb(t) - F_0(t) \|_{H^{-1}(\Omega)}, \quad t\ge0.
\]
Therefore, by the triangle inequality, we have
\begin{equation}
\label{supercool}
\| u^\eb(t) - \Cal T_\eb u^0(t) \|_{H^1_0(\Omega)} \le C \big( \sqrt{\eb} \| F_0 \|_{L^\infty(\R_+;L^2(\Omega))} +  \| F_\eb(t) - F_0(t) \|_{H^{-1}(\Omega)} \big), \quad t\ge0.
\end{equation}

Now, upon estimating the non-linear term as in the proof of Theorem \ref{th.Dterror}, along with utilising Remark \ref{rem:resolvents} and Theorem \ref{th.Dterror}, we readily deduce that
	\begin{equation*}
	\Vert F_\eb(t) - F_0(t)\Vert_{H^{-1}(\Omega)} \le  M e^{Kt}\sqrt{\eb}, \quad t\ge 0.
	\end{equation*}
	The above  inequality along with \eqref{supercool} imply the desired result and the  proof is complete.
\end{proof}

Let us now provide  estimates on the distance in the energy space. As in Corollary \ref{co.|ue-uh|E} this requires adding an appropriate correction to the attractor $\A^0$. To this end, we
introduce the corrector\linebreak $\textsc{T}_\eb : \Cal E^2_0 \rightarrow \big( L^2(\Omega) \big)^2$ which maps the pair $\xi =(\xi^1,\xi^2)$ to the pair 
\begin{equation}
\label{Tepsets}
\textsc{T}_\eb \xi = ( \Cal T_\eb \xi^1,  \xi^2). 
\end{equation}
By \eqref{Tepbound}, we readily deduce the following inequality: there exists a constant $C >0$, independent of $\eb$, such that the inequality
\begin{equation} \label{distTep}
\dist^s_{\E} ( \textsc{T}_\eb A, \textsc{T}_\eb B) \le C\big (\dist^s_{\E}(A,B) + \eb   \dist^s_{\Cal E^2_0} ( A , B ) \big), \quad A,B \subset \Cal E^2_0,
\end{equation}
holds.

By inequality \eqref{dttreasy} and Corollary \ref{co.|ue-uh|E} we have shown the following result.
\begin{cor} \label{cor.SebS0inE}
Let $\Cal E^2_\eb$ be the set \eqref{Dep},  $\xi\in B_{\Cal E^2_\eb}(0,R)$ and set $\xi_{u^\eb}(t):=S_\eb(t)\xi$, $\xi_{u^0}(t):=S_0(t)\Pi_\eb\xi$. Then, the inequality 
\[
\| S_\eb(t) \xi - \textsc{T}_\eb S_0 (t)\Pi_\eb \xi \|_{\E} \le M e^{Kt} \sqrt{\eb},
\] 
holds for some non-decreasing $M=M(R, \| g\|)$ and $K = K(R,\| g \|)$ independent of $\eb$.
\end{cor}
The following estimate on the global attractors in $\E$ holds.
\begin{theorem}\label{th.gEest}
Assume \eqref{mainassumptions} and \eqref{d(StB,Ah)<e-t}. Let $\A^\eb$ and $\A^0$ the global attractors of problems \eqref{eq.dw} and \eqref{eq.dwh} respectively, and let $\textsc{T}_\eb$ be given by \eqref{Tepsets}. 
Then, the following estimate
\begin{equation*}
\dist_{\E}(\A^\eb, \textsc{T}_\eb \A^0)\leq M{{\sqrt{\eb}}^{\varkappa}},
\end{equation*}
holds for some $M=M(\|g\|)$ which is independent of $\eb$. Here $\varkappa$ is as in Theorem \ref{th.d(Ae,Ah)<e1/2.w}.
\end{theorem}
\begin{proof}
distance in operator-norm between theThe method of proof follows along the same lines as the argument for Theorem \ref{th.d(Ae,Ah)<e1/2.w} and so we shall only sketch it here.

For $\xi_\eb \in \A^\eb$ and $T>0$,   consider $\xi_{-T,\eb} \in \A^\eb$ that satisfies $ S_\eb(T) \xi_{-T,\eb} = \xi_0$. Then, by Corollary \ref{cor.SebS0inE} we have
\[
\| \xi_\eb - \textsc{T}_\eb S_0(T) \Pi_\eb \xi_{-T,\eb} \|_{\E} \le M  e^{KT} \sqrt{\eb}.
\]
Furthermore, by \eqref{distTep} we have
\[
\dist_{\E}( \textsc{T}_\eb S_0(T) \Pi_\eb \xi_{-T,\eb} , \textsc{T}_\eb \A^0 ) \le C \big(  \dist_{\E} (S_0(T) \Pi_\eb \xi_{-T,\eb}, \A^0) + \eb  \dist_{\Cal E^2_0} (S_0(T) \Pi_\eb \xi_{-T,\eb}, \A^0) \big).
\]
Now, to control the second term on the above right  we use the fact that $\Pi_\eb \A^\eb$ and $\A^0$ are bounded subsets of $\Cal E^2_0$ (see Remark \ref{rem.holder3} and inequality \eqref{est.|xih|De}) and that we have a dissipative estimate for $S_0(t)$ on $\Cal E^2_0$ (see Theorem \ref{th.disD}). Consequently, we compute
\begin{flalign*}
\dist_{\E}(\xi_\eb, \textsc{T}_\eb \A^0) &\le \dist_{\E}( \xi_\eb,  \textsc{T}_\eb S_0(T) \Pi_\eb \xi_{-T,\eb}  ) + \dist_{\E}( \textsc{T}_\eb S_0(T) \Pi_\eb \xi_{-T,\eb} , \textsc{T}_\eb \A^0 ) \\
& \le  M_1 e^{KT} \sqrt{\eb}+ M_2 \dist_{\E} (S_0(T) \Pi_\eb \xi_{-T,\eb}, \A^0),
\end{flalign*}
and the remainder of the proof utilises the exponential attraction property of $\A^0$, as in Theorem \ref{th.d(Ae,Ah)<e1/2.w}.
\end{proof}

\begin{rem}
\label{rem.optimalbdy} { \ }

\begin{enumerate}
	\item{
	The appearance of $\sqrt{\eb}$ in \eqref{ell.err.L2COR}	is a well-known consequence of the fact that the correction $\Cal T_\eb u^0$ does not approximate well the function $u^\eb$ in a $\eb$-neighbourhood of the boundary. In particular, the reduced power of $\eb$ appears  in the estimate due to the fact that $\Cal T_\eb u^0$ does not satisfy the Dirichlet boundary conditions and a `boundary correction' is needed. In general, the explicit $\eb$-dependence  (i.e. leading-order asymptotics) of this boundary correction is not known.
	}
	\item{ In certain situations, such as when $\Omega$ is the whole space or a torus (see Remark \ref{rem.rationalperiod}), there is no need for the boundary correction and, consequently,  the error estimate \eqref{ell.err.L2COR} is order $\eb$.   In such situations we expect order $\eb^{\varkappa}$ in our estimate on the distance  between global attractors in $\E$ (Theorem \ref{th.gEest}). As it stands, our argument does not provide such an estimate and this is because the power in the right-hand side of  \eqref{Dttestimate1} is not optimal. This is consciously done to avoid unnecessary complications and we provide an argument in Appendix \ref{app.improvement} that gives the expected power.
	}
\item{ Let us return to Remark \ref{rem:convE}. In this case it is interesting to note that estimate  \eqref{ell.err.L2COR} is order $\eb$. This is simply because the cell solutions $N_i$ are trivial ($N_i \equiv  0$) and there is no need for boundary corrections; indeed, this can be readily seen by noting that  the right-hand-side in problem \eqref{Ni} is zero in this situation. Consequently  $\Cal T_\eb = I$  and (under the refinement in Appendix \ref{app.improvement}) we have the following improvement of Theorem \ref{th.gEest}:
	\[
\dist_{\E}\big( \A^\eb, \A^0) \le M \eb^{\varkappa}.
	\]
}
\end{enumerate}
\end{rem}	

\section{Exponential attractors: existence, homogenisation and convergence rates}
\label{s.ea}
 Let us recall the definition of an exponential attractor for a dynamical system.

\begin{Def}\label{de.eA} Let $S(t):\E\to\E$, $t\geq 0$, be a semi-group acting on a Banach space $\E$. Then a set $\Cal M$ is called an exponential attractor for the dynamical system $(\E,S(t))$ if it possesses the following properties:
	\par
	1. The set $\Cal M$ is compact in $\E$ with finite fractal (box-counting) dimension $\dim_f(\Cal M,\Cal E)$;
	\par
	2. The set $\Cal M$ is positively invariant:
	\[
	S(t) \Cal M \subset \Cal M, \quad \forall t \ge0;
	\]
	\par
	3. The set $\Cal M$ exponentially attracts every bounded set $B$ of $\E$, that is 
	\begin{equation*}
	\dist_{\E}(S(t)B,\Cal M)\leq M(\|B\|_\E)e^{-\sigma t},\quad t\geq 0,
	\end{equation*}
for some non-decreasing $M$ and constant  $\sigma>0$.
\end{Def}
\vspace{0pt}
\subsection{Existence of exponential attractors and continuity in $\E^{-1}$}\label{s.ExE-1}{\ }

\vspace{10pt}
Let us present our main result for this subsection.
\begin{theorem}
	\label{th.exAe}
	Assume \eqref{mainassumptions}. Then, the dynamical systems $(\E,S_{\eb}(t))$, $\eb >0 $ and $(\E,S_0(t))$ generated by problems \eqref{eq.dw} and \eqref{eq.dwh} respectively possess exponential attractors $\Cal M^\eb,\ \Cal M^0 \subset (H^1_0(\Omega))^2$ such that the following properties hold:
	\begin{enumerate}
		\item{
			$
			\| \di(a(\tfrac{\cdot}{\eb}) \Nx \xi^1) + g \|_{H^1_0(\Omega)} + \| \di(a(\tfrac{\cdot}{\eb}) \Nx \xi^2 ) \| +\ \|\xi\|_{(C^\alpha(\bar\Omega))^2}\leq M(\|g\|),\ \ \mbox{for all }\xi=(\xi^1,\xi^2)\in \Cal M^\eb;
			$
		}
		\item{
			$	
			\| \di(a^h \Nx \xi^1) + g \|_{H^1_0(\Omega)} + \| \di(a^h \Nx \xi^2 ) \| +\ \|\xi\|_{(H^2(\Omega))^2}\leq M(\|g\|),\ \ \mbox{for all }\xi=(\xi^1,\xi^2)\in \Cal M^0;
			$
		}
		\item{For every bounded set $B\subset\E$ one has 
			\begin{equation*}
			\dist_{\E}(S_\eb(t)B,\Cal M^\eb) +	\dist_{\E}(S_0(t)B,\Cal M^0)\leq M(\|B\|_\E)e^{-\sigma t},\quad t\geq 0;
			\end{equation*}
		}
		\item{
			$\hspace{.32\textwidth} \dim_f(\Cal M^\eb,\Cal E) + \dim_f(\Cal M^0,\Cal E) \le D;$
		}
		\item{$\hspace{.3\textwidth}
			\dist^s_{\E^{-1}} (\Cal M^\eb , \Cal M^0)  \leq M\|A^{-1}_\eb-A^{-1}_0\|^{\varkappa}_{\Cal L(L^2(\Omega))}.$
		}
	\end{enumerate}
	
	Here $\alpha$ is the same as in Remark \ref{rem.holder}
and the constants $M>0$, $\sigma>0$, $0<\varkappa<1$ and $D\ge 0$ are independent of $\eb$.
\end{theorem}
\begin{cor}\label{co.expat}
	Assume \eqref{mainassumptions}. Let $\alpha>0$ be given by Remark \ref{rem.holder}, $\varkappa$ as in Theorem \ref{th.exAe} and $0\leq \beta < \alpha$. Then the inequality
	\begin{equation*}
	\dist_{(C^\beta(\overline{\Omega}))^2}^s(\Cal M^\eb,\Cal M^0)\leq M\Vert A^{-1}_\eb - A^{-1}_0 \Vert_{\Cal L(L^2(\Omega))}^{\theta \varkappa}, \quad \theta = \tfrac{\alpha-\beta}{2+\alpha}, 
	\end{equation*}
	for some non-decreasing function $M=M(\|g\|)$ which is independent of $\eb$.
\end{cor}
The remainder of the section is dedicated to the proof of Theorem \ref{th.exAe}. First, we {recall} 
a variation of an abstract result which establishes the existence of an exponential attractor $\Cal M^\eb$, for a parameter-dependent family of semi-groups $S_\eb$, whose characteristics  are independent of $\eb$ (see Appendix \ref{app.E-1},  \cite[Theorem 2.10]{EMZ2005}  and \cite[Section 3, Theorem 3.1]{FG2004}).

\begin{theorem}\label{th.expAttr.E-1}
	Let $\E$ be a Banach space and  $\E^1_\eb$, $\eb \ge 0$, be a family of Banach spaces compactly embedded into $\E$ uniformly in the following sense:
	\begin{enumerate}[(i)]	
		\item{There exists $c_0$ independent of $\eb\geq 0$ such that $\|\xi\|_{\E} \leq c_0 \|\xi\|_{\E^1_\eb}$ for all $\xi\in \E^1_\eb$;}
\item{ For all $\mu>0$, $r>0$ there exists a finite cover of $B_{\E^1_\eb}(0,r)\, $ consisting of balls radius of $\mu$ in $\E$  with centers $\Cal U_\eb(\mu,r) \subset B_{\E^1_\eb}(0,\delta_r)$, for some $\delta_r \ge r$, satisfying
	\[ \mathrm{card}\, \Cal U_\eb(\mu,r) \leq N(\mu,r),
	\]
	for some finite $N(\mu,r)$ independent of $\eb$.
}
	\end{enumerate}
Let us consider, for each $\eb\ge0$, a map defined on $\E$ such that 
	\[S_\eb:\Cal O(B_\eb)\to B_\eb, \qquad \Cal O(B_\eb) : = B_\eb + \bigcup_{r\in[0,\delta_1]} r\, \Cal U_\eb(\tfrac{1}{4K},1),
	\]
where  the set $B_\eb \subset B_{\E^1_\eb}(0,R)$ is closed in $\E$. 
 Furthermore, we assume $S_\eb$ satisfies the following properties:
\begin{enumerate}
\item{	for every $\xi_1$ and $\xi_2$ from $\Cal O(B_\eb)$, the difference $S_\eb\xi_1-S_\eb\xi_2$ can be represented in the form:
	\begin{align} 
	\label{th.expAttr:spl}
	& S_\eb\xi_1-S_\eb\xi_2=v_\eb+w_\eb, \quad \text{with} \quad 
	\|v_\eb\|_\E\leq \tfrac{1}{2}\|\xi_1-\xi_2\|_\E,\quad \|w_\eb\|_{\E^1_\eb}\leq K\|\xi_1-\xi_2\|_\E,
	\end{align}
	for $K>0$ independent of $\eb$. }
\item{Furthermore, there exists a Banach space $\E^{-1} \supset \E$ such that
		\begin{align*}
	\Vert \xi \Vert_{\E^{-1}} \le c_{-1} \Vert \xi \Vert_{\E}, \quad \forall \xi \in \E; \quad & &		\Vert S_0 \xi_1 - S_0 \xi_2 \Vert_{\E^{-1}} \le L \Vert \xi_1 - \xi_2 \Vert_{\E^{-1}}, \qquad \forall \xi_1 \in \Cal O(B_\eb),\, \forall \xi_2 \in \Cal O(B_0),
		\end{align*}
for constants $c_{-1}$ and $L >0$. }
\end{enumerate}
	Then, for every $\eb\geq 0$, the discrete dynamical system $(B_\eb,S_\eb)$ possesses an exponential attractor \mbox{$\Cal M^\eb\subset \Cal O(B_\eb)$}. The exponent of attraction $\sigma>0$ is independent of $\eb\geq 0$ and   $\dim_f(\Cal M^\eb, \E) \le D$  for some positive $D$ independent of $\eb$ (see Definition \ref{de.eA}). Moreover
\begin{equation}
\label{est.dE-1(Me,Mo).g}
\begin{aligned}
&\dist^s_{\E^{\text{-}1}}(\Cal M^\eb, \Cal M^0)  \le  \\ & \hspace{8pt} C \Big( \sup_{\xi \in \Cal O(B_\eb) }\Vert S_\eb \xi  - S_0 \xi \Vert_{\E^{\text{-}1}} + \dist^s_{\E^{\text{-}1}}\big(\Cal U_\eb(\tfrac{1}{4K},1),\Cal U_0(\tfrac{1}{4K},1)\big) 
+ \dist^s_{\E^{\text{-}1}} \big(\Cal U_\eb(\tfrac{1}{K},R), \Cal U_0(\tfrac{1}{K},R)\big)   \Big)^{\varkappa},
\end{aligned}
\end{equation}
where the constants $C>0$ and $\varkappa=\varkappa(c_0,L,K,\delta_1)$ are independent of $\eb$.
\end{theorem}
The proof of Theorem \ref{th.expAttr.E-1} is postponed to Appendix \ref{app.E-1}. 
	
We now move on to the proof of Theorem \ref{th.exAe}.	As in the usual way,  we first construct exponential attractors for the \emph{discrete} dynamical systems with maps $S_\eb:=S_\eb(T)$, $S_0:=S_0(T)$, for large enough $T>0$. Then by a standard procedure, clarified below, one arrives at exponential attractors for the \emph{continuous} dynamical systems $(\E,S_\eb(t))$, $t \geq 0$.
	\begin{proof}[Proof of Theorem \ref{th.exAe}] { \ }

	\emph{Step 1: Construction of discrete exponential attractors}.
		Recall the maps $A_\eb$ and $A_0$ given by \eqref{op.A}. Let $\E= H^1_0(\Omega) \times L^2(\Omega)$, $\E^{-1} = L^2(\Omega) \times H^{-1}(\Omega)$,   and let $\E^1_\eb$ and $\E^1_0$  be given by \eqref{E1space} for $a(\cdot) = a(\tfrac{\cdot}{\eb})$ and $a(\cdot) = a^h$ respectively).
		 Then property (\emph{i}) is an immediate consequence of the uniform ellipticity of $a(\cdot)$ and Poincar\'{e}'s inequality.

{\emph{Proof of (ii)}.} We shall provide an explicit construction for the covers. Moreover, it will be important later that we produce a cover such that
\begin{equation}
\label{Ecover}
\Cal U_\eb(\mu, r) \subset \Cal E^2_\eb \cap B_{\E^1_\eb}(0,\delta_r), \qquad \& \qquad  \dist^s_{\E^{-1}} \big(\Cal U_\eb(\mu,r), \Cal U_0(\mu,r)\big) \le C_r \| A^{-1}_\eb - A^{-1}_0 \|_{\Cal L(L^2(\Omega))},
\end{equation}
for some $C_r >0$ independent of $\eb\ge 0$.

For this reason we seek a cover of $B_{\E^1_\eb}(0,r)$ in the form
\[
\bigcup_{i=1}^{N(\mu,r) }B_{\E} \big( \xi_{i \eb}, \mu \big), \quad \text{for $ \xi_{i\eb}=( A^{-1}_\eb( p_i + g) , q_{i\eb} ) \in \Cal E^2_\eb$.}
\]
To ensure $\xi_{i \eb}$ are in $\Cal E^2_\eb$ we see that $(p_i,q_{i\eb}) $ should belong to $(H^1_0(\Omega))^2$ with $ A_\eb q_{i \eb} \in L^2(\Omega)$.

We now proceed with the construction of such a cover. As $L^2(\Omega)\times H^1_0(\Omega)$ is compactly embedded in $H^{-1}(\Omega) \times L^2(\Omega)$ then, for each $\hat{\mu}>0$,  there exist finitely many $(p_i,q_{i0}), {i=1},\ldots, {N({\hat\mu,r})}$, such that
\begin{multline*}
B_{L^2(\Omega)\times H^1_0(\Omega)}\big((-g,0), r \big) \subset \bigcup_{i=1}^{N(\hat\mu,r)} B_{H^{-1}(\Omega)\times L^2(\Omega)}\big((p_i,q_{i0}), \hat{\mu}\big), \quad  (p_i,q_{i0}) \in  B_{L^2(\Omega)\times H^1_0(\Omega)}\big((-g,0),r\big).
\end{multline*}
Additionally, due to density arguments,  we can suppose
\[
(p_i,q_{i0}) \in H^1_0(\Omega) \times H^2(\Omega).
\]
Moreover,  as the eigenfunctions of $A_\eb$  form an orthonormal basis for $L^2(\Omega)$ we can find $q_{i\eb}$ such that $A_\eb q_{i  \eb} \in L^2(\Omega)$ and
\begin{equation}
\label{e.qiebound}
\| q_{i\eb}- q_{i0} \| \le \min\{ \hat \mu, \| A^{-1}_\eb - A^{-1}_0 \|_{\Cal L(L^2(\Omega))} \}, \quad i=1,\ldots,N(\hat\mu,r).
\end{equation}
Therefore, we have the covering
\[
B_{L^2(\Omega)\times H^1_0(\Omega)}\big((-g,0), r \big) \subset \bigcup_{i=1}^{N(\hat\mu,r)} B_{H^{-1}(\Omega)\times L^2(\Omega)}\big((p_i,q_{i\eb}), 2\hat{\mu}\big), \quad \eb \ge 0.
\]

Now, for fixed $\xi \in B_{\E^1_\eb}(0,r)$ we readily deduce  from the ellipticity of $a$  that \[
\| \nabla\big( \xi^1 - A^{-1}_\eb(p_i +g) \big)\| \le \nu^{-1} \| A_\eb \xi^1 -p_i -g \|_{H^{-1}(\Omega)}.
\] 
Furthermore, it is clear that  $(A_\eb \xi^1 - g, \xi^2) \in B_{L^2(\Omega)\times H^1_0(\Omega)}\big((-g,0),r\big)$.  Consequently, one can readily check that 
\[
B_{\E^1_\eb}(0,r) \subset \bigcup_{i=1}^{N(\hat\mu,r)} B_{\E}\big( (A^{-1}_\eb(p_i +g), q_{i\eb}) , 2(1 \vee \nu^{-1}) \hat\mu   \big).
\]
Additionally, since $q_{i\eb}$ are obtained by truncating $q_{i 0}$ with respect to the eigenfunctions of $A_\eb$, we compute 
\[
\| \nabla q_{i \eb} \|^2 \le \nu^{-1} (A_\eb q_{i\eb}, q_{i\eb}) \le \nu^{-1} (A_\eb q_{i0}, q_{i0}) \le \nu^{-2} \| \nabla q_{i0}\|^2, 
\]
and so we deduce that
\[(A^{-1}_\eb(p_i +g), q_{i\eb})  \in B_{\E^1_\eb}(0, (1\vee \nu^{-1}) r).\]
Hence, upon setting $\hat\mu = \tfrac{1}{2(1\vee \nu^{-1})} \mu$, we see that the centers 
\begin{equation}
\label{Ue(mu,r)}
\Cal U_{\eb}(\mu,r) : = \big\{ \big( A^{-1}_\eb (p_i + g), q_{i\eb} \big) \, | \, i =1,\ldots, N\big(  \tfrac{1}{2(1\vee \nu^{-1})} \mu, r\big)  \big\}, \quad \eb\ge0,
\end{equation}
satisfy \emph{(ii)} for $\delta_r = (1\vee \nu^{-1}) r$.  Also the additional desired properties \eqref{Ecover} hold.

\vspace{10pt}

{\emph{Construction of $B_\eb$ and $S_\eb$.}} We set $B_\eb:=B_{\Cal E^2_\eb}(0,R_2)$ to be the absorbing ball provided by Theorem \ref{th.disD} for $\Cal E^2=\Cal E^2_\eb$, and $a(\cdot)=a\left(\tfrac{\cdot}{\eb}\right)$ in the case $\eb>0$ and $a(\cdot)\equiv a^h$ for $\eb=0$. The radius $R_2$ is independent of $\eb$ and clearly  $B_\eb$ is closed in $\E$. 

 Since $B_\eb$ is an absorbing set in $\Cal E^2_\eb$ and,  by  \eqref{Ecover},  $\Cal O(B_\eb)$ is a subset of $\Cal E^2_\eb$,   we can choose $T_1$ large enough (and independent of $\eb$) such that $S_\eb : = S_\eb(T)$, $\eb \ge0$, satisfies
\[
S_\eb:\Cal O(B_\eb)\to B_\eb, \qquad \Cal O(B_\eb)  = B_\eb + \bigcup_{r\in[0,\delta_1]} r\, \Cal U_\eb(\tfrac{1}{4K},1).
\]
 Let us verify properties $(\emph{1})$ and $(\emph{2})$ of $S_\eb$.

{\emph{Proof of (1)}.} For $\xi_i \in \Cal O(B_\eb) \subset \Cal E^2_\eb$, $i=1,2$, let $u_i(t) = S_\eb(t) \xi_i$. Consider the splitting $u_i = v_i + w_i$ given by \eqref{eq.ve2}-\eqref{eq.we2}, and set $v = v_1 - v_2$ and $w = w_1 - w_2$.

As the equation for $v$ is linear then obviously the inequality
\[
\| v(T_2) \|_{\E} \le \tfrac{1}{2} \| \xi_1 - \xi_2 \|_{\E},
\]
holds for large enough time $T_2$ (independent of $\eb$).

From \eqref{eq.we2} we find that $w$ solves 
\begin{equation}
\label{eq.we2.1}
\begin{cases}
\Dt^2 w+\gamma\Dt w-\di(a \Nx w)=f(u_2)-f(u_1),\quad x\in\Omega,\ t\geq 0,\\
\xi_{w}|_{t=0}=(0,0),\quad w|_{\d\Omega}=0,\\
\end{cases}
\end{equation}
for $a = a(\tfrac{\cdot}{\eb})$ or $a\equiv a^h$. Moreover, $p = \Dt w$ solves
\begin{equation}
\label{eq.we2.2}
\begin{cases}
\Dt^2 p+\gamma\Dt p-\di(a \Nx p)=f'(u_2)\Dt u_2 -f'(u_1)\Dt u_1,\quad x\in\Omega,\ t\geq 0,\\
\xi_{p}|_{t=0}=(0,f(\xi^1_1) - f(\xi^1_2)),\quad p|_{\d\Omega}=0.\\
\end{cases}
\end{equation}
Using the fact that our initial data is from $\Cal E^2_\eb$ we conclude that $u_i$, $\Dt u_i$ are bounded in $L^\infty(\Omega)$ uniformly in $\eb$. Then upon testing the first equation in \eqref{eq.we2.2} with $\Dt p$, rewriting the subsequent right-hand-side in the form 
\[
\big( f'(u_2) (\Dt u_2 - \Dt u_1)  ,\Dt p\big) + \big( (f'(u_2) - f'(u_1) )  \Dt u_1  ,\Dt p\big),
\]
we obtain via standard arguments, and the Lipschitz continuity of $S_\eb(t)$ in $\E^{-1}$ (Corollary \ref{cor.contE-1}),  the uniform estimate
\[
\| \Dt p(t) \| + \| \Nx p(t) \| \le M e^{Kt} \| \xi_1 - \xi_2 \|_{\E^{-1}}, \quad t\ge 0.
\]
Consequently, we use $p = \Dt w$ and \eqref{eq.we2.1} to conclude
\[
\| \xi_{w}(t) \|_{\E^1_\eb}  \le Me^{Kt} \| \xi_1 - \xi_2 \|_{\E^{-1}}, \quad t\ge0,
\] 
for some positive constants $M$ and $K$ independent of $\eb$ and $\xi_i$. Therefore, for $T = \max\{T_1,T_2\}$, property $\emph{(1)}$ holds.

\emph{Proof of (2)}. This property is given by Corollary \ref{cor.contE-1} for $a \equiv a^h$.

Hence, the assumptions of Theorem \ref{th.expAttr.E-1} hold and therefore Theorem \ref{th.exAe} holds for the discrete dynamical systems $(B_\eb,S_\eb(T))$ with discrete exponential attractors $\Cal M^{\eb}_d$. Indeed, Theorem \ref{th.exAe} \emph{(1)}-\emph{(4)} hold due to the choice of $B_\eb$ and $\Cal U_\eb$, and \emph{(5)} follows from \eqref{est.dE-1(Me,Mo).g}, \eqref{Ecover} and Theorem \ref{th.|ue-uh|E-1}.

\emph{Step 2: Discrete to continuous dynamics.} From the discrete exponential attractors $\Cal M^\eb_d$ we can build exponential attractors $\Cal M^\eb$ for the original dynamical systems $(\E,S_\eb(t))$ by the following standard construction  (\cite{MZ.hbk2008}):
\begin{equation}
\Cal M^\eb := \bigcup_{\tau\in [0,T]}S_\eb(\tau)\Cal M^\eb_d,\qquad \eb\geq 0.
\end{equation} 
Indeed, the properties \emph{(1)}-\emph{(4)} can be easily verified due to dissipative estimate in $\Cal E^2_\eb$, Lipschitz continuity with respect to initial data in $\E$ (Corollary \ref{co.St.inC}) on the bounded set $B_\eb$:
\[
\| S_\eb(t) \xi_1 - S_\eb(t) \xi_2 \|_{\E} \le M \| \xi_1 - \xi_2 \|_{\E}, \quad  \xi_1,\xi_2 \in B_\eb, \quad \eb \ge0, 
\] and Lipschitz continuity with respect to time:
\begin{equation*}
	\|S_\eb(\tau_1)\xi-S_\eb(\tau_2)\xi\|_\E\leq M|\tau_1-\tau_2|,\quad \tau_1,\tau_2\in[0,T],\ \xi\in B_\eb,\qquad \eb\geq 0,
\end{equation*}
for some constant $M>0$ (independent of $\eb$). Indeed, the continuity in time follows from the uniform boundedness of $B_\eb$ in the space $\E^1_\eb$. It remains to check the continuity property \emph{(5)} for the exponential attractors $\Cal M^\eb$. This readily follows from the fact that \emph{(5)} holds for the discrete exponential attractors $\Cal M^\eb_d$, {Theorem \ref{th.|ue-uh|E-1}} and the following computation:
\begin{flalign*}
\dist^s_{\E^{-1}}(\Cal M^\eb, \Cal M^0) & = \dist^s_{\E^{-1}}\big(  \bigcup_{\tau\in [0,T]}S_\eb(\tau)\Cal M^\eb_d,  \bigcup_{\tau\in [0,T]}S_0(\tau)\Cal M^0_d \big) \\
& \le \sup_{\tau \in [0,T]}  \dist^s_{\E^{-1}}\big( S_\eb(\tau)\Cal M^\eb_d,  S_0(\tau)\Cal M^0_d \big) \\
& \le \sup_{\tau \in [0,T]}  \dist^s_{\E^{-1}}\big( S_\eb(\tau)\Cal M^\eb_d,  S_0(\tau)\Cal M^\eb_d \big) + \sup_{\tau \in [0,T]}  \dist^s_{\E^{-1}}\big( S_0(\tau)\Cal M^\eb_d,  S_0(\tau)\Cal M^0_d \big) \\
& \le \sup_{\tau \in [0,T]} M e^{K\tau} \| A^{-1}_\eb - A^{-1}_0 \|_{\Cal L(L^2(\Omega))} + \sup_{\tau \in [0,T]} L \dist^s_{\E^{-1}}\big( \Cal M^\eb_d,  \Cal M^0_d \big).
\end{flalign*}
\end{proof}

\subsection{Continuity of exponential attractors in $\E$.}\label{s.dE(Me,M0)}{\ }

\vspace{10pt}
Theorem \ref{th.exAe}{\em (5)} demonstrates H\"{o}lder continuity  between the exponential attractors $\Cal M^\eb$ and $\Cal M^0$ in the space $\E^{-1}$.
In this section we provide continuity results in the energy space $\E$. Unlike in $\E^{-1}$, in the stronger topology of $\E$ this requires a correction (such as in Definition \ref{Tepsets}) of the exponential attractor $\Cal M^0$. More precisely, the main result of this section is the following theorem. 
\begin{theorem}
	\label{th.expAttr.E.wv}	
	Assume \eqref{mainassumptions} and let $\Cal M^\eb$, $\Cal M^0$ be the exponential attractors constructed in \mbox{Theorem \ref{th.exAe}}. Then, the following estimate is valid:
	\begin{equation}
	\label{est.dE(Me,TeMo).wv}
	\dist^s_{\E}(\Cal M^\eb,\textsc{T}_\eb \Cal M^0)\leq M\sqrt{\eb}^{\varkappa},\quad \eb>0,
	\end{equation}
	where the `correction' operator $\textsc{T}_\eb$ is given by \eqref{Tepsets}, $0<\varkappa<1$ as in Theorem \ref{th.exAe}  and the constant $M>0$  is independent of $\eb$. 
\end{theorem} 
To prove this result, we make an important development of Theorem \ref{th.expAttr.E-1}  to provide estimates between exponential attractors which admit correction.  That is we establish the following new result.
%
%
%
%
%
\begin{theorem}
\label{th.expAttr.E}	
Let assumptions of Theorem \ref{th.expAttr.E-1} be satisfied and $\Cal M^\eb$, $\Cal M^0$ be the exponential attractors constructed therein.
Additionally, assume that:
\begin{enumerate}
	\setcounter{enumi}{2}
	\item for every $\eb> 0$ there exists a  bijection $\Pi_\eb : \E^1_\eb  \rightarrow \E^1_0$
		that satisfies 
		\[
 \Pi_\eb B_\eb=  B_0;
		\]
		\item for every $\eb>0$ there exists a `correction' operator $\textsc{T}_\eb:\E^1_0\to\E$ which possesses the property
		\begin{equation*}
			\|\textsc{T}_\eb\xi_1-\textsc{T}_\eb\xi_2\|_{\E}\leq L_{\rm cor} \|\xi_1-\xi_2\|_\E + m(\eb)\qquad \mbox{for all }\xi_1,\ \xi_2\in\Cal O(B_0);
		\end{equation*}
for some constant $L_{\rm cor}>0$ independent of $\eb$ and positive function $m(\cdot)$  with $m(0^+)=0$.
		\item the maps $S_\eb$ are uniformly Lipschitz continuous in $\E$ with respect to $\eb>0$, that is
		\begin{equation*}
			\|S_\eb\xi_1-S_\eb\xi_2\|_{\E}\leq L\|\xi_1-\xi_2\|_\E,\qquad \forall \xi_1, \xi_2\in\Cal O(B_\eb),
		\end{equation*}
		with some constant $L>1$ independent of $\eb>0$.
\end{enumerate}
Then the following estimate
\begin{equation}
\label{est.dE(Me,Mo).g}
\begin{aligned}
\dist^s_{\E}(\Cal M^\eb, \textsc{T}_\eb\Cal M^0)  \le
  &C \Big(  \sup_{\xi \in  \Cal O(B_0)}\Vert S_\eb \Pi_\eb^{-1} \xi  - \textsc{T}_\eb S_0 \xi \Vert_{\E}  + \sup_{\xi \in \Cal O(B_0) }\Vert \textsc{T}_\eb\xi - \Pi_\eb^{-1}\xi\Vert_{\E}+m(\eb)  \\ 
&\dist^s_{\E}\big(\Cal U_\eb(\tfrac{1}{4K},1),\textsc{T}_\eb\,\Cal U_0(\tfrac{1}{4K},1)\big) + \dist^s_{\E} \big(\Cal U_\eb(\tfrac{1}{K},R),\textsc{T}_\eb\, \Cal U_0(\tfrac{1}{K},R)\big)   \Big)^{\varkappa},
\end{aligned}
\end{equation}
holds for constant $C>0$ independent of $\eb$ and $\varkappa$  as in Theorem \ref{th.expAttr.E-1}.
\end{theorem}
The proof of this result is 
{presented} in Appendix \ref{app:estinE}.

\begin{proof}[Proof of Theorem \ref{th.expAttr.E.wv}]
	Let the sets $B_\eb$, $\Cal O(B_\eb)$, $\eb\geq 0$, and the operator $S_\eb=S_\eb(T)$ be as in \mbox{Theorem \ref{th.exAe}}.
	
	We first establish, based on the abstract result Theorem \ref{th.expAttr.E}, the estimate \eqref{est.dE(Me,TeMo).wv} for the discrete exponential attractors $\Cal M^\eb_d$ (defined in the proof of Theorem \ref{th.exAe}). That is we prove the following inequality:
	\begin{equation}
	\label{est.dE(Med,TeMod).wv}
	\dist^s_{\E}(\Cal M^\eb_d,\textsc{T}_\eb \Cal M^0_d)\leq M\sqrt{\eb}^{\varkappa},\quad \eb>0,
	\end{equation}
for some constant $M>0$. 

Let us check that the assumptions of Theorem \ref{th.expAttr.E} hold. Indeed,  assumption \emph{(3)} follows from the fact that $B_\eb=B_{\Cal E^2_\eb}(0,R_2)$ (see the proof of Theorem \ref{th.exAe}) and Definition \ref{xih.ch} of the projector $\Pi_\eb$ (where we note that $\Pi_\eb$ can be trivially extended to the map from $\E^1_\eb$ onto $\E^1_0$, preserving the bijection property). Assumption \emph{(4)} holds with $m(\eb)=C\eb$ (for some constant $C>0$, independent of $\eb$) due to the multiplier estimate \eqref{distTep} and the fact that $\Cal O(B_0)$ is a bounded subset of $\Cal E^2_0$ by construction. Assumption \emph{(5)} is a consequence of Corollary \ref{co.St.inC}.  Hence the assumptions of Theorem \ref{th.expAttr.E} hold and \eqref{est.dE-1(Me,Mo).g} hold for the discrete exponential attractors $\Cal M^\eb_d$ and $\Cal M^0_d$.

Let us now estimate the terms on the right-hand side of \eqref{est.dE(Me,Mo).g} in terms of $\eb$. Since $\Pi_\eb:\Cal E^2_\eb\to\Cal E^2_0$ is bijective and preserves the norm (Lemma \ref{le.xih}), and since $\Cal O(B_0)\subset \Cal E^2_0$ is bounded, we see that $\|\Pi_\eb^{-1}\Cal O(B_0)\|_{\Cal E^2_\eb}=\|\Cal O(B_0)\|_{\Cal E^2_0}$; that is the set $\Pi_\eb^{-1}\Cal O(B_0)$ is bounded in $\Cal E^2_\eb$. Therefore, this observation and Corollary \ref{cor.SebS0inE} imply that   
\begin{equation}
	\sup_{\xi \in  \Cal O(B_0)}\Vert S_\eb \Pi_\eb^{-1} \xi  - \textsc{T}_\eb S_0 \xi \Vert_{\E}=\sup_{\xi \in \Pi^{-1}_\eb \Cal O(B_0)}\Vert S_\eb \xi  - \textsc{T}_\eb S_0\Pi_\eb \xi \Vert_{\E}\leq M\sqrt{\eb},
\end{equation}
for some $M>0$ independent of $\eb>0$. Also from the  identity
\begin{equation}
	\Cal T_\eb A_0^{-1}A_\eb w-w= (\Cal T_\eb A_0^{-1}-A_\eb^{-1})A_\eb w,
\end{equation}
and Remark \ref{rem:resolventsCOR} we deduce that
\begin{equation}
\label{est.Te-Pie-1.E}
	\sup_{\xi \in \Cal O(B_0) }\Vert \textsc{T}_\eb\, \xi - \Pi_\eb^{-1}\xi\Vert_{\E}=\sup_{\xi \in \Pi_\eb^{-1}\Cal O(B_0) }\Vert \textsc{T}_\eb\,\Pi_\eb\,\xi - \xi\Vert_{\E}\leq M\sqrt{\eb},
\end{equation}
for some constant $M>0$ independent of $\eb>0$. It remains to compare the distance between the covers present in the right-hand side of \eqref{est.dE(Me,Mo).g}. To this end, we notice that if $\xi_{i\eb}:=(A_\eb^{-1}(p_i+g),q_{i\eb})\in \Cal U_\eb(\mu,r)$, then
\begin{equation}
	\xi_{i\eb}-\textsc{T}_\eb\xi_{i0}=\big((A_{\eb}^{-1}-\Cal T_\eb A_0^{-1})(p_i+g), q_{i\eb}-q_{i0} \big),\qquad \eb>0.
\end{equation} 
Consequently, due to Remark \ref{rem:resolventsCOR}  and the properties of $q_{i\eb}$ (see \eqref{e.qiebound}) one can see that
\begin{equation}
	\dist^s_{\E}(\Cal U_\eb(\mu,r),\textsc{T}_\eb\, \Cal U_0(\mu,r))\leq C_r\sqrt{\eb},
\end{equation}  
for some constant $C_r>0$ independent of $\eb$, $\mu$. Upon collecting the above estimates we derive \eqref{est.dE(Med,TeMod).wv}.

	It remains to establish \eqref{est.dE(Me,TeMo).wv} for the  exponential attractors $\Cal M^\eb$. It is sufficient to show that
	\begin{equation}
	\label{e.experrorshow}
	\begin{aligned}
&	\dist^s_{\E}(\Cal M^\eb, \textsc{T}_\eb \Cal M^0)  \le L\dist^s_{\E}(\Cal M^\eb_d,\textsc{T}_\eb\Cal M^0_d)+L\sup_{\xi\in\Pi^{-1}_\eb \Cal O(B_0)}\|\textsc{T}_\eb\Pi_\eb \xi-\xi\|_{\E}+ \\
& \hspace{6cm} +	\sup_{\tau \in [0,T]}\sup_{\xi\in \Pi_\eb^{-1}\Cal O(B_0)} \|S_\eb(\tau)\xi- \textsc{T}_\eb S_0(\tau)\Pi_\eb\xi\|_{\E}.
	\end{aligned}
	\end{equation}
Indeed, since $\varkappa <1$, the above inequality, \eqref{est.dE(Med,TeMod).wv}, \eqref{est.Te-Pie-1.E} and Corollary \ref{cor.SebS0inE} implies \eqref{est.dE(Me,TeMo).wv}.

Let us demonstrate \eqref{e.experrorshow}: 
\begin{flalign*}
		\dist^s_{\E}(&\Cal M^\eb,  \textsc{T}_\eb \Cal M^0) 	= \dist^s_{\E}\big( \bigcup_{\tau\in [0,T]}S_\eb(\tau)\Cal M^\eb_d,  \bigcup_{\tau\in [0,T]}\textsc{T}_\eb S_0(\tau)\Cal M^0_d \big) \\
		&
		\le \sup_{\tau \in [0,T]}  \dist^s_{\E}\big( S_\eb(\tau)\Cal M^\eb_d,  \textsc{T}_\eb S_0(\tau)\Cal M^0_d \big)\\
		&\le \sup_{\tau \in [0,T]}  \dist^s_{\E}\big( S_\eb(\tau)\Cal M^\eb_d, S_\eb(\tau)\Pi^{-1}_{\eb}\Cal M^0_d \big)+\sup_{\tau \in [0,T]}  \dist^s_{\E}\big(S_\eb(\tau)\Pi^{-1}_{\eb}\Cal M^0_d, \textsc{T}_\eb S_0(\tau)\Cal M^0_d \big)\\ 	
		&\leq L\dist^s_{\E}(\Cal M^\eb_d,\Pi^{-1}_\eb\Cal M^0_d)+\sup_{\tau \in [0,T]}  \dist^s_{\E}\big(S_\eb(\tau)\Pi^{-1}_{\eb}\Cal M^0_d, \textsc{T}_\eb S_0(\tau)\Cal M^0_d \big)\\
		&\leq L\dist^s_{\E}(\Cal M^\eb_d,\textsc{T}_\eb\Cal M^0_d)+L\dist^s_{\E}(\textsc{T}_\eb\Cal M^0_d,\Pi^{-1}_\eb\Cal M^0_d)+\sup_{\tau \in [0,T]}  \dist^s_{\E}\big(S_\eb(\tau)\Pi^{-1}_{\eb}\Cal M^0_d, \textsc{T}_\eb S_0(\tau)\Cal M^0_d \big)\\
		&\leq L\dist^s_{\E}(\Cal M^\eb_d,\textsc{T}_\eb\Cal M^0_d)+L\sup_{\xi\in\Cal O(B_0)}\|\textsc{T}_\eb\xi-\Pi^{-1}_\eb\xi\|_{\E}+\sup_{\tau \in [0,T]}\sup_{\xi\in \Cal O(B_0)} \|S_\eb(\tau)\Pi^{-1}_{\eb}\xi- \textsc{T}_\eb S_0(\tau)\xi\|_{\E}\\
		&\leq L\dist^s_{\E}(\Cal M^\eb_d,\textsc{T}_\eb\Cal M^0_d)+L\sup_{\xi\in\Pi^{-1}_\eb \Cal O(B_0)}\|\textsc{T}_\eb\Pi_\eb \xi-\xi\|_{\E}+
		\sup_{\tau \in [0,T]}\sup_{\xi\in \Pi_\eb^{-1}\Cal O(B_0)} \|S_\eb(\tau)\xi- \textsc{T}_\eb S_0(\tau)\Pi_\eb\xi\|_{\E}. 
\end{flalign*}
Hence the theorem is proved.
\end{proof}
 
\section{The case of different boundary conditions}
\label{s.difbc}

In this section we are going to show that the analogues of the obtained homogenisation error estimates for the global and exponential attractors still hold  if we change the Dirichlet boundary conditions to be either Neumann or periodic.

Let $\Omega\subset \R^3$ be a smooth bounded domain and $\Cal H^1 : = H^1(\Omega)$ or $\Omega$ be a three-dimensional torus $\mathbb{T}^3:=[0,\ell)^3$, $\ell>0$,
with \[
\Cal H^1 : = \big\{u\in H^1(\Omega)|\ u(x+\ell e_k)=u(x),\ k\in\{1,\ 2,\ 3\}\big\}.
\]
 In both cases we endow $\Cal H^1$ with the norm
\[
\|u\|^2_{\Cal H^1}:=\|\nabla u\|^2+\|u\|^2,\quad u\in\Cal H^1.
\]
For the maps $A_\eb$ be given by \eqref{op.A},  $\eb\geq 0$, we consider the problem  
\begin{equation}\label{eq.dw.NP}
\begin{cases}
\Dt^2\u +\gamma\Dt \u +(A_\eb+1)\u+f(\u)=g(x),\quad x\in\Omega,\ t\geq 0,\\
(\u,\Dt \u)|_{t=0}=\xi,
\end{cases}
\end{equation} 
endowed with either Neumann
\begin{equation}\tag{N}\label{Neqn}
 \begin{cases}a\big(\tfrac{\cdot}{\eb}\big)\nabla \u\cdot n|_{\partial\Omega}=0,\quad \eb>0,\\[2pt]
a^h\nabla u^0 \cdot n|_{\partial\Omega}=0,\quad \eb=0,\end{cases}
\end{equation}
 or periodic
 \begin{equation}\tag{P}\label{Peqn}
\left\{  \,\begin{aligned}
& u^\eb(x+\ell e_k)=u^\eb(x),  \\
&  \Nx u^\eb(x+\ell e_k) = \Nx u^\eb(x),
\end{aligned} \right. \qquad \ k\in\{1,\ 2,\ 3\},\ \eb\ge 0, 
 \end{equation}
  boundary conditions.

It is well-known that problem \eqref{eq.dw.NP} with either boundary conditions (N) or (P)  is well-posed in the energy space $\E: = \Cal H^1\times L^2(\Omega)$  and, therefore, defines a dynamical system $(\E,S_\eb(t))$ where
\[
 S_{\eb}(t)\xi:=\xi_{\u}(t),\quad t\geq0,
\] for  $\u(t)$  the unique solution of the corresponding problem with initial data $\xi$. 

Moreover, is well-known that $A_\eb+1 : \Cal D(A_\eb+1) \subset L^2(\Omega) \rightarrow L^2(\Omega)$ is self-adjoint, where
\[
\Cal D(A_\eb+1)  = \left\{ \begin{aligned}
&\{u\in\Cal H^1|\ A_\eb u\in L^2(\Omega),\ a\big(\tfrac{\cdot}{\eb}\big)\nabla u\cdot n|_{\partial\Omega}=0\},\quad \eb>0,\\
&\{u\in\Cal H^1|\ A_0 u\in L^2(\Omega),\ a^h\nabla u\cdot n|_{\partial\Omega}=0\},\quad\ \eb=0,
\end{aligned} \right.
\]
for condition \eqref{Neqn} or 
\[
\Cal D(A_\eb+1)  = \{u\in\Cal H^1|\ A_\eb u\in L^2(\Omega), \nabla u(x+\ell e_k)=\nabla u(x),\ k\in\{1,\ 2,\ 3\}\},\ \eb\geq 0,
\]
for condition \eqref{Peqn}. Setting \begin{equation}
\label{Dep.NP}
\left\{ \begin{aligned}
	&\Cal E^2_{\eb} : = \big\{ \xi  \in  \big( \Cal D(A_\eb+1) \big)^2 \, | \, \big( A_\eb\xi^1 - g\big) \in \Cal H^1   \big\},\\
	& \Vert \xi \Vert_{\Cal E^2_{\eb}}^2 : =  \| A_\eb\xi^1  - g \|^2_{\Cal H^1} +\|(A_\eb+1)\xi^1\|^2 + \| (A_\eb+1)\xi^2  \|^2, 
\end{aligned} \right.\qquad \eb\geq 0,
\end{equation}
it is straightforward to see from Appendix \ref{s.sma} and Sections \ref{s.smth}-\ref{s.ea} that the following theorem holds.
\begin{theorem}
	\label{th.Attr1.NP}Assume \eqref{mainassumptions}. Then, for every $\eb \ge0$, the dynamical systems $(\E,S_\eb(t))$ generated by problem \eqref{eq.dw.NP} with boundary conditions \eqref{Neqn} or \eqref{Peqn} possesses a global attractor $\Cal A^\eb$, and exponential attractor $\Cal M^\eb$, of finite fractal dimension such that:
	\begin{align*}
	&\A^\eb\subset\Cal M^\eb\subset \Cal E^2_\eb,\qquad \|\A^\eb \|_{\Cal E^2_\eb}\leq\|\Cal M^\eb\|_{\Cal E^2_\eb}\leq M(\|g\|),\qquad \Cal A^\eb=\Cal K^\eb|_{t=0},\\
	&\dist_{\E}(S_\eb(t)B,\Cal M^\eb)\leq e^{-\sigma t}M(\|B\|_\E),\quad t\geq 0,\qquad \mbox{for all bounded}\ B\subset\E,\\
	&\dim_f(\A^\eb,\E)\leq\dim_f (\Cal M^\eb,\E)\leq D,
	\end{align*}
	where the constants $\sigma,\ D>0$ and non-decreasing function $M$ are independent of $\eb$. Here $\Cal K^\eb$ is the set of all bounded energy solutions to problem \eqref{eq.dw.NP}, with \eqref{Neqn} or \eqref{Peqn}, defined \emph{for all} $t\in\R$. 
\end{theorem}

Let us now discuss error estimates between the anisotropic and homogenised attractors. It is known  that the main homogenisation results, Theorems \ref{th.Er.ell} and \ref{th.Er.ellCOR}, remain valid for the case of Neumann and periodic boundary conditions.
\begin{theorem}[\cite{ZhPas16}]
	\label{th.Er.ell.NP}
	Let $\Omega\subset\R^3$ be a bounded smooth domain or three-dimensional torus $\mathbb{T}^3$, $\eb>0$, periodic matrix $a(\cdot)$ satisfying uniform ellipticity and boundedness assumptions, $A_\eb$ and $A_0$ given by \eqref{op.A} and $g\in L^2(\Omega)$. Let also $u^\eb\in\Cal D(A_\eb+1)$, $u^0 \in \Cal D(A_0+1)$, solve the equations
	\begin{equation*}
	(A_\eb+1) u^\eb=g\  \text{in $\Omega$},\qquad (A_0+1) u^0=g\ \text{in $\Omega$}.
	\end{equation*}
Then, the following estimates
	\begin{align}
	\label{ell.err.L2.NP}
	&\|u^\eb-u^0\|\leq C\eb\|g\|,\\
	\label{ell.err.H1.NP}
	&\|u^\eb-\Cal {T}_\eb u^0\|_{\Cal H^1}\leq C\sqrt{\eb}\|g\|,
	\end{align}
	hold for some constant $C=C(\nu,\Omega)$. Here the operator $\Cal T_\eb$ is given in \eqref{ue1}.
\end{theorem}

\begin{rem}
	\label{rem:resolvents.NP}
	Note that inequalities \eqref{ell.err.L2.NP} and \eqref{ell.err.H1.NP} are equivalent to the following operator estimates:
	\begin{align*}
	&\| (A_\eb+1)^{-1} - (A_0+1)^{-1} \|_{\Cal L(L^2(\Omega))} \le C \eb,\\
	& \|(A_\eb+1)^{-1} g - \Cal T_\eb (A_0+1)^{-1} g \|_{\Cal H^1} \le C \sqrt{\eb} \| g \|, \quad \forall g \in L^2(\Omega).
	\end{align*}
\end{rem}
\begin{rem}
		\label{rem.rationalperiod}
		In the case of periodic boundary conditions \eqref{Peqn}, where $Q = [0,1)^3$ and $\Omega = [0,\ell)^3$,  if $\tfrac{\ell}{\eb} \in \N$ then for $w \in \Cal D(A_0 +1) $ the   corrector $\Cal T_\eb w$ belongs to $\Cal H^1$.
		In this setting it is well-known that one can improve the bound in \eqref{ell.err.H1.NP} from $\sqrt{\eb}$ to $\eb$. Consequently, as discussed in Remark \ref{rem.optimalbdy}, for this case we can replace $\sqrt{\eb}$ with $\eb$ in the relevant  results below. 
\end{rem}
Let us also define the energy space of order $-1$: \[
\E^{-1}:=L^2(\Omega)\times \Cal (\Cal H^1)^*, \]
where $(\Cal H^1)^*$ stands for the dual space of $\Cal H^1$.

We now draw the reader's  attention to the fact that the  key theorems (Theorems \ref{th.|ue-uh|E-1} and \ref{th.Dterror})  on the distance between trajectories in $\E^{-1}$ are in terms of resolvents of the operator $A_\eb$, $\eb\geq 0$. The key point to note is that the proofs of these results essentially rely on the fact $A_\eb$ is self-adjoint and (uniformly in $\eb$) bounded and positive. Since the operator $A_\eb+1$, for  Neumann \eqref{Neqn} or periodic \eqref{Peqn} boundary conditions, also possesses these properties one can see that analogues of Theorems \ref{th.|ue-uh|E-1}-\ref{th.Dterror} readily hold 
(after  appropriately changing  the projector $\Pi_\eb$). Namely, upon defining $\Pi_\eb:\Cal E^2_\eb\to\Cal E^2_0$, for $\Cal E^2_\eb$ given by \eqref{Dep.NP}, as follows
\begin{equation}
\label{xih.ch.NP}
\Pi_\eb (\xi^1, \xi^2): = (\xi^1_0, \xi^2_0), \quad \text{ where } \quad \begin{cases}\text{the term $\xi^i_0 \in \Cal D(A_0 +1)$, $i=1,2$, satisfies} \\
	(A_0+1)\xi^i_{0}=(A_\eb+1)\xi^i,
\end{cases}
\end{equation}
we have the following result.
%
%
\begin{theorem}
	\label{th.|ue-uh|.NP}
	Let $\Cal E^2_\eb$ be given by \eqref{Dep.NP} and $S_\eb(t)$ be the solution operator to the problem \eqref{eq.dw.NP} with Neumann \eqref{Neqn} or periodic \eqref{Peqn} boundary conditions. Then, for all $ \xi \in \Cal E^2_\eb,\, \| \xi \|_{\Cal E^2_\eb} \le R$,  $R>0$, the inequalities 
\[
\left. \begin{aligned}
	&\|S_{\eb}(t)\xi-S_0(t)\xi\|_{\E^{-1}}+\|S_{\eb}(t)\xi-S_0(t)\Pi_\eb\xi\|_{\E^{-1}}\leq M e^{K t} \| (A_\eb +1)^{-1} - (A_0 +1)^{-1} \|_{\Cal L(L^2(\Omega))},\\[2pt]
		&\|\partial_t S_{\eb}(t)\xi-\partial_t S_0(t)\Pi_\eb\xi\|_{\E^{-1}}\leq Me^{Kt}\| (A_\eb +1)^{-1} - (A_0 +1)^{-1} \|_{\Cal L(L^2(\Omega))}^{1/2},
	\\
	&\|S_{\eb}(t)\xi-\textsc{T}_\eb S_0(t)\Pi_\eb\xi\|_{\E}\leq Me^{Kt}\sqrt{\eb},
	\end{aligned} \right. \quad t\geq 0,
\]
	hold for some non-decreasing functions $M=M(R,\|g\|)$ and $K=K(R,\|g\|)$ which are independent of $\eb>0$.	
	\end{theorem}
   Based on Theorem \ref{th.|ue-uh|.NP} and arguing along the same lines as in Sections \ref{s.Er} - \ref{s.ea} we obtain the following theorem on the comparison of distances between anisotropic and homogenised attractors in terms of $\eb$.
\begin{theorem}
	\label{th.Attr2.NP}
Assume \eqref{mainassumptions} and \eqref{d(StB,Ah)<e-t}.	Let $\Cal A^\eb$, $\Cal M^\eb$, $\eb\geq 0$ be attractors corresponding to problem \eqref{eq.dw.NP}, with Neuman \eqref{Neqn} or periodic \eqref{Peqn} boundary conditions, provided by Theorem \ref{th.Attr1.NP}. Let also $\alpha>0$ be such an exponent that $(A_\eb+1)^{-1}\in \Cal L\big(L^2(\Omega),C^\alpha(\overline{\Omega})\big)$ and $0\leq \beta<\alpha$. Then, the following estimates
	\begin{align*}
		& \dist_{\E^{-1}}(\Cal A^\eb,\Cal A^0)\leq M\eb^\varkappa,\ & &\dist_{\E}(\Cal A^\eb,\textsc{T}_\eb\Cal A^0)\leq M\sqrt{\eb}^\varkappa,\ & &\dist_{(C^\beta(\overline{\Omega}))^2}(\Cal A^\eb,\Cal A^0)\leq M\eb^{\theta\varkappa},\quad\\
		& \dist^s_{\E^{-1}}(\Cal M^\eb,\Cal M^0)\leq M\eb^\varkappa,\ & &\dist^s_{\E}(\Cal M^\eb,\textsc{T}_\eb\Cal M^0)\leq M\sqrt{\eb}^\varkappa,\ & &\dist^s_{(C^\beta(\overline{\Omega}))^2}(\Cal M^\eb,\Cal M^0)\leq M\eb^{\theta\varkappa},  
	\end{align*}
	hold for some non-decreasing $M=M(\|g\|)$ and  constants $\varkappa\in(0,1)$, $\theta=\tfrac{\alpha-\beta}{2+\alpha}$ independent of $\eb$. Here $\textsc{T}_\eb$ is the `correction' operator defined by \eqref{Tepsets}.	
\end{theorem}
\appendix
\section{Proof of Theorem \ref{cor.smooth1}}
\label{s.sma}
To prove Theorem \ref{cor.smooth1} we perform a splitting of the solution $u=v+w$ to the problem \eqref{eq.generic} into asymptotically contractive and compact parts. This form of splitting was intoduced in \cite{PataZel}. 

Let us consider 
\begin{equation}
\label{eq.v}
\begin{cases}
\Dt^2 v+\gamma\Dt v-\di(a \Nx v)+Lv+f(u)-f(w)=0,\ x\in\Omega,\ t\geq 0,\\
\xi_{v}|_{t=0}=\xi_{u}(0),\ v|_{\d\Omega}=0,\\ 
\end{cases}
\end{equation}
and
\begin{equation}
\label{eq.w}
\begin{cases}
\Dt^2 w+\gamma\Dt w-\di (a\Nx w)+Lw+f(w)=Lu+g,\ x\in\Omega,\ t\geq 0,\\
\xi_{w}|_{t=0}=0,\ w|_{\d\Omega}=0,\\
\end{cases}
\end{equation} 
where the fixed constant $L>0$ is specified below. 

Recall that $\B$ denotes a positive invariant absorbing set of the semigroup $(\E,S(t))$ (see \eqref{StB C B}). Similar to Theorem \ref{th.dis} we have the following result.
\begin{lem}
	\label{le.|w|E<C}
	Assume \eqref{mainassumptions}, $\xi_{u}(0)\in\B$, $L>0$ be an arbitrary constant and $w$ solve the equation \eqref{eq.w}. Then the  estimate 
	\begin{equation*}
	\|\xi_{w}(t)\|_\E\leq M_L(\|\B\|_\E), \quad t\geq 0,
	\end{equation*}
	holds for some non-decreasing function $M_L$ that depends only on $\nu$ and $L$.
\end{lem} 
The proof of Lemma \ref{le.|w|E<C} follows from the multiplication of the first equation in  \eqref{eq.w} by $\Dt w+\kappa w$ with sufficiently small $\kappa>0$ and the fact that we already know that $\|\xi_{u}(t)\|_\E\leq M(\|\B\|_\E)$ for all $t\geq 0$ (due to the dissipative estimate  \eqref{est.dis}).
\begin{lem}
	\label{le.|w'|to0}
	Assume \eqref{mainassumptions}, $\xi_{u}(0)\in\B$, $L>0$ be an arbitrary  constant and $w$ solve \eqref{eq.w}. Then, for every $\mu>0$ the estimate 
	\begin{equation}
	\label{est.|w'|to0}
	\int_s^t\|\Dt w(\tau)\|^2\,d\tau\leq \mu(t-s)+\frac{M_L(\|\B\|_\E)}{\mu},\quad t\geq s\geq 0,
	\end{equation}
	holds for some non-decreasing function $M_L$ that depends only on $\nu$ and $L$.  
\end{lem}
\begin{proof}
	Multiplying the equation \eqref{eq.w} by $\Dt w$, integrating in $\Omega$ and using Lemma \ref{le.|w|E<C} we obtain
	
	\begin{equation}
	\label{E1}
	\frac{d}{dt}\L +\gamma\|\Dt w\|^2=-L(\Dt u,w)\leq \gamma \mu+\frac{\gamma^{-1}L^2M_L(\|\B\|_\E)}{\mu}\|\Dt u\|^2,
	\end{equation}
	where 
	\begin{equation*}
	\L =\frac{1}{2}\Big(\|\Dt w\|^2+(a\Nx w,\Nx w)+L\| w\|^2\Big)+(F( w),1)-L( u, w)-(g, w).
	\end{equation*}
	From the dissipative estimate \eqref{est.dis} and positive invariance \eqref{StB C B} we see that 
	\begin{equation}
	\label{int|dtu|}
	\int_s^t\|\Dt  u(\tau)\|^2\,d\tau\leq M(\|\B\|_\E),\quad t\geq s\geq 0.
	\end{equation}
	Integrating \eqref{E1} in time from $s$ to $t$, using Lemma \ref{le.|w|E<C} and \eqref{int|dtu|} we derive the desired inequality \eqref{est.|w'|to0} for some new function $M_L$.
\end{proof}

Before continuing, let us recall the following modified Gronwall's lemma.
\begin{lem}[Modified Gronwall's Lemma \cite{PataZel}]\label{le.Gr}
	Let $\L:\R^+\to\R^+$ be an absolutely continuous function satisfying
	\begin{equation*}
	\frac{d}{dt}\L(t)+2\mu \L(t)\leq h(t)\L(t)+k,
	\end{equation*}
	where $\mu>0$, $k\geq 0$ and $\int_s^th(\tau)\,d\tau\leq \mu(t-s)+m$, for all $t\geq s\geq 0$ and some $m\geq 0$. Then
	\begin{equation*}
	\L(t)\leq \L(0)e^{m}e^{-\mu t}+ \frac{ke^{m }}{\mu},\quad t\geq 0.
	\end{equation*} 
\end{lem}

We are now ready to show that $v$ exponentially goes to $0$ in the energy space $\E$.
\begin{prop}\label{pr.vto0} Assume \eqref{mainassumptions} and $\xi_{u}(0)\in\B$. Then, for sufficiently large constant $L=L(\gamma,\nu,f)$, the estimate
	\begin{equation*}
	\|\xi_{v}(t)\|_\E\leq M_L(\|\B\|_\E)e^{-\beta t},\quad t\geq 0,
	\end{equation*}
	holds for some non-decreasing function $M_L$ and constant $\beta>0$ that depend only on $\nu$ and $L$.
\end{prop}
\begin{proof} Fix  $\kappa>0$ to be specified below.
	Multiplying equation \eqref{eq.v} by $\Dt v+\kappa v$ in $L^2(\Omega)$ we find (after some algebraic manipulation) that
	\begin{equation}
	\label{E2.1}
	\begin{aligned}
	&\frac{d}{dt}\L+(\gamma-\kappa)\|\Dt v\|^2+\kappa\big((a \Nx v,\Nx v)+L\| v\|^2+(f(u)-f(w),v)\big)=\\
	&\hspace{.5\textwidth}(f'(u)-f'(w),\Dt w v)-\frac{1}{2}(f''(u)\Dt u,|v|^2),
	\end{aligned}
	\end{equation}
	for
	\begin{equation}
	\label{defLam}
	\begin{aligned}
	&\L: =\frac{1}{2}\Big(\|\Dt v\|^2+(a \Nx v,\Nx v)+L\| v\|^2\Big)+ \kappa(\Dt v,v)+\frac{\kappa\gamma}{2}\| v\|^2+ \\ 
	& \hspace{.55\textwidth} (f(u)-f(w),v)-\frac{1}{2}(f'(u),|v|^2).
	\end{aligned}
	\end{equation}
	Now by the lower bound on $f'$ (see \eqref{mainassumptions})  we compute
	\[
L\| v\|^2 +	(f(u)-f(w),v) = L\| v\|^2 + \big(\int_0^1f'(\lambda u+(1-\lambda) w)d\lambda,|v|^2\big) \ge (L-K_2) \| v\|^2.
	\]
 Thus, for $L >  K_2$, \eqref{E2.1} implies
		\begin{equation}
	\label{E2.1.1}
	\begin{aligned}
	&\frac{d}{dt}\L+(\gamma-\kappa)\|\Dt v\|^2+\kappa(a \Nx v,\Nx v) \le (f'(u)-f'(w),\Dt w v)-\frac{1}{2}(f''(u)\Dt u,|v|^2).
	\end{aligned}
	\end{equation}

	We shall establish below, for sufficiently large $L$,  the equivalence 
	\begin{equation}
	\label{t1p}
	C_\nu \L \le \tfrac{1}{2}\|\Dt v\|^2+\tfrac{1}{2}(a \nabla v,\nabla v) \le 2 \L.
	\end{equation} as well as the inequalities
	\begin{align}
	\label{t2}
&	(f'(u)-f'(w),\Dt w v) \leq  \tfrac{\kappa}{4} (a \Nx v,\Nx v)+M(L,\|\B\|_\E)\|\Dt w\|^2 \L, \\
		\label{t3}
&	-\frac{1}{2}(f''(u),\Dt u|v|^2)\leq \tfrac{\kappa}{4}(a \Nx v,\Nx v)+M(\|\B\|_\E)\|\Dt u\|^2\L,
	\end{align}
	Consequently,
for $0 <\kappa<\gamma /2$,  inequalities \eqref{E2.1.1}-\eqref{t3} imply
\[
\frac{d}{dt}\L +C_\nu \kappa \L  \le h\L, \qquad \text{ for }\ h(t) = M(\|\B\|_\E) (\|\Dt w(t)\|^2  + \|\Dt u(t)\|^2).
\]
This inequality,  Lemma \ref{le.|w'|to0}  and \eqref{int|dtu|} show that  the assumptions of the  Modified Gronwall's Lemma (Lemma \ref{le.Gr}) hold with  $2\mu = C_\nu \kappa$ and $k=0$. Whence
\[
\L(t)\leq M(\Vert \mathcal{B}\Vert) \L(0)e^{-\tfrac{1}{2}C_\nu \kappa  t},\quad t\geq 0.
\] 
From  \eqref{t1p}, and the fact $\xi_v(0) = \xi_u(0)$, we prove the desired result. Therefore, to complete the proof it remains to establish \eqref{t1p}-\eqref{t3}.

Let us prove  \eqref{t1p}. We shall prove the upper bound, as the argument for the lower bound is similar. For $\kappa\in (0,\gamma/2)$, utilising the dissipative estimate in for $u$ (\eqref{est.dis}) and the bounds on $f'$ (see \eqref{mainassumptions} and Remark \ref{rem.nonlinearity}.\ref{fcubic}) we compute
	\begin{flalign*}
	\L &  \geq\tfrac{1}{4}\|\Dt v\|^2+ \tfrac{1}{2}\big((a \nabla v,\nabla v)+L\|v\|^2\big)+\kappa(\tfrac{\gamma}{2}-\kappa)\| v\|^2 +(\int_0^1f'(\lambda u+(1-\lambda) w)d\lambda,|v|^2)-\tfrac{1}{2}(f'(u),|v|^2)\\
	& \geq \tfrac{1}{4}\|\Dt v\|^2+\tfrac{1}{2}\big((a \nabla v,\nabla v)+L\|v\|^2\big)- K_2\|v\|^2 - \tfrac{K_4}{2} ( 1 + |u|^2 , |v|^2)\\
	& \geq\tfrac{1}{4}\|\Dt v\|^2+\tfrac{1}{2}\big((a\nabla v,\nabla v)+L\|v\|^2\big)-\big(K_2+\tfrac{K_4}{2}\big)\|v\|^2-\tfrac{K_4}{2}\| u\|^2_{L^4(\Omega)} \| v\|^{1/2} \| v\|^{3/2}_{L^6(\Omega)} \\ 
	& \ge \tfrac{1}{4}\|\Dt v\|^2+\tfrac{1}{4}(a\nabla v,\nabla v)+ \big( \tfrac{L}{2} - K_2+\tfrac{K_4}{2}-M(\| \B\|_\E) \big)\|v\|^2.
	\end{flalign*}
	Then for large enough $L$, we deduce $
	\L \ge \tfrac{1}{4}\|\Dt v\|^2+\tfrac{1}{4}(a\nabla v,\nabla v)$
	and the upper bound in  \eqref{t1p} holds.

	To prove \eqref{t2} and \eqref{t3}, we use dissipative bounds on $u$ and $w$(Lemma \ref{le.|w|E<C}) plus the growth assumption on $f''$  to establish
	\begin{flalign*}
	(f'(u)-f'(w),\Dt w v)&\leq K_5(1+|u|+|w|,|\Dt w|| v|^2)\leq K_5\| 1+|u|+|w|\|_{L^6(\Omega)}  \| \Dt w \| \, \| | v|^2 \|_L^3(\Omega) \\ 
	& \le  M(L,\|\B\|_\E)\|\Dt w\|\| \Nx v\|^2  \leq \frac{\kappa}{4}\| \Nx v \|^2+M(L,\|\B\|_\E)\|\Dt w\|^2\| \Nx v \|^2,
	\end{flalign*}
	and 
	\[
	-\frac{1}{2}(f''(u),\Dt u| v|^2)\leq M(\|\B\|_\E)\|\Dt u\|\| v\|^2_{L^6(\Omega)}\leq \frac{\kappa\nu}{4}\| \nabla v\|^2+M(\|\B\|_\E)\|\Dt u\|^2\| \nabla v\|^2.
	\]
	Then the desired inequalities follow by invoking the ellipticity of $a$ and the now established \eqref{t1p}. The proof is complete.
\end{proof}
To complete the proof of Theorem \ref{cor.smooth1} it remains to prove that $\xi_w$ is a bounded trajectory in $\E^1$, this is the subject of the next result. 
\begin{prop}\label{pr.w'sm}
	Assume \eqref{mainassumptions} and $\xi_{u}(0)\in\B$. Then, for sufficiently large constant $L=L(\gamma,\nu,f)$, the inequality
	\begin{equation*}
\| \di(a\Nx w)(t) \| +	\|\Nx {\Dt w}(t)\| + \| \Dt^2 w(t) \| \leq M_L(\|\B\|_\E),\quad  t\geq 0,
	\end{equation*}
	holds for some non-decreasing function $M_L$ that depends only on $\nu$ and $L$.
\end{prop} 
\begin{proof}
	Let us set $q:=\Dt w$, then $q$ solves
	\begin{equation*}
	\begin{cases}
	\Dt^2 q+\gamma\Dt q-\di (a \Nx q)+L q+f'(w) q=L\Dt u,\quad x\in\Omega,\ t\geq0,\\
	\xi_{q}|_{t=0}=(0,L u(0)+g),\qquad	q|_{\partial \Omega}=0.
	\end{cases}
	\end{equation*}
	Multiplying the first equation above  by $\Dt q+\kappa q$ and integrating in $\Omega$ we find
	\begin{equation}
	\label{E3.1}
	\begin{aligned}
	&\frac{d}{dt}\L+(\gamma-2\kappa)\|\Dt q\|^2+\kappa\Big(\|\Dt q\|^2+(a \Nx q,\Nx q)+L\| q\|^2+(f'(w),|q|^2)\Big)=\\
	& \hspace{.4\textwidth}L(\Dt u,\Dt q)+\kappa L(\Dt u,\Dt w)+\frac{1}{2}(f''(w)\Dt w,|q|^2),
	\end{aligned}
	\end{equation}
	for
	\begin{equation*}
	\L:=\frac{1}{2}\Big(\|\Dt q\|^2+(a \Nx q,\Nx q)+L\|q \|^2+(f'(w),|q|^2)\Big)+\kappa(\Dt q, q)+\frac{\kappa\gamma}{2}\|q\|^2.
	\end{equation*}
	The identity \eqref{E3.1} can be rewritten in the form
	\begin{equation}
	\label{E3.2}
	\begin{aligned}
	&\frac{d}{dt}\L+(\gamma-2\kappa)\|\Dt q\|^2+2\kappa\L=2\kappa^2(\Dt w,\Dt q)+\kappa^2\gamma\|\Dt w\|^2+ \\
	&\hspace{.25\textwidth}+L(\Dt u,\Dt q)+\kappa L(\Dt u,\Dt w)+\frac{1}{2}(f''(w),\Dt w|q|^2)=:H.
	\end{aligned}
	\end{equation}
	Arguing in a similar manner as in the proof of \eqref{t1p} we have
	\begin{equation}
	\label{<Lq<}
C_\nu	\|\xi_{ q}\|^2_{\E}\leq\L
	\end{equation}
	for some $C_\nu$, 
	as long as  $L=L(\gamma,\nu,f)$ is large enough.
	Using the growth condition of $f''$ (see \eqref{mainassumptions}), the dissipative estimate for $u$ (\eqref{est.dis}),  energy estimate for $w$ (Lemma \ref{le.|w|E<C}) and arguing as in the proof of  \eqref{t2}, the right-hand side $H(t)$ can be estimated as follows:
	\begin{equation}
	\label{t4}
	H\leq M_L(\|\B\|_\E)+\delta\|\xi_{q}\|^2_\E+\frac{M_L(\|\B\|_\E)}{\delta}\big(\|\Dt u\|^2+\|\Dt w\|^2\Big)\|\xi_{q}\|^2_\E ,
	\end{equation}
	for any $\delta>0$. Choosing $0<\kappa < \tfrac{\gamma}{2}$,  $\delta$ small, and collecting \eqref{E3.2}, \eqref{<Lq<}, \eqref{t4} we derive
	\begin{equation*}
	\frac{d}{dt}\L+\kappa\L\leq M_L(\|\B\|_\E)+M_L(\|\B\|_\E)\big(\|\Dt u\|^2+\|\Dt w\|^2\Big)\L.
	\end{equation*}
	Consequently, using \eqref{int|dtu|}, Lemma \ref{le.|w'|to0} and applying the modified Gronwall's lemma we determine that
\begin{equation}\label{pa2old}
\|\Nx {\Dt w}(t)\| + \| \Dt^2 w(t) \| \leq M_L(\|\B\|_\E),\quad  t\geq 0.
\end{equation}
It now readily follows that 
\[
\|\di (a \Nx w)\|\leq M_L(\|\B\|_\E),\quad  t\geq 0.
\]
Indeed, by rewriting equation \eqref{eq.w} in the form
\begin{equation*}
-\di(a \Nx w)= -\Dt^2 w-\gamma\Dt w-L w-f(w)+Lu+g =:H,\quad x\in\Omega,\ t\geq 0,
\end{equation*}
then due to Theorem \ref{th.dis},  Lemma \ref{le.|w|E<C} and \eqref{pa2old} we see that $\|H(t)\|\leq M_L(\|\B\|_{\E})$. Hence, the proof is complete.
\end{proof}

\section{Proof of Theorem \ref{th.expAttr.E-1}}\label{app.E-1}
The proof of Theorem \ref{th.expAttr.E-1} is an adaptation of a construction for exponential attractors presented in \cite[Theorem 2.10]{EMZ2005}. The difference here is one needs to keep track on the parameter dependence of all the sets used in the construction and incorporate the fact we compare the symmetric distance in a topology different to that in which the exponential attractors are constructed. For the reader's convenience we shall provide the details here.

\subsection{Construction of the exponential attractors.
}
\par
Let us introduce notations for the `starting' cover $\Cal U_\eb(\tfrac{1}{K},R)$ and the `model' cover $\Cal U_\eb(\tfrac{1}{4K},1)$:
\[
\Cal V_0(\eb):=\Cal U_\eb(\tfrac{1}{K},R),\quad 
\Cal U(\eb):=\Cal U_\eb(\tfrac{1}{4K},1)=\{\xi_{i\eb}\}^N_{i=1},\quad \eb\geq 0,
\]
where $N_0: =\mathrm{card}\Cal V_0(\eb)= N(\tfrac{1}{K},R) $ and $N:=N(\tfrac{1}{4K},1)$ are, by assumption, independent of $\eb\geq 0$.

We shall begin with constructing a family of sets  $\Cal V_k(\eb)$, $k\in \N$, that satisfy
\begin{equation}\label{ap.1}
\Cal V_k(\eb)\subset \Cal O(B_\eb),\ \footnote{Here $S_\eb(k)$ denotes the $k^{th}$ iteration of $S_\eb$.}S_\eb(k)B_\eb \subset \bigcup_{\xi\in\Cal V_k(\eb)}B_{\E}\big(\xi,\tfrac{1}{K}\left(\tfrac{3}{4}\right)^k\big),\quad \ k \in \N, \ \eb\geq 0.
\end{equation}
Note that,  by the assumptions of Theorem \ref{th.expAttr.E-1}, the above property holds for  $k=0$. We now assume that the set $\Cal V_k(\eb)$ exists, for some fixed $k$, and are going to construct from it the set $\Cal V_{k+1}(\eb)$.
From \eqref{ap.1} it follows that
\[
S_\eb(k+1)B_\eb \subset \bigcup\limits_{\xi\in\Cal V_k(\eb)}S_\eb B_{\E}(\xi,\tfrac{1}{K}\left(\tfrac{3}{4}\right)^k),\ \eb\geq 0.
\]
Let us consider an element $S_\eb \zeta\in S_\eb B_{\E}(\xi,\tfrac{1}{K}\left(\tfrac{3}{4}\right)^k)$ for some $\xi\in\Cal V_k(\eb)$. Due to the splitting \eqref{th.expAttr:spl} we have
\[
S_\eb\zeta-S_\eb\xi=v_\eb+w_\eb,\quad \|v_\eb\|_\E\leq \tfrac{1}{2K}\left(\tfrac{3}{4}\right)^k,\quad \|w_\eb\|_{\E^1_\eb}\leq \left(\tfrac{3}{4}\right)^k,\quad \eb\geq 0.
\]
Therefore, by using the model cover $\Cal U(\eb)$ of $B_{\E^1_\eb}(0,1)$, we see that
\[
w_\eb\in B_{\E^1_\eb}\big(0,\left(\tfrac{3}{4}\right)^k\big)\subset\bigcup_{i=1}^N B_{\E}\left(\left(\tfrac{3}{4}\right)^k\xi_{i\eb},\tfrac{1}{4K}\left(\tfrac{3}{4}\right)^k\right).
\]
Since $S_\eb \zeta=S_\eb\xi+v_\eb+w_\eb$ we deduce that 
\[
S_\eb(k+1)B_\eb\subset \bigcup_{\xi\in \Cal V_k(\eb)}\bigcup_{i=1}^N B_{\E}\left(S_\eb\xi+\left(\tfrac{3}{4}\right)^k\xi_{i\eb},\tfrac{1}{K}\left(\tfrac{3}{4}\right)^{k+1}\right),\quad \eb\geq 0.
\]
As $\|\xi_{i,\eb}\|_{\E^1_\eb}\le \delta_1$ we conclude that \eqref{ap.1} holds for
\begin{equation}
\label{ap.3}
\Cal V_{k+1}(\eb):=S_\eb \Cal V_k(\eb)+\left(\tfrac{3}{4}\right)^k\Cal U(\eb)\subset \Cal O(B_\eb),\quad k\in \mathbb{Z}_+,\ \eb\geq 0.
\end{equation}

Now, it is straightforward to verify the following properties of $\Cal V_k(\eb)$:
\begin{equation}
\label{ap.4}
\begin{cases}
\mathrm{card}\Cal V_k(\eb)=N_0N^{k},\\
\dist_{\E}(S_\eb(k)B_\eb,\Cal V_k(\eb))\le \tfrac{1}{K}\left(\tfrac{3}{4}\right)^k,\\
\dist^s_\E(\Cal V_{k+1}(\eb),S_\eb \Cal V_k(\eb))\le c_0\, \delta_1\left(\tfrac{3}{4}\right)^k, 
\end{cases}\quad
k\in\N,\ \eb\geq 0.
\end{equation}
Based on the sets $\Cal V_k(\eb)$ we construct the sets $E_k(\eb)\subset \Cal O(B_\eb)$:
\begin{equation}\label{ap.Eke}
E_1(\eb):=\Cal V_1(\eb),\quad E_{k+1}(\eb):=\Cal V_{k+1}(\eb)\cup S_\eb E_k(\eb),\quad k\in \N,\ \eb\geq 0,
\end{equation}
that clearly satisfy
\begin{equation}
\label{ap.4'}
\begin{cases}
\mathrm{card}E_k(\eb)\leq kN_0N^{k},\\
S_\eb E_k(\eb)\subset E_{k+1}(\eb),\\
\dist_{\E}(S_\eb(k)B_\eb,E_k(\eb))< \tfrac{1}{K}\left(\tfrac{3}{4}\right)^k,
\end{cases}\quad
k\in \N,\ \eb\geq 0.
\end{equation}

We shall now demonstrate that the sets
\begin{equation}
\label{ap.5}
\Cal M^\eb:=\left[\hat{\Cal M}^\eb\right]_\E,\qquad \hat{\Cal M}^\eb:=\bigcup_{k=1}^\infty E_k(\eb),\ \eb\geq 0,
\end{equation}
are exponential attractors for the discrete dynamical systems $(B_\eb, S_\eb)$. To this end we use the following  result.
\begin{lem}
	\label{le.d(Ek,SnB)}
	Let the assumptions of Theorem \ref{th.expAttr.E-1} hold and the sets $E_k(\eb)$, $k\in\N$, $\eb\geq 0$, be given by \eqref{ap.Eke}. Then, there exist constants $M_1=M_1(c_0,K,\delta_1)>0$ and $\omega=\omega(c_0,K,\delta_1)\in(0,1)$ (both independent of $\eb$) such that for all $\eb\geq 0$ we have
	\[\dist_\E(E_k(\eb),S_\eb(n)B_\eb)\leq M_1\left(\tfrac{3}{4}\right)^{\omega k},\qquad \mbox{for all }n\in\N,\ k\in\N:\ k\geq\tfrac{n}{\omega}.\]
\end{lem}
The proof of this lemma, basically, repeats the proof of Lemma 2.3 from \cite{EMZ2005}, so we omit the proof. 

Now, we are ready to verify that the constructed sets $\Cal M^\eb$ satisfy Definition \ref{de.eA}. The positive invariance and the uniform exponential attraction property (with $\sigma=\ln\left(\tfrac{4}{3}\right)$) 
\begin{equation}
\label{ap.5'}
\dist_{\E}(S_\eb(k)B_\eb,\Cal M^\eb)\leq \frac{1}{K}\left(\frac{3}{4}\right)^k,\quad k\in \N,\ \eb\geq 0,
\end{equation}
 follow directly from \eqref{ap.4'}$_2$, \eqref{ap.4'}$_3$ and \eqref{ap.5}. From the construction it also follows that $\Cal M^\eb\subset \Cal O(B_\eb)$ and thus $\Cal M^\eb$ is compact in $\E$ for every $\eb\geq 0$. Let us check that $\dim_f(\Cal M^\eb,\E)\leq D$ uniformly with respect to $\eb\geq 0$. To this end we need to estimate the minimal number $N_r(\Cal M^\eb,\E)$ of open balls with radius $r>0$ in $\E$ needed to cover $\Cal M^\eb$. Note that, since the cover is open,  $N_r(\Cal M^\eb,\E)=N_r(\hat{\Cal M}^\eb,\E)$. We argue that for any $r>0$ there exist $k_r\in\N$ and $n_r\in\N$ (independent of $\eb$) such that
\begin{equation}
\label{ap.6}
\dist_\E\bigg(\bigcup_{k=k_r+1}^\infty E_k(\eb),\Cal V_{n_r}(\eb)\bigg)<r,\quad \eb\geq 0.
\end{equation}
Indeed, let $k_r$ and $n_r$ be parameters, then by the triangle inequality we have
\[\dist_\E\bigg(\bigcup_{k=k_r+1}^\infty E_k(\eb),\Cal V_{n_r}(\eb)\bigg)\leq \dist_\E\bigg(\bigcup_{k=k_r+1}^\infty E_k(\eb),S_\eb(n_r)B_\eb\bigg)+\dist_\E\left(S_\eb(n_r)B_\eb,\Cal V_{n_r}(\eb)\right),\ \eb\geq 0.\] 
Using \eqref{ap.4}$_2$ and taking $n_r\geq\footnote{Here $\lfloor c \rfloor$ denotes the the largest integer which does not exceeds $c\in\R$.}\left\lfloor\frac{1}{\ln(4/3)}\ln\left(\frac{2}{rK}\right)\right\rfloor\vee 0+1$ we obtain
\[\dist_\E\left(S_\eb(n_r)B_\eb,\Cal V_{n_r}(\eb)\right)<\tfrac{r}{2},\quad \eb\geq 0.\]
Also applying Lemma \ref{le.d(Ek,SnB)} for any $k_r\in\N$ such that $k_r\geq\tfrac{n_r}{\omega}$, we find that
\[
\dist_\E\bigg(\bigcup_{k=k_r+1}^\infty E_k(\eb),S_\eb(n_r)B_\eb\bigg)\leq M_1\left(\tfrac{3}{4}\right)^{\omega k_r}\leq M_1\left(\tfrac{3}{4}\right)^{n_r}<\tfrac{r}{2},\quad \eb\geq 0,
\]
if $n_r\geq\left\lfloor\tfrac{1}{\ln(4/3)}\ln\left(\frac{2 M_1}{r}\right)\right\rfloor\vee 0+1$. Therefore \eqref{ap.6} is valid for $n_r$ and $k_r$ of the form
\[n_r=\left\lfloor\tfrac{1}{\ln(4/3)}\ln\left(\tfrac{1}{r}\right)\right\rfloor\vee 0\,+ C_1(c_0,K,\delta_1),\quad k_r=\left\lfloor\tfrac{1}{\omega\ln(4/3)}\ln\left(\tfrac{1}{r}\right)\right\rfloor\vee 0\,+ C_2(c_0,K,\omega,\delta_1).
\] 

Using the control on the number of elements for $\Cal V_k(\eb)$ and $E_k(\eb)$, \eqref{ap.5} and \eqref{ap.6} we can estimate $N_r(\hat{\Cal M}^\eb,\E)$ as follows
\[
N_r(\hat{\Cal M}^\eb,\E)\leq \sum_{k=1}^{k_r}\mathrm{card} E_k(\eb)+\mathrm{card} \Cal V_{n_r}(\eb)\leq \sum_{k=1}^{k_r}kN_0N^k+N_0N^{n_r}\leq (k_r^2+1)N_0N^{k_r}.
\]
This estimate readily yields
\begin{equation}
\dim_f(\Cal M^\eb,\E):=\limsup_{r\to +0}\tfrac{\ln N_r(\Cal M^\eb,\E)}{\ln\left(\tfrac{1}{r}\right)}\leq \tfrac{\ln N}{\omega\ln(4/3)}=:D,\quad \eb\geq 0.
\end{equation}
\subsection{Estimate on the symmetric distance
}

Derivation of the estimate on the symmetric distance $\dist^s_{\E^{-1}}(\Cal M^\eb,\Cal M^0)$ relies on the following  result.
\begin{lem}
	\label{le.d(Eke,Ek0)}
	Let the assumptions of Theorem \ref{th.expAttr.E-1} hold and the sets $E_k(\eb)$, $k\in\N$, $\eb\geq 0$, be given by \eqref{ap.Eke}. Then for all $k\in\N$ and $\eb\geq 0$ the following estimate
	\begin{equation}
	\label{est.d(Eke,Ek0)}
\begin{aligned}
&	\dist^s_{\E^{-1}}(E_k(\eb),E_k(0))\leq M L^k\Big( \sup_{\xi \in \Cal O(B_\eb) }\Vert S_\eb \xi  - S_0 \xi \Vert_{\E^{-1}} + \dist^s_{\E^{-1}}\big(\Cal U_\eb(\tfrac{1}{4K},1),\Cal U_0(\tfrac{1}{4K},1)\big)\\
&\hspace{10cm}	+ \dist^s_{\E^{-1}} \big(\Cal U_\eb(\tfrac{1}{K},R), \Cal U_0(\tfrac{1}{K},R)\Big),
\end{aligned}
	\end{equation}
	holds for some constant $M=M(L)$ independent of $\eb$ and $k$.
\end{lem}
\begin{proof}
Fix $\eb\geq 0$. 
	
	\emph{Step 1}. We first establish \eqref{est.d(Eke,Ek0)} for the sets $\Cal V_k(\eb)$, $\Cal V_k(0)$. To this end it is convenient to introduce the notations
	\begin{align*}
	& d_k:=\dist^s_{\E^{-1}}(\Cal V_k(\eb),\Cal V_k(0)),\ k\in\mathbb{Z}_+,\qquad \hat d_0:=\dist^s_{\E^{-1}}(\Cal U(\eb),\Cal U(0));\\
	& s_0:=\sup_{\xi \in \Cal O(B_\eb) }\Vert S_\eb \xi  - S_0 \xi \Vert_{\E^{-1}}.
	\end{align*}
It is sufficient to  establish that the following recurrent chain of inequalities
	\begin{equation}
	\label{ap.7}
	d_{k+1}\leq s_0+\hat d_0+Ld_k,\qquad k\in\mathbb{Z}_+.
	\end{equation}
Indeed, upon iterating these inequalities one finds
	\begin{equation}
	\label{ap.10}
	d_k\leq \tfrac{L^{k+1}-1}{L-1}(s_0+\hat d_0+d_0),\qquad k\in\mathbb{Z}_+.
	\end{equation}

Let us prove \eqref{ap.7}.	Note that, from the construction of $\Cal V_k(\eb)$ \eqref{ap.3}, we readily have the following inequalities
	\begin{equation}
	\label{ap.8}
	\dist^s_{\E^{-1}}(\Cal V_{k+1}(\eb),\Cal V_{k+1}(0))\leq \dist^s_{\E^{-1}}(S_\eb \Cal V_k(\eb),S_0\Cal V_k(0))+\hat d_0,\qquad k\in\mathbb{Z}_+.\end{equation}
	Let us now verify the inequality
	\begin{equation}
	\label{ap.9}
	\dist^s_{\E^{-1}}(S_\eb A,S_0C)\leq s_0+L\dist^s_{\E^{-1}}(A,C),\quad \mbox{ for all } A\subset\Cal O(B_\eb),\ C\subset \Cal O(B_0).
	\end{equation}
	Fixing arbitrary $a\in A$, $c\in C$ and using Lipschitz continuity of $S_0$ in $\E^{-1}$ we obtain
	\begin{multline*}
	\|S_\eb a-S_0 c\|_{\E^{-1}}\leq \|S_\eb a-S_0 a\|_{\E^{-1}}+\|S_0 a-S_0 c\|_{\E^{-1}}
	\\
	\leq\|S_\eb a-S_0 a\|_{\E^{-1}}+L\|a-c\|_{\E^{-1}}\leq s_0+L\|a-c\|_{\E^{-1}}.  
	\end{multline*}
	Consequently  \eqref{ap.9} holds. Hence, upon combining \eqref{ap.8} with \eqref{ap.9}, we deduce  \eqref{ap.7} and \emph{step 1} is complete.

	\emph{Step 2}. We claim that the sets $E_k(\eb)$, $E_k(0)$ satisfy the same inequality as in \eqref{ap.10}, namely 
	\begin{equation}
	\label{ap.11}
	\dist^s_{\E^{-1}}(E_k(\eb),E_k(0))\leq \tfrac{L^{k+1}-1}{L-1} (s_0+d_0+\hat d_0),\ k\in \N. 
	\end{equation}
	Since $E_1(\eb)=\Cal V_1(\eb)$ for all $\eb\geq 0$, the above inequality is true for $k=1$.
	Assume \eqref{ap.11} holds for $k=m$ and let us verify it for $k=m+1$.
	It is straightforward to check that for any $A_1$, $A_2\subset \Cal O(B_\eb)$, $C_1$, $C_2\subset \Cal O(B_0)$ the following inequality
	\begin{equation}
	\label{ap.12}
	\dist^s_{\E^{-1}}(A_1\cup A_2,C_1\cup C_2)\leq \dist^s_{\E^{-1}}(A_1,C_1)\vee \dist^s_{\E^{-1}}(A_2,C_2) 
	\end{equation}  
	holds. Therefore, due to \eqref{ap.Eke}, it is enough to show that
	\begin{equation*}
	\dist^s_{\E^{-1}}(S_\eb E_{m}(\eb),S_0E_{m}(0))\leq (s_0+d_0+\hat d_0)\frac{L^{m+2}-1}{L-1}.
	\end{equation*}
	This inequality is a direct consequence of  \eqref{ap.9} and the induction assumption. Indeed, we compute 
	\begin{flalign*}
	\dist^s_{\E^{-1}}(S_\eb E_{m}(\eb),S_0E_{m}(0))&\leq s_0+L\dist^s_{\E^{-1}}(E_m(\eb),E_m(0))\\
&	\leq(s_0+d_0+\hat d_0)\left(1+L\tfrac{L^{m+1}-1}{L-1}\right)=(s_0+d_0+\hat d_0)\tfrac{L^{m+2}-1}{L-1}, 
	\end{flalign*}
	as required. Hence, inequality \eqref{ap.11} yields the desired result with $M(L)=\frac{L^2}{L-1}$.
\end{proof}
We proceed to the proof of the estimate \eqref{est.dE(Me,Mo).g} on the distance $\dist^s_{\E^{-1}}(\Cal M^\eb,\Cal M^0)$. We fix $\eb\geq 0$ and set
\begin{equation}
	\tilde d:= \sup_{\xi \in \Cal O(B_\eb) }\Vert S_\eb \xi  - S_0 \xi \Vert_{\E^{-1}}
	+ \dist^s_{\E^{-1}}\big(\Cal U_\eb(\tfrac{1}{4K},1),\Cal U_0(\tfrac{1}{4K},1)\big) + \dist^s_{\E^{-1}} \big(\Cal U_\eb(\tfrac{1}{K},R), \Cal U_0(\tfrac{1}{K},R)\big).
\end{equation}
In fact, we will only demonstrate how to obtain the estimate \eqref{est.dE-1(Me,Mo).g} for $\dist_{\E^{-1}}(\Cal M^\eb,\Cal M^0)$ as the other side ($\dist_{\E^{-1}}(\Cal M^0,\Cal M^\eb)$) can be done similarly. Let $k\in\N$ be arbitrary and fix $\xi_\eb\in E_k(\eb)$. Due to the just proved Lemma \ref{le.d(Eke,Ek0)} we have 
\begin{equation}
\label{ap.13}
\dist_{\E^{-1}}(\xi_\eb,\Cal M^0)\leq \dist_{\E^{-1}}(\xi_\eb,E_k(0))\leq ML^k\tilde d,\quad k\in \N,\ \eb\geq 0. 
\end{equation} 

On the other hand, we will show below that
\begin{equation}
\label{toshowexpat}
\dist_{\E^{-1}}(\xi_\eb,\Cal M^0) \le M\big(\tilde d L^{\tfrac{n}{\omega}}+\left(\tfrac{3}{4}\right)^n\big), \quad \mbox{for all }n\in\N,\ k\in\N:\ k\geq\tfrac{n}{\omega}.
\end{equation}
Using \eqref{ap.13} for $k\leq\tfrac{n}{\omega}$ and \eqref{toshowexpat} we deduce that
\begin{equation}
\dist_{\E^{-1}}(\xi_\eb,\Cal M^0)\leq M\left(\tilde d L^{\tfrac{n}{\omega}}+\left(\tfrac{3}{4}\right)^n\right),
\end{equation} 
for some $M=M(c_0,c_{-1},K,L,\delta_1)$ which is independent of $\eb$. Optimizing $n$ in the above inequality, for example taking $n=\left\lfloor\tfrac{\omega}{\omega\ln(4/3)+L}\ln\left(\tfrac{\omega\ln(4/3)}{\tilde d\ln L}\right)\right\rfloor\vee 0$, we conclude the desired estimate \eqref{est.dE-1(Me,Mo).g} with $\varkappa=\tfrac{\omega\ln(4/3)}{\omega\ln(4/3)+\ln(L)}$.

It remains to prove \eqref{toshowexpat}. By the triangle inequality we have
\begin{equation}
\label{ap.14}
\dist_{\E^{\text{-}1}}(\xi_\eb,\Cal M^0)\leq \dist_{\E^{\text{-}1}}(\xi_\eb,S_\eb(n)B_\eb)+\dist_{\E^{\text{-}1}}(S_\eb(n)B_\eb,S_0(n)B_0)+\dist_{\E^{\text{-}1}}(S_0(n) B_0,\Cal M^0).
\end{equation}
Let us estimate each of the terms on the right hand side of \eqref{ap.14} separately. Using Lemma \ref{le.d(Ek,SnB)} and considering $k\geq \frac{n}{\omega}$ we obtain
\begin{equation}
\label{ap.15}
\dist_{\E^{-1}}(\xi_\eb,S_\eb(n)B_\eb)\leq M_1\left(\tfrac{3}{4}\right)^n,\quad \mbox{for all }n\in\N,\ k\in\N:\ k\geq\tfrac{n}{\omega}.
\end{equation}

Iterating \eqref{ap.9} we find
\begin{equation}
\label{ap.16}
\dist_{\E^{-1}}(S_\eb(n)B_\eb,S_0(n) B_\eb)\leq s_0 \tfrac{L^n}{L-1}\leq \tilde d\tfrac{L^n}{L-1}.
\end{equation}

Finally, due to the continuous embedding $\E \subset \E^{-1}$ (assumption (2)) and the exponential attraction property of $\Cal M^0$ \eqref{ap.5'} we see that
\begin{equation}
\label{ap.17}
\dist_{\E^{-1}}(S_0(n)B_0,\Cal M^0)\leq c_{-1}\dist_{\E}(S_0(n)B_0,\Cal M^0)\leq c_{-1}\tfrac{1}{K}\left(\tfrac{3}{4}\right)^n,\ n\in\N.
\end{equation}  
Hence \eqref{toshowexpat} holds and the proof is complete.

\section{Proof of Theorem \ref{th.expAttr.E}}
\label{app:estinE}
Derivation of the estimate on the symmetric distance with correction relies on the following \\ interesting modification of Lemma \ref{le.d(Eke,Ek0)}.

\begin{lem}
	\label{le.d(Eke,TeEk0)}
	Let the assumptions of Theorem \ref{th.expAttr.E} hold and the sets $E_k(\eb)$, $k\in\N$, $\eb\geq 0$, be given by \eqref{ap.Eke}. Then for all $k\in\N$ and $\eb\geq 0$ the following estimate
	\begin{multline}
	\label{est.d(Eke,TeEk0)}
	\dist^s_{\E}\big(E_k(\eb),\textsc{T}_\eb E_k(0)\big)\leq M L^k\Big(  \sup_{\xi\in  \Cal O(B_0) }\Vert S_\eb \Pi_\eb^{-1} \xi  - \textsc{T}_\eb S_0 \xi \Vert_{\E} + \sup_{\xi \in \Cal O (B_0) }\Vert \textsc{T}_\eb  \xi  - \Pi_\eb^{-1} \xi   \Vert_{\E}   \\ + \dist^s_{\E}\big(\Cal U_\eb(\tfrac{1}{4K},1),\textsc{T}_\eb\,\Cal U_0(\tfrac{1}{4K},1)\big) + \dist^s_{\E} \big(\Cal U_\eb(\tfrac{1}{K},R), \textsc{T}_\eb\,\Cal U_0(\tfrac{1}{K},R)\big)\Big),
	\end{multline}
	holds with some constant $M=M(L)$ which is independent of $\eb$ and $k$.
\end{lem}
\begin{proof}
	We follow the strategy of Lemma \ref{le.d(Eke,Ek0)} and fix $\eb\geq 0$.
	
	We first derive an estimate on the distance between $\Cal V_k(\eb)$ and $\textsc{T}_\eb \Cal V_k(0)$. Let us introduce the notations
	\begin{align*}
		d_k:=\dist^s_{\E}\big(\Cal V_k(\eb),\textsc{T}_\eb \Cal V_k(0)\big),\ k\in\mathbb{Z}_+,\qquad \hat d_0:=\dist^s_{\E}\big(\Cal U(\eb),\textsc{T}_\eb\, \Cal U(0)\big);\\
		s_0:=\sup_{\xi\in  \Cal O(B_0) }\|S_\eb  \Pi_\eb^{-1} \xi-\textsc{T}_\eb S_0\xi\|_\E+L\sup_{\xi\in \Cal O(B_0)}\|\Pi_\eb^{-1}  \xi - \textsc{T}_\eb\xi\|_{\E}. 
	\end{align*}
	We are going to verify the recurrent chain of inequalities
	\begin{equation}
	\label{ap.7'}
		d_{k+1}\leq s_0+\hat d_0+Ld_k,\quad k\in\mathbb{Z}_+.
	\end{equation}
	From the construction of $\Cal V_k(\eb)$ \eqref{ap.3} we see
	\begin{equation}
	\label{ap.8'}
	\dist^s_{\E}\big(\Cal V_{k+1}(\eb),\textsc{T}_\eb\Cal V_{k+1}(0)\big)\leq \dist^s_{\E}\big(S_\eb \Cal V_k(\eb),\textsc{T}_\eb S_0\Cal V_k(0)\big)+\hat d_0,\ k\in\mathbb{Z}_+.
	\end{equation}
	We now argue that
	\begin{equation}
	\label{ap.9'}
		\dist^s_{\E}\big(S_\eb A,\textsc{T}_\eb S_0 C\big)\leq s_0+L\dist^s_{\E}\big(A,\textsc{T}_\eb C\big),\qquad \mbox{for all }A\subset \Cal O(B_\eb),\ C\subset\Cal O(B_0).
	\end{equation}
	Indeed, fixing $a\in A$, $c\in C$ and using the uniform  (with respect to $\eb>0$)  Lipschitz continuity of $S_\eb$ in $\E$ (assumption \emph{(5)} of Theorem \ref{th.expAttr.E}) we compute
	\begin{flalign*}
		\|S_\eb a - \textsc{T}_\eb S_0 c\|_\E & \leq \|S_\eb a-S_\eb \Pi_\eb^{-1} c  \|_\E+\| S_\eb \Pi_\eb^{-1} c-\textsc{T}_\eb S_0  c\|_\E\\
	&	\leq L \| a- \Pi_\eb^{-1} c \|_\E+ \| S_\eb \Pi_\eb^{-1} c-\textsc{T}_\eb S_0 c\|_\E \\
	& \le  L \| a- \textsc{T}_\eb  c \|_\E+ L \| \textsc{T}_\eb c -\Pi_\eb^{-1} c \|_\E + \| S_\eb \Pi_\eb^{-1} c-\textsc{T}_\eb S_0 c\|_\E.
	\end{flalign*}
The above inequality, obviously, implies \eqref{ap.9'}. Combining \eqref{ap.8'} and \eqref{ap.9'} we establish the recurrent inequalities \eqref{ap.7'} which yield
	\begin{equation*}
		d_k\leq (s_0+d_0+\hat d_0)\tfrac{L^{k+1}-1}{L-1},\quad  k\in\mathbb{Z}_+.
	\end{equation*}  
To derive the estimate \eqref{est.d(Eke,TeEk0)} on the distance $\dist^s_\E\big(E_k(\eb),\textsc{T}_\eb E_k(0)\big)$ we simply argue as in \emph{Step 2} of Lemma \ref{le.d(Eke,Ek0)}.
\end{proof}

We are ready to prove the theorem. We fix $\eb\geq 0$ and set
\begin{multline*}
	\tilde d := { \sup_{\xi\in \Cal O(B_0) }\|S_\eb \Pi_\eb^{-1}  \xi-\textsc{T}_\eb  S_0\xi\|_\E+\sup_{\xi\in\Cal O(B_0)}\|\textsc{T}_\eb\xi - \Pi_\eb^{-1} \xi\|_\E}\\
	+ \dist^s_{\E}\big(\Cal U_\eb(\tfrac{1}{4K},1),\textsc{T}_\eb\,\Cal U_0(\tfrac{1}{4K},1)\big) + \dist^s_{\E} \big(\Cal U_\eb(\tfrac{1}{K},R), \textsc{T}_\eb\,\Cal U_0(\tfrac{1}{K},R)\big).
\end{multline*} 
As in the proof of Theorem \ref{th.expAttr.E-1} we will only consider $\dist_{\E}(\Cal M^\eb,\textsc{T}_\eb\Cal M^0)$ as the other side can argued in a similar manner. Let $k\in\N$ and $\xi_\eb\in E_k(\eb)$ be fixed. Then according to Lemma \ref{le.d(Eke,TeEk0)} we have
\begin{equation}
	\label{ap.13'}
	\dist_\E\big(\xi_\eb,\textsc{T}_\eb\Cal M^0\big)\leq \dist_{\E}\big(\xi_\eb, \textsc{T}_\eb E_k(0)\big)\leq ML^k\tilde d,\quad k\in\N,\ \eb\geq 0. 
\end{equation}
{On the other hand we deduce below  that
\begin{equation}
\label{ap.13''}
\dist_{\E}(\xi_\eb, \textsc{T}_\eb \Cal M^0)\leq M\big( \tilde d  L^{\tfrac{n}{\omega}}+\left(\tfrac{3}{4}\right)^n\big) + m(\eb), \quad k\geq \tfrac{n}{\omega}, \, n \in \N,
\end{equation}
for $\omega$ given in Lemma \ref{le.d(Ek,SnB)}. The estimate \eqref{ap.13'} for $k\leq \tfrac{n}{\omega}$ together with \eqref{ap.13''} implies
\begin{equation}
\dist_{\E}(\xi_\eb, \textsc{T}_\eb \Cal M^0)\leq M\big( (\tilde d  +  m(\eb))L^{\tfrac{n}{\omega}}+\left(\tfrac{3}{4}\right)^n\big) ,
\end{equation} 
for some $M=M(c_0,K,L,L_{\rm cor},\delta_1)$ which is independent of $\eb$. Optimizing $n$ in the above inequality provides the desired result.}

It remains to prove \eqref{ap.13''}. By the triangle inequality we deduce that
\begin{equation}\label{ap.14'}
\begin{aligned}
&	\dist_\E\big(\xi_\eb,\textsc{T}_\eb\Cal M^0\big)\leq\\ 
& \hspace{2cm}\dist_\E(\xi_\eb,S_\eb(n) B_\eb)+\dist_\E(S_\eb(n)B_\eb,\textsc{T}_\eb S_0(n) \Pi_\eb B_\eb)
+\dist_\E(\textsc{T}_\eb S_0(n) \Pi_\eb B_\eb,\textsc{T}_\eb \Cal M^0).
\end{aligned}
\end{equation}

The first term on the right hand side of \eqref{ap.14'} can be controlled by Lemma \ref{le.d(Ek,SnB)} for $k\geq \tfrac{n}{\omega}$:
\begin{equation}
	\label{ap.15'}
	\dist_\E(\xi_\eb,S_{\eb}(n)B_\eb)\leq M_1\left(\tfrac{3}{4}\right)^n.
\end{equation}
By {the identity $\Pi_\eb B_\eb = B_0$ (assumption \emph{(3)} of Theorem \ref{th.expAttr.E}) }and iterations of \eqref{ap.9'} we  estimate the second term on the right hand side of \eqref{ap.14'}:
\begin{equation}
\label{ap.16'}
{\dist_\E(S_\eb(n)B_\eb,\textsc{T}_\eb S_0(n) \Pi_\eb B_\eb) = \dist_\E(S_\eb(n)B_\eb,\textsc{T}_\eb S_0(n)  B_0)}\leq s_0\tfrac{L^n-1}{L-1}\leq \ \tilde d\tfrac{L}{L-1}L^n.
\end{equation}

The last term on the right hand side of \eqref{ap.14'} can be estimated using { $\Pi_\eb B_\eb = B_0$ and the property of $\textsc{T}_\eb$ (assumption \emph{(4)} of Theorem \ref{th.expAttr.E}) and the exponential attraction property of $\Cal M^0$:
\begin{equation}
	\label{ap.17'}
	\begin{aligned}
	\dist_\E(\textsc{T}_\eb S_0(n) \Pi_\eb B_\eb,\textsc{T}_\eb \Cal M^0) & = 	\dist_\E(\textsc{T}_\eb S_0(n) B_0,\textsc{T}_\eb \Cal M^0) \\
& \leq L_{\rm cor}{\dist_\E (S_0(n)B_0,\Cal M^0)}+m(\eb)
	\leq L_{\rm cor} \tfrac{1}{K}\left(\tfrac{3}{4}\right)^n+m(\eb).
	\end{aligned}
\end{equation}
Hence \eqref{ap.13''} follows from \eqref{ap.14'}-\eqref{ap.17'} and the theorem is proved.}

{
	\section{On the refinement of inequality \eqref{Dttestimate1}}
	\label{app.improvement}
	Let us begin by noting that in Section \ref{s.gaE}  we were actually in the position to prove the following improvement of inequality \eqref{Dttestimate1} (in Theorem \ref{th.Dterror}).
	\begin{prop}
		For every  $\xi \in B_{\Cal E^2_\eb}(0,R)$ the inequality
		\begin{align}
		\label{e.imp1} 
		\Vert \Dt S_\eb(t) \xi - \Dt S_0(t) \Pi_\eb \xi \Vert_{\E^{-1}}\leq Me^{Kt} \Vert A^{-1}_\eb - A^{-1}_0 \Vert_{\Cal L(L^2(\Omega))}^{2/3},\quad t\geq 0,
		\end{align}
		holds for some non-decreasing functions $M=M(R,\|g\|)$ and $K=K(R,\|g\|)$ which are independent of $\eb>0$.
	\end{prop}
	\begin{proof}
		The proof of this result follows along the same lines as in the proof of Theorem \ref{th.Dterror} except for the following minor alterations:
		\begin{enumerate}
			\item{In the uniform bounds \eqref{Eerror.e1} (due to Theorem \ref{th.disD}) we actually have
				\[
				\| \Dt^2 u^\eb \|_{H^1_0(\Omega)}^2 + 		\| \Dt^2 u^0 \|_{H^1_0(\Omega)}^2 \le M.
				\]
			}
			\item{From (1) we can see that $q^\eb = \Dt u^\eb - \Dt u^0$ satisfies the bound
				\[
				\| \Dt q^\eb \| \le \| \Dt q^\eb \|_{H^{-1}(\Omega)}^{1/2} \| \Dt q^\eb \|_{H^1_0(\Omega)}^{1/2} \le M \| \Dt q^\eb \|_{H^{-1}(\Omega)}^{1/2},
				\]
				and so we can improve \eqref{e.toimprove} as follows:
				\begin{equation*}
				\begin{aligned}
				| \big( A_0 \Dt u^\eb-A_\eb \Dt u^\eb,A_{0}^{-1}\Dt \q \big) | & = | \big( A_\eb \Dt u^\eb,(A^{-1}_\eb -A_{0}^{-1})\Dt \q \big) | \le \Vert A_\eb \Dt \u \Vert \Vert A^{-1}_\eb - A^{-1}_0 \Vert_{\Cal L(L^2(\Omega))} \Vert \Dt \q \Vert  \\ & \le M\Vert A^{-1}_\eb - A^{-1}_0 \Vert_{\Cal L(L^2(\Omega))}\| \Dt q^\eb \|_{H^{-1}(\Omega)}^{1/2} \\
				& \le M \big( \tfrac{3}{4} \Vert A^{-1}_\eb - A^{-1}_0 \Vert_{\Cal L(L^2(\Omega))}^{4/3} + \tfrac{1}{4}\| \Dt q^\eb \|_{H^{-1}(\Omega)}^{2}  \big).
				\end{aligned} 
				\end{equation*}
			}
			\item{From (2) we can replace \eqref{dttrgron} with 
				\begin{equation*}
				\frac{d}{dt} \Lambda \le M_1 e^{Kt} \Vert A^{-1}_\eb - A^{-1}_0 \Vert_{\Cal L(L^2(\Omega))}^{4/3} + M_2 \Lambda, \quad \Lambda : = \frac{1}{2} 
				{ \big( \Dt\q, A_{0}^{-1}\Dt\q\big) }+\frac{1}{2} \|\q\|^2,
				\end{equation*}	
				which then leads to the desired result.
			}
		\end{enumerate}
	\end{proof}
	
	In order to further improve \eqref{Dttestimate1} (or rather \eqref{e.imp1}), and achieve the optimal bound with power one, we intend to argue as in the proof of Theorem \ref{th.|ue-uh|E-1}. For this reason, we require additional regularity on the initial data $\xi$. In particular, we shall show that it is sufficient for $\xi \in \Cal E^2_\eb$ to be  such that the solution $u^\eb$ to \eqref{eq.dw} (with initial data $\xi$) satisfies
	\[
	\| A_\eb \Dt^2 u^\eb \| \le M, \quad t\ge 0.
	\]
	Then,  we shall demonstrate that this additional regularity is `natural' in the sense that the global attractor $\Cal A^\eb$ possesses such smoothness  under the additional mild assumption on the non-linearity $f$: 
	\begin{equation}\tag{H3}
	\label{H2}
	f \in C^3(\R),\quad |f'''(s)|\leq K_6,\quad s\in\R.
	\end{equation}

	Let us introduce the mapping
	\[
	A u : = -\di ( a \nabla u),
	\]
	recall
	\begin{equation*}
	\left\{\begin{aligned}
	& \Cal E^2 = \big\{ \xi \in (H^1_0(\Omega))^2 \, | \, \big( A \xi^1 - g \big) \in H^1_0(\Omega) \text{ and } A \xi^2 \in L^2(\Omega)  \big\}, \\
	& \Vert \xi \Vert_{\Cal E^2}^2 =  \Vert A \xi^1 - g \Vert^2_{H^1_0(\Omega)} +\| A \xi^1 \|^2 + \Vert A \xi^2  \Vert^2,
	\end{aligned} \right.
	\end{equation*}
	and introduce
	\begin{equation*}
	\left\{\begin{aligned}
	& \Cal E^3: = \big\{ \xi \in \Cal E^2 \, | \, A\big( A \xi^1 + f(\xi^1) - g \big) \in L^2(\Omega) \text{ and } A \xi^2 \in H^1_0(\Omega)  \big\}, \\
	& \Vert \xi \Vert_{\Cal E^3}^2 : =  \Vert A\big( A \xi^1 + f(\xi^1) - g \big) \Vert^2 +\| \Nx A \xi^2 \|^2 + \Vert \xi  \Vert^2_{\Cal E^2}.
	\end{aligned} \right.
	\end{equation*}
	Our first result is that a dissipative estimate holds in $\Cal E^3$.
	\begin{theorem}\label{th.disE2B}
		Assume \eqref{mainassumptions} and \eqref{H2}. Then for any initial data $\xi\in\Cal E^3$ the energy solution $u$ to problem \eqref{eq.generic} is such that $\xi_u\in L^\infty(\R_+;\Cal E^3)$ and the following dissipative estimate is valid:
		\begin{equation*}
		\|\Dt^4 u(t)\|+\|\Dt^3 u(t) \|_{H^1_0(\Omega)}+ \| A\Dt^2 u(t) \| + \|\xi_u(t)\|_{\Cal E^3}\leq M(\|\xi\|_{\Cal E^3})e^{-\beta t}+M(\|g\|),\quad t\geq 0,
		\end{equation*} 
		for some non-decreasing function $M$ and constant $\beta>0$ that depend only on $\nu>0$. 	
	\end{theorem}
	\begin{proof}
		We begin by noting that since $\xi \in \Cal E^2$ then, by the dissipative estimate in $\Cal E^2$ (Theorem \ref{th.disD}), $\xi_u(t) : = S(t) \xi$ satisfies
		\begin{equation}
		\label{disE2.e1}
		\|\Dt^3 u(t)\|+\|\nabla\Dt^2 u(t) \|+\|\xi_u(t)\|_{\Cal E^2}\leq M(\|\xi\|_{\Cal E^2})e^{-\beta t}+M(\|g\|),\quad t\geq 0.
		\end{equation}
		In particular, we have
		\begin{equation}
		\label{disE2.e2}
		\| u(t)\|_{C^\alpha(\overline{\Omega})} + \| \Dt u(t)\|_{C^\alpha(\overline{\Omega})}\leq M(\|\xi\|_{\Cal E^2})e^{-\beta t}+M(\|g\|),\quad t\geq 0,
		\end{equation}
		where $\alpha$ is given in Remark \ref{rem.holder}.
		
		Now upon differentiating \eqref{eq.generic}, in time, three times we deduce that $r(t) : = \Dt^3 u(t)$ solves the equation
		\[
		\Dt^2 r + \gamma \Dt r + A r = -f'''(u)(\Dt u)^3 - 3f''(u) \Dt u\, \Dt^2 u-f'(u)\Dt^3u =: F(t), \quad t\ge 0,
		\]
		with initial data
		\[
		r(0) = \gamma^2 \xi^2 + \gamma ( A \xi^1 + f(\xi^1) -g) - A \xi^2 - f'(\xi^1)\xi^2,
		\]
		and 
		\[
		\Dt r(0) = - \gamma r(0) + \gamma A \xi^2 +A\big( A\xi^1 + f(\xi^1) - g \big) - f''(\xi^1) \big(\xi^2\big)^2 + f'(\xi^1) \big( \gamma \xi^2 + A\xi^1 + f(\xi^1) - g \big).
		\]
		Now by \eqref{disE2.e1} and \eqref{disE2.e2} we readily deduce that $F\in L^\infty(\R_+;L^2(\Omega))$. Additionally, since $\xi \in \Cal E^3$ we see that $r(0) \in H^1_0(\Omega)$ and $\Dt r(0) \in L^2(\Omega)$, i.e. $\big(r(0),\Dt r(0) \big) \in \E$. Consequently, by standard linear dissipative estimates for $r$, we find
		\begin{equation}
		\label{disE2.e3}
		\|\Dt^4 u(t)\|+\|\Dt^3 u(t) \|_{H^1_0(\Omega)} \leq M(\|\xi\|_{\Cal E^3})e^{-\beta t}+M(\|g\|),\quad t\geq 0,
		\end{equation}
		for some $M$ that depends only on $\nu$.
		
		Now, the remaining claims are proven by differentiating \eqref{eq.generic} once to get
		\[
		\| A \Dt u(t) \|_{H^1_0(\Omega)} \le M(\|\xi\|_{\Cal E^3})e^{-\beta t}+M(\|g\|), \quad t\ge 0. 
		\]
		Then differentiating \eqref{eq.generic} one more time to get
		\[
		\| A \Dt^2 u(t) \| \le M(\|\xi\|_{\Cal E^3})e^{-\beta t}+M(\|g\|), \quad t\ge 0,
		\]
		and finally  re-arranging \eqref{eq.generic} to get
		\[
		\| A\big( Au(t) + f(u(t)) - g  \big) \| \le M(\|\xi\|_{\Cal E^3})e^{-\beta t}+M(\|g\|), \quad t\ge 0.
		\]
		
	\end{proof}
	Equipped with Theorem \ref{th.disE2} we are ready to prove the desired improvement of \eqref{e.imp1}. Namely, upon setting $\Cal E^3_\eb$ to be $\Cal E^3$ for the case $a = a(\tfrac{\cdot}{\eb})$ and $B_{\Cal E^3_\eb}(0,R) : = \{ \xi \in \Cal E^3_\eb \, | \, \| \xi \|_{\Cal E^3_\eb} \le R \}$, the following result holds. 
	\begin{theorem}\label{th.improvedinequality}	Assume \eqref{mainassumptions} and \eqref{H2}.
		Then, for every  $\xi \in B_{\Cal E^3_\eb}(0,R) $, the following inequality
		\begin{align*}
		\Vert \Dt S_\eb(t) \xi - \Dt S_0(t) \Pi_\eb \xi \Vert_{\E^{-1}}\leq Me^{Kt} \Vert A^{-1}_\eb - A^{-1}_0 \Vert_{\Cal L(L^2(\Omega))},\quad t\geq 0,
		\end{align*}
		holds for some non-decreasing functions $M=M(R,\|g\|)$ and $K=K(R,\|g\|)$ which are independent of $\eb>0$.
	\end{theorem}
	\begin{proof}
		The argument is similar to that in Theorem \ref{th.|ue-uh|E-1} so we shall just outline the main ideas.

		Set $\xi_{u^\eb}(t):=S_\eb(t)\xi$, $\xi_{u^0}(t):=S_0(t)\Pi_\eb\xi$ and recall $\xi_0 = \Pi_\eb \xi$. Then by the dissipative estimates for $\xi_{u^\eb}$ in $\Cal E^3_\eb$ (Theorem \ref{th.disE2B}) and  $\xi_{u^0}$ in $\Cal E^2_0$ (Theorem \ref{th.disD}) we have the following uniform bounds in $t$ and $\eb$:
		\begin{equation*}
		\begin{aligned}
		\|  u^\eb \|_{H^1_0(\Omega)} + \| A_\eb  \Dt u^\eb \|  +  \| A_\eb  \Dt^2 u^\eb \|   +     \|  u^0 \|_{H^1_0(\Omega)}  \le M.
		\end{aligned}
		\end{equation*}
		
		The difference $q^\eb : = \Dt u^\eb - \Dt u^0$ solves
		\begin{equation*}
		\begin{cases}
		\Dt^2\q+\gamma\Dt\q + A_0 \q= A_0 \Dt  u^\eb - A_\eb \Dt u^\eb + f'(u^0) \Dt u^0  - f'(\u)\Dt \u,\quad x\in\Omega,\ t\geq0,\\
		\xi_{\q}|_{t=0}= \big(\xi^2 - \xi^2_0 , \gamma ( \xi^2_0 -  \xi^2) + f(\xi^1_{0}) - f(\xi^1) \big),\quad \q|_{\d\Omega}=0,
		\end{cases}
		\end{equation*}
		and we have
		\begin{equation*}
		\Vert \xi_{q^\eb}|_{t=0}  \Vert_{\E^{-1}} \le C \Vert A^{-1}_\eb - A^{-1}_0 \Vert_{\Cal L(L^2(\Omega))}.
		\end{equation*}

		After testing the first equation in the above problem  with $A_{0}^{-1}\Dt\q$ and some algebra (similar to that in Theorem \ref{th.|ue-uh|E-1}) we deduce that    
		\begin{multline*}
		\frac{d}{dt} \Lambda \le 
		-\big( A_\eb \Dt^2 u^\eb, (A^{-1}_\eb-A^{-1}_0) q^\eb \big)+\big(f'(u^0)\Dt u^0 -f'(\u)\Dt \u,A_{0}^{-1}\Dt\q\big), \\
		\text{ where }\Lambda : =\frac{1}{2} \|\q\|^2 + \frac{1}{2} { \big( \Dt\q, A_{0}^{-1}\Dt\q\big) } - \big(A_\eb \Dt u^\eb, (A^{-1}_\eb - A^{-1}_0) q^\eb \big).
		\end{multline*}
		Now in the proof of Theorem \ref{th.Dterror} we showed that
		\[
		\big| \big(f'(u^0)\Dt u^0 -f'(\u)\Dt \u,A_{0}^{-1}\Dt\q\big) \big| \le M_1 \big( e^{Kt} \Vert A^{-1}_\eb - A^{-1}_0 \Vert_{\Cal L(L^2(\Omega))}^2 + \tfrac{1}{2} \| q^\eb \|^2 + \tfrac{1}{2}( \Dt \q , A^{-1}_0 \Dt \q) \big).
		\]
		Therefore
		\begin{multline*}
		\frac{d}{dt} \Lambda \le \big( 2M_1 A_\eb \Dt u^\eb - A_\eb \Dt^2 u^\eb ,  (A^{-1}_\eb - A^{-1}_0) q^\eb \big)  - 2M_1  \big( A_\eb \Dt u^\eb  ,  (A^{-1}_\eb - A^{-1}_0) q^\eb \big)\\
		+ M_1 \big( e^{Kt} \Vert A^{-1}_\eb - A^{-1}_0 \Vert_{\Cal L(L^2(\Omega))}^2 + \tfrac{1}{2} \| q^\eb \|^2 + \tfrac{1}{2}( \Dt \q , A^{-1}_0 \Dt \q) \big),
		\end{multline*}
		and since
		\[
		\big|  \big( 2M_1 A_\eb \Dt u^\eb - A_\eb \Dt^2 u^\eb ,  (A^{-1}_\eb - A^{-1}_0) q^\eb \big)  \big| \le C \| A^{-1}_\eb - A^{-1}_0\|^2_{\Cal L(L^2(\Omega)} + M_1 \tfrac{1}{2} \| q^\eb\|^2, 
		\]we find
		\[
		\frac{d}{dt} \Lambda \le 2M_1 \Lambda +  C e^{Kt} \Vert A^{-1}_\eb - A^{-1}_0 \Vert_{\Cal L(L^2(\Omega))}^2,
		\]
		from which the desired result follows.
	\end{proof}
	
	We finish this section with the following result on the smoothness of the global attractor.
	\begin{theorem}\label{th.ExAtS2}
		Assume \eqref{mainassumptions} and \eqref{H2}, and let $\A$ be the global attractor of  the dynamical system $(\E,S(t))$ given by \eqref{ds.e}. Then
		\begin{equation*}
		\|\A \|_{\Cal E^3 }\leq M(\|g\|),
		\end{equation*}
		for some non-decreasing $M$ that depends only on $\nu$.
	\end{theorem}
	Indeed this result can be proved by arguing as in Section \ref{s.smth} for the following splitting: for initial data $\xi \in B_{\Cal E^2}(0,R_1)$ we consider $H \in H^1_0(\Omega)$ that satisfies
	\[
	-\di(a\Nx H) = -f'(\xi^1) \xi^2 \in L^2(\Omega),
	\]
	and $G \in H^1_0(\Omega)$ that satisfies
	\[
	-\di(a\Nx G) = g - f(\xi^1) - \gamma H \in L^2(\Omega).
	\]
	Then, we decompose the solution $u$ to \eqref{eq.generic} as $u = v + w$ where
	\begin{equation*}
	\begin{cases}
	\Dt^2 v+\gamma\Dt v-\di (a \Nx v)=0,\quad x\in\Omega,\ t\geq 0,\\
	\xi_{v}|_{t=0}=(\xi^1 - G, \xi^2 - H),\quad v|_{\partial\Omega}=0,\\
	\end{cases}
	\end{equation*}
	and
	\begin{equation*}
	\begin{cases}
	\Dt^2 w+\gamma\Dt w-\di(a \Nx w)=-f(u) + g,\quad x\in\Omega,\ t\geq 0,\\
	\xi_{w}|_{t=0}=(G,H),\quad w|_{\d\Omega}=0.\\
	\end{cases}
	\end{equation*}
	
	The main points to highlight are that we can argue as in the proof of Theorem \ref{th.disE2} (to produce an analogue of Lemma \ref{le.we2}) and  establish that 
	\[
	\dist_{\E} \big( S(t) B_{\Cal E^2}(0,R_1) , B_{\Cal E^3}(0,R_2) \big) \le M e^{- \beta t}, \quad t\ge 0,
	\]
	holds for some positive constants $R_2$, $M$ and $\beta$ that depend only on $\nu$. Then, we use the transitivity of exponential attraction (Theorem \ref{thm.trans}) and Corollary \ref{cor.smooth2} to deduce that $B_{\Cal E^3}(0,R_2)$ attracts bounded sets in $\E$:
	\[
	\dist_{\E} \big( S(t) B, B_{\Cal E^3}(0,R_2) \big) \le M( \| B \|_{\E}) e^{-\beta t}, \quad t\ge 0.
	\]
	This finally allows us to argue as in the proof of Theorem \ref{th.ExAttr} to prove Theorem \ref{th.ExAtS2}.
}

Consequently, the improved regularity of the attractor (Theorem \ref{th.ExAtS2}) allows us to apply, when appropriate, the improved inequality (Theorem \ref{th.improvedinequality}) in obtaining error estimates in homogenisation (cf. Remark \ref{rem.optimalbdy}).

\section*{Acknowledgements}
The authors would like to  thank Sergey Zelik for helpful discussions related to this work. S. Cooper was supported by the EPSRC grant EP/M017281/1 (``Operator asymptotics, a new approach to length-scale interactions in metamaterials").

\end{document}